\documentclass[11pt]{imsart}
\RequirePackage[OT1]{fontenc}
\RequirePackage{amsthm,amsmath}
\usepackage{graphicx,amsmath,amsthm,amsfonts,amssymb}
\usepackage{booktabs, rotating, float}
\usepackage[hmargin=3cm,vmargin=3cm]{geometry}
\usepackage{mathtools}
\usepackage{bbm}
\mathtoolsset{showonlyrefs}
\usepackage{mathrsfs}
\usepackage[square,sort,comma,numbers]{natbib}
\usepackage{enumerate}
\usepackage[usenames,dvipsnames,svgnames,table]{xcolor}
\usepackage[pagebackref,letterpaper=true,colorlinks=true,pdfpagemode=none,urlcolor=blue,linkcolor=blue,citecolor=BrickRed,pdfstartview=FitH]{hyperref}
\usepackage{hypernat}
\usepackage{enumerate}
\usepackage{color}
\usepackage{listings}
\usepackage{epstopdf}
\usepackage{makecell}
\usepackage{tikz}
\usetikzlibrary{decorations.pathreplacing}
\usepackage{array}
\usepackage{caption}
\usepackage{subcaption}
\usepackage{soul}
\usepackage{verbatim}
\usepackage{siunitx}
\usetikzlibrary{calc} \tikzset{>=latex}
\usetikzlibrary{calc,decorations.markings}
\usetikzlibrary{positioning,shapes,decorations.text}
\theoremstyle{plain}
\newtheorem{thm}{Theorem}[section]
\newtheorem{cor}[thm]{Corollary}

\newtheorem{prop}[thm]{Proposition}
\newtheorem{lemma}[thm]{Lemma}
\newtheorem{assumption}[thm]{Assumption}
\theoremstyle{definition}
\newtheorem{defn}[thm]{Definition}
\newtheorem{rem}[thm]{Remark}
\newtheorem{eg}[thm]{Example}
\newtheorem{fact}[thm]{Fact}
\newtheorem{observe}[thm]{Observation}

\numberwithin{equation}{section}

\def\llangle{\langle\!\langle}
\def\rrangle{\rangle\!\rangle}

\renewcommand{\tablename}{\textbf{Table}}
\newcommand{\rpm}{\sbox0{$1$}\sbox2{$\scriptstyle\pm$}
  \raise\dimexpr(\ht0-\ht2)/2\relax\box2 }

\newcommand{\R}{\mathbb{R}}
\newcommand{\eps}{\varepsilon}
\renewcommand{\rho}{\varrho}
\renewcommand{\phi}{\varphi}

\makeatletter
\newcommand*\bigcdot{\mathpalette\bigcdot@{.5}}
\newcommand*\bigcdot@[2]{\mathbin{\vcenter{\hbox{\scalebox{#2}{$\m@th#1\bullet$}}}}}
\makeatother

\tikzstyle{nd} = [anchor=base, inner sep=0pt]
\tikzstyle{ndpic} = [remember picture, baseline, every node/.style={nd}]

\newcommand{\E}{\mathbb E}

\def\beq{\begin{equation}}
\def\eeq{\end{equation}}
\def\ba{\begin{enumerate}[(a)]}
\def\bei{\begin{enumerate}[(i)]}
\def\be{\begin{enumerate}[(1)]}
\def\ee{\end{enumerate}}
\def\bi{\begin{itemize}}
\def\ei{\end{itemize}}
\def\beg{\begin{eg}}
\def\eeg{\end{eg}}
\def\bd{\begin{defn}}
\def\ed{\end{defn}}
\def\bt{\begin{thm}}
\def\et{\end{thm}}
\def\bl{\begin{lemma}}
\def\el{\end{lemma}}
\def\bfac{\begin{fact}}
\def\efac{\end{fact}}

\def\bc{\begin{cor}}
\def\ec{\end{cor}}
\def\bp{\begin{prop}}
\def\ep{\end{prop}}
\def\bo{\begin{observe}}
\def\eo{\end{observe}}
\def\bas{\begin{assumption}}
\def\eas{\end{assumption}}
\def\RR{\mathbb{R}}
\def\CC{\mathbb{C}}
\def\EE{\mathbb{E}}
\def\ZZ{\mathbb{Z}}
\def\NN{\mathbb{N}}

\def\PP{\mathbb{P}}

\def\ii{\item}

\def\beg{\begin{eg}}
\def\eeg{\end{eg}}

\def\p{\thinspace}
\def\CC{\mathcal{C}}
\def\CI{\mathcal{I}}

\def\CK{\mathcal{K}}
\def\CS{\mathcal{S}}

\def\CY{\mathcal{Y}}
\def\CM{\mathcal{M}}
\def\fC{\mathfrak{C}}
\def\SE{\mathscr{E}}

\def\I{\mathfrak{I}}
\def\ff{\mathfrak{f}}
\def\fa{\mathfrak{a}}

\def\T{\mathbb{T}}

\newcommand{\1}[1]{\mathbbm{1}_{\{#1\}}}

\def\nablae{\nabla_{\!\!\eps}}
\def\Deltae{\Delta_\eps}
\def\DDelta{\Delta}
\def\${|\!|\!|}

\makeindex

\numberwithin{equation}{section}
\numberwithin{table}{section}

\startlocaldefs
\endlocaldefs

\begin{document}

\begin{frontmatter}

% "Title of the paper"
\title{Stochastic PDE limit of the dynamic ASEP}
\runtitle{SPDE limit of the dynamic ASEP}
\runauthor{Corwin, Ghosal \& Matetski}
% indicate corresponding author with \corref{}
% \author{\fnms{John} \snm{Smith}\corref{}\ead[label=e1]{smith@foo.com}\thanksref{t1}}
% \thankstext{t1}{Thanks to somebody}
% \address{line 1\\ line 2\\ printead{e1}}
% \affiliation{Some University}

\begin{aug}
  \author{\fnms{Ivan}  \snm{Corwin}\textsuperscript{1}\corref{}\ead[label=e1]{corwin@math.columbia.edu}\hspace{0.1in}}
  \author{\fnms{Promit} \snm{Ghosal}\textsuperscript{2}\ead[label=e2]{pg2475@columbia.edu}\hspace{0.1in}}
  \author{\fnms{Konstantin} \snm{Matetski}\textsuperscript{3}\ead[label=e3]{matetski@math.columbia.edu}\vspace{0.1in} \\ {\it Columbia University}}

\address{\textsuperscript{1}Department of Mathematics, 2990 Broadway, New York, NY 10027, USA\\ \printead{e1}}
\address{\textsuperscript{2}Department of Statistics, 1255 Amsterdam, New York, NY 10027, USA\\ \printead{e2}}
\address{\textsuperscript{3}Department of Mathematics, 2990 Broadway, New York, NY 10027, USA\\ \printead{e3}}

   %\affiliation{Columbia University}

  %\thankstext{t2}{Supported by NSF grant DMS-1150435}

  %\runauthor{Ghosal}

%\author[I. Corwin]{Ivan Corwin}
%\address{I. Corwin, Columbia University,
%Department of Mathematics,
%2990 Broadway,
%New York, NY 10027, USA}
%\email{ivan.corwin@gmail.com}

%\note{change addresses to be math and stats;funding etc in acknowledgements}
%\address{\textsuperscript{1} Department of Mathematics, 2990 Broadway, New York, NY 10027\\ \printead{e1}}
 %\address{ Department of Statistics, 1255 Amsterdam Avenue, New York, NY 10027\\ \printead{e1}}

\end{aug}

\begin{abstract}
\noindent We study a stochastic PDE limit of the height function of the dynamic asymmetric simple exclusion process (dynamic ASEP). Introduced in \cite{Borodin17a}, the dynamic ASEP has a jump parameter $q\in (0,1)$ and a dynamical parameter $\alpha>0$. It degenerates to the standard ASEP height function when $\alpha$ goes to $0$ or $\infty$. We consider a \emph{very weakly asymmetric regime}, i.e. for $\varepsilon$ tending to zero we set $q=e^{-\varepsilon}$. We show that under the parabolic scaling the height function of the dynamic ASEP converges to the  solution of the space-time Ornstein-Uhlenbeck (OU) process. We also introduce the dynamic ASEP on a ring with generalized rate functions. Under the very weakly asymmetric scaling, we show that the dynamic ASEP (with generalized jump rates) on a ring also converges to the solution of the space-time OU process on $[0,1]$ with periodic boundary conditions.
\end{abstract}

%\begin{abstract}
%\end{abstract}

%\begin{keyword}[class=MSC]
%\kwd[Primary ]{}
%\kwd{}
%\kwd[; secondary ]{}
%\end{keyword}


\end{frontmatter}

% AOS,AOAS: If there are supplements please fill:
%\begin{supplement}[id=suppA]
%  \sname{Supplement A}
%  \stitle{Title}
%  \slink[doi]{10.1214/00-AOASXXXXSUPP}
%  \sdatatype{.pdf}"
%  \sdescription{Some text}
%\end{supplement}

%\documentclass[aoas,preprint]{imsart}

%\endlocaldefs

%\title{On Univariate Convex Regression}

%\author{Promit Ghosal and Bodhisattva Sen\footnote{Supported by NSF grant DMS-1150435} \\
%Columbia University, New York}

%\maketitle

%\begin{keyword}[class=MSC]
%\kwd[Primary ]{60K35}
%\kwd{60K35}
%\kwd[; secondary ]{60K35}
%\end{keyword}

%{\bf Keywords:}

%\note{Add table of contents -- Sections}
\tableofcontents

\begin{figure}[t]
\begin{center}
{\scalebox{.65}{\includegraphics{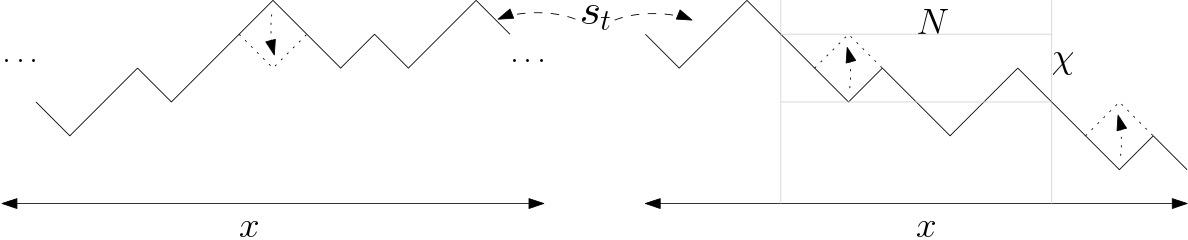}}}  \caption{Dynamic ASEP height function $s_t$ on the full line (left) and periodic domain (right). In the periodic case, the $N$ represents the periodicity length and $\chi$ represents the change in the height function over that length.}% For the full line, the rates of height change are given by \eqref{eq:JumpRate} and for the periodic case, the rates are given by \eqref{eq:ModJumpRate_intro}.}
\end{center}
\end{figure}

\section{Introduction}\label{introduction}
%Recently, a dynamical version of the stochastic six vertex model is introduced in the work of Borodin which is further supplemented in the work of Aggarwal. One of the degeneration of the dynamical higher spin vertex model is dynamical q-Hahn boson model. In this work, we try to find some duality relation between the dynamical q-Hahn boson model and q-Hahn TASEP. It is expected that this duality relation might reveal the limiting behavior of the dynamical q-Hahn boson model. Duality between two interacting particle system connects a system with finite number of particles with another system consisting infinite number of the particles. To begin with, we recall the definition of the duality between the two Markov processes.

%The dynamic ASEP was introduced by \cite{Borodin17a} as a degeneration of the stochastic Interaction Round-a-Face model (which was constructed using the representation theory of the elliptic quantum group $E_{\tau, \eta}(\mathfrak{sl}_2)$). 
The dynamic ASEP, introduced in \cite{Borodin17a}, is a continuous time Markov process defined in terms of a temporally evolving height function $s_t(x)\in \ZZ$ with time $t\in \RR_{\geq 0}$ and space $x\in \ZZ$. (We may also extend the space variable to $x\in \RR$ by linear interpolation between values at $x\in \ZZ$.) The height function satisfies the \emph{solid-on-solid condition} that $s_t(x)-s_t(x + 1)\in \{1,-1\}$ for all $t$ and $x$. The Markovian update rule depends on two parameters: the asymmetry $q\in (0,1)$ and the dynamic parameter $\alpha\in (0,\infty)$.
The values of the height function at each $x\in \ZZ$ are updated according to independent exponential clocks with the following rates (assuming that the change does not violate the solid-on-solid condition):
\begin{equation}\label{eq:JumpRate}
s_t(x)\mapsto s_t(x)-2 \text{ at rate } \frac{q(1+\alpha q^{-s_t(x)})}{1+ \alpha q^{-s_t(x)+1}}, \quad s_t(x)\mapsto s_t(x)+2  \text{ at rate } \frac{1+\alpha q^{-s_t(x)}}{1+\alpha q^{-s_t(x)-1}}.
\end{equation}
The term \emph{dynamic} alludes to the fact that the above rates depend on the height function.
On taking $\alpha$ to $0$ or $\infty$, the rate parameters converge to $q$ and $1$, or $1$ and $q$, thus recovering the standard ASEP height function rates. The dynamic ASEP with $\alpha\in(0,\infty)$ has a preferred height of roughly $\log_q(\alpha)$ --- above this height there is a tendency for the height function to decrease (i.e. the rate for decreasing exceeds that for increasing) and below this height the opposite happens. In light of this markedly different behavior compared to the standard ASEP, it is natural to explore what becomes of the various asymptotic phenomena enjoyed by the standard ASEP. Notice that via a height shift $s\mapsto s-\log_q\alpha$, the rates above reduce to the $\alpha=1$ rates. (The resulting shifted height function now lives on a shift of the lattice $\ZZ$.) Owing to this observation, we may, without loss of generality assume that $\alpha=1$ throughout the rest of this work. We will still use $\alpha$ in stating our main results, but in the proof we will set $\alpha=1$ to simplify notation.

In this paper we prove a stochastic PDE (SPDE) limit for the dynamic ASEP under \emph{very weakly asymmetric scaling}. This is the first SPDE limit result shown for this type of system. The limiting SPDE is a space-time  Ornstein-Uhlenbeck (OU) process. For the standard ASEP under the same scaling, the limiting SPDE is the additive stochastic heat equation (or Edward-Wilkinson equation). See \cite{MPS89, DG91} for reference. The difference between these two equations is the presence of a linear drift in the OU which introduces a preferred height which the process drifts towards. Thus, the effect of the dynamic parameter survives in our limit.

Before going into greater depth about our present contribution, let us recall the previous work on this process.

Besides introducing the model, \cite{Borodin17a} developed a generalization of the method introduced by \cite{BP16a,BP16b} (in studying non-dynamic higher spin vertex models) in order to compute contour integral formulas for expectations of a class of observables for certain initial conditions (in particular for the wedge, where $s_0(x) = |x|$). Taking the limit $q\to 1$ leads to a dynamic version of the SSEP where the corresponding jump rates are 
\begin{align}
s_t(x)\mapsto s_t(x)-2 \text{ at rate } \frac{s_t(x)-\lambda}{s_t(x)-1-\lambda}, \quad s_t(x)\mapsto s_t(x)+2  \text{ at rate } \frac{s_t(x)-\lambda}{s_t(x)+1 -\lambda},
\end{align}
 for some $\lambda<0$. In that limit, the observables and formulas simplify sufficiently so that \cite{Borodin17a} was able to probe some asymptotics (for wedge initial data).

In particular \cite[Theorem 11.2]{Borodin17a} explores the hydrodynamic scaling for the dynamic SSEP. There is some freedom in how to scale the dynamic parameter. For one choice, the hydrodynamic limit is the same as for the standard SSEP (i.e. governed by the heat equation). Surprisingly, for other choices of scaling for the dynamic parameter, there is no deterministic hydrodynamic limit --- the height function remains random under scaling.

The next development regarding dynamic ASEP was the Markov duality derived by \cite{BC17} between it and the standard ASEP (Section \ref{Prelim} herein). The duality function which intertwines these two processes is the same observable which had arisen in the earlier work of \cite{Borodin17a}. \cite{BC17} also derived a translation invariant, stationary measure for the dynamic ASEP (see Definition \ref{Stationary}).

There are a few other works related to the dynamic ASEP (though we will not make use of them herein). \cite{Aggarwal2018} introduced dynamic analogs of the higher spin vertex models from \cite{Corwin2016}, and \cite{BM18} and \cite{AAA18} introduced other types of dynamic analogs of growth models and vertex models (i.e. the growth rates depend on the height through an additional dynamic parameter). For the dynamic stochastic six vertex model, \cite{BG18} used the formulas from \cite{Borodin17a} to derive a law of large numbers and Gaussian central limit theorem under a particular choice of scaling. %Under the same scaling, \cite{BG18} showed that the standard stochastic six vertex model converges in the law of large numbers sense to the telegraph equation and has fluctuations described by the stochastic telegraph equation.

\subsection{Main results}

We consider the scaling limit of the dynamic ASEP defined on the full line $\ZZ$ as well a more general version of it defined on a finite interval of $\ZZ$ with periodic boundary conditions. %\note{Is the dynamic ASEP in this periodic setting ever properly defined? How does the height function wrap around? Do we assume something like half-filling at the particle level to achieve this?}
In each of these two settings we provide a different method of proof. In the full line case we employ a remarkable generalization of the microscopic Hopf-Cole (or G\"{a}rtner) transform \cite{GJ88} along with some estimates similar to \cite{Bertini1997}. In the periodic setting we employ a variant of ``da Prato-Debussche trick'' from \cite{DaPrato}. This second approach can be used to prove a scaling limit of a generalization of the dynamic ASEP, defined in Section~\ref{sec:periodic_intro}, to which the discrete Hopf-Cole transform cannot be applied. We consider this generalized model in a periodic setting, because this allows us to avoid some challenging technical points regarding the growth of stochastic processes at infinity that would arise in applying the  method of \cite{DaPrato} on the line. %\note{perhaps we should comment on whether why we use these two different methods, and what each one offers?}
Besides technicalities, we believe that both methods present in this paper have value and may find further applications in studying other scaling limits of this or related systems.

\subsubsection{Full line results}\label{sec:FLineintro}

We define the following linear stochastic heat equation with additive space-time white noise (with $A<0$ this is sometimes called a \emph{space-time Ornstein-Uhlenbeck process}) %    \note{are we making use of capital letters for macroscopic variants and lower case for microscopic? If so, lets be consistent and also state this somewhere}
 \begin{align}\label{eq:SHEAdditive}
\partial_t \mathcal{Z}_t(x) = \partial^2_x \mathcal{Z}_t(x) + A\mathcal{Z}_t(x) + B_t\,\xi(t,x)%, \qquad \mathcal{Z}(0, X) =\mathcal{Z}_0(X)
\end{align}
on $[0, +\infty) \times \R$, with the initial state $\mathcal{Z}_0 : \R \to \R$ at time $t = 0$, where $\xi$ is the space-time white noise on $\R^2$ on some probability space $(\Omega, \mathfrak{F}, \mathbb{P})$, $A\leq 0$ is a constant and $B : \R \to \R$ is a differentiable, locally bounded function. The mild formulation of \eqref{eq:SHEAdditive} is
%More precisely, \note{this notation is a bit confusing; namely how the rhs depends on $\bigcdot$. Also, we should say that the mild solution (which is unique) is defined by this relation holding and $\mathcal{Z}$; maybe also give a reference to some spde text.} \note{Perhaps this should all be phrased as a definition, rather than in the more informal tone as below?}
\begin{equation}\label{eq:MildSolSHE}
\mathcal{Z}_t(x) = (e^{t\partial^2_x} \mathcal{Z}_0)(x) + A \int^{t}_0 (e^{(t-s)\partial^2_x}\mathcal{Z}_s)(x)ds +  \int^{t}_{0} B_s (e^{(t-s)\partial^2_x}\xi_s)(x)ds,
\end{equation}
where $(e^{t\partial^2_x} f)(x)$ denotes the action of the heat semigroup $\{e^{t\partial^2_x }\}_{t\geq 0}$ on a function $f:\RR\to \RR$
 via
 \begin{align}
 (e^{t\partial^2_x} f)(x) = \int_{\RR} p_t(x - y)f(y)dy,
 \end{align}
 where $p_{t}(\bigcdot)$ denotes the Green's function of the operator $\partial_t - \partial^2_x$, i.e. $p_t$ is the heat kernel, solving the PDE
 \begin{align}
\partial_t p_t(x) = \partial^2_x p_t(x) , \qquad p_0(x)= \delta_{x},
 \end{align}
 where $\delta_{x}$ is the Dirac delta-function. The last term on the r.h.s. of \eqref{eq:MildSolSHE} is defined as a Wiener integral, see \cite{DZ92} or \cite{Walsh}. For the existence, uniqueness and the continuity of the mild solution \eqref{eq:MildSolSHE}, see \cite[Thm.~3.5]{Walsh}.

We denote the space of all continuous functions on $\RR$ by $\mathcal{C}(\RR)$, and endow it with the topology of uniform convergence on the compact sets of $\RR$. Then $D([0,\infty), \mathcal{C}(\RR))$ denotes the space of all $\mathcal{C}(\RR)$-valued c\`{a}dl\`{a}g functions on $[0,\infty)$, equipped with the Skorokhod topology \cite{Billingsley2}. %Endow the space $C(\RR)$ and $C([0,\infty), C(\RR))$ with the topology of weak convergence.
In what follows, we use the notation ``$\Rightarrow$" to denote weak convergence in the respective topology, and $f^{\eps}\Rightarrow f$ means that the random function $f^\eps$ weakly converges to $f$ as $\eps \to 0$. The $L^n$-norm with respect to the underlying probability measure will be denoted  $\|\bigcdot \|_n := \E[|\bigcdot|^n]^{1/n}$. %For functions like $s(t,x)$ of time and space, we often find it convenient to use the notation $s_t(x):=s(t,x)$.

We are now ready to state our main theorem for the dynamic ASEP on the full line.

\bt\label{Cor1}
 Consider the dynamic ASEP with $\eps$-dependent asymmetry parameter $q = e^{-\eps}$ and fixed dynamic parameter $\alpha\in (0,\infty)$. Define the $\eps$-rescaled height function as
 \begin{align}\label{eq:ReHeight}
 \hat{s}^{\eps}_t(x): = \eps^{\frac{1}{2}}\big(s_{\eps^{-2}t}(\eps^{-1}x) - \log_{q}\alpha\big).
 \end{align}
 Assume that $\hat{s}^{\eps}_0(\bigcdot)$ is a sequence of near stationary initial condition, i.e. for some $u>0$, $\beta \in (0, \frac{1}{4})$ and all $k\in \NN$, there exists $C_0=C_0(u,\beta,k)$ such that the following bounds hold
 \begin{subequations}\label{eqs:NSofS}
\begin{align}
\| e^{|\hat{s}^{\eps}_0(x)|}\|_{2k}&\leq C_0 e^{u |x|},\label{eq:NSofS1}\\
\|\hat{s}^{\eps}_0(x)-\hat{s}^{\eps}_0(x^\prime)\|_{2k}&\leq C_0 |x-x^\prime|^{2\beta} e^{u (|x| + |x^\prime|)},\label{eq:NSofS2}
\end{align}
\end{subequations}
for all $x, x^{\prime} \in \RR$.
 Assume that $\hat{s}^{\eps}_0(\bigcdot)\Rightarrow\mathcal{Z}_0(\bigcdot)$ in $\mathcal{C}(\RR)$. Then $\hat{s}_{\eps}$ converges weakly as $\eps \to 0$ in
$D([0,\infty), \mathcal{C}(\RR))$ to the unique solution of \eqref{eq:SHEAdditive} with $A=-\frac{1}{4}$ and $B\equiv \sqrt{2}$, and initial data $\mathcal{Z}_0$.
 \et

%\note{this discussion needs to be refined a bit. For instance, it seems like \ref{MainTheorem} is the key result and the theorem we are proving follows pretty easily from that. This should be made clear. The issue with this reduction is that then you need to still explain why we can prove \ref{MainTheorem}. All this is presently being done without any reference to the actual microscopic Hopf-Cole transform.}
The proof of Theorem~\ref{Cor1} relies on the remarkable fact that dynamic ASEP admits a microscopic Hopf-Cole (or G\"{a}rtner) transform. Namely, as a function of $t\geq 0$ and $x\in \ZZ$,
\begin{equation}\label{eq:expZ}
e^{(1-\sqrt{q})^2t} \big( \alpha^{\frac{1}{2}}q^{-\frac{s_t(x)}{2}} -\alpha^{-\frac{1}{2}}q^{\frac{s_t(x)}{2}}\big)
\end{equation}
solves a microscopic stochastic heat equation (see Proposition \ref{thm:DSHETheo} for a precise statement). This result is closely related to the one-particle version of  the Markov duality proved in \cite{BC17}. We note that under any minute perturbation of the rates, the duality relation would be lost and consequently the Hopf-Cole transform would fail to satisfy a semi-discrete stochastic heat equation as in Proposition~\ref{thm:DSHETheo} below. Theorem~\ref{MainTheorem} shows that the process in \eqref{eq:expZ} converges to a certain space-time OU process. Since the Hopf-Cole transform of \eqref{eq:expZ} is a bijective continuously differentiable transformation of the height function, the desired convergence claimed in Theorem~\ref{Cor1} follows from Theorem~\ref{MainTheorem} via continuous mapping theorem. Hence, the challenge boils down to proving Theorem~\ref{MainTheorem}. This is done via convergence of martingale problems. The proof of the convergence of a linear martingale problem as well as tightness are standard. Identifying the noise (via the quadratic martingale problem) relies upon a `key estimate' (Lemma~\ref{lem:KeyEstimate}) which is similar to that used in \cite{Bertini1997} (and developed further in subsequent generalizations such as \cite{CSL16}) when proving the KPZ equation limit of the standard ASEP.

It was not immediately clear to us which scalings would produce a meaningful limiting SPDE for dynamic ASEP. In our investigation, we were informed by the analysis of the stationary measure for dynamic ASEP. As we now explain, the stationary measure, defined below, converges to a spatial OU process when $q=e^{-\varepsilon}$, height fluctuations are scaled like $\varepsilon^{-1}$ and space is scaled like $\varepsilon^{-2}$. This suggested that we should apply the same scalings to the dynamic ASEP, though it does not tell us how to scale time. Any candidate SPDE limit for dynamic ASEP should have the spatial OU as its stationary measure. In fact, the space-time OU process has a spatial OU process as its stationary measure (this can be proved by writing the exact formula for the space-time OU process and sending the time variable to infinity, see \cite[Ex.~2.2]{Hai09}).

%\note{what happens if we start with the wedge initial data? Is there a limit in that case? It seem like the OU initial data would be $-\infty$ near the origin and $+\infty$ outside some symmetric set (or width like $\log\alpha$. At least from the particle system this seems like it might make sense.}

\bd[Stationary measure]\label{Stationary}
%\note{need to explain the mod 2 stationarity in $x$ as the reason why we get this initial data for all $x$ not just negative $x$}
We say that $\{s_0(x)\}_{x\in \ZZ}$ is distributed according to the stationary initial data if it equals in law the trajectory of a Markov process in $x$ with the following transition probability from $s_0(x)$ to $s_0(x-1)$:
%\note{I changed from $s_x$ to $s(x)$ since this is the notation used above.  It is fine to use $s$ in the proofs, but better to stick with one notation in the intro.}
\begin{align}
s_0(x-1) &=s_0(x)+1 \quad \text{with probability }\frac{q^{s_0(x)}}{\alpha + q^{s_0(x)}},\\
\qquad s_0(x-1) &=s_0(x)-1 \quad \text{with probability }\frac{\alpha}{\alpha+q^{s_0(x)}},
\end{align}
and with the one-point marginal at any $x\in 2\ZZ$ given by
\begin{align}\label{eq:stationary}
\mathbb{P}\big(s_0(x) =2n\big) = \frac{\alpha^{-2n}q^{n(2n-1)}(1+\alpha^{-1}q^{2n})}{(-\alpha^{-1}, -q\alpha, q;q)_{\infty}}, \quad n\in \ZZ,
\end{align}
and at any $x\in 2\ZZ+1$ given by
\begin{align}\label{eq:stationary1}
\mathbb{P}\big(s_0(x) = 2n+1\big) = \frac{\alpha^{-1}q^{n(2n+1)}(1+\alpha^{-1}q^{2n+1})}{(-q\alpha^{-1}, -\alpha, q; q)_{\infty}}, \quad n \in \ZZ.
\end{align}
Here, $(a_1, a_2, a_3; q)_{\infty} := (a_1; q)_{\infty} (a_2; q)_{\infty} (a_3; q)_{\infty}$, where $(a; q)_{\infty} := \prod_{i = 0}^\infty (1 - q^{i} a)$ is the $q$-Pochhammer symbol, defined for the values $a$ and $q$, for which the infinite product converges.
Theorem~2.15 of \cite{BC17} shows that this initial data is stationary for the dynamic ASEP, meaning that for any fixed $t>0$, $\{s_t(x)\}_{x\in \ZZ}$ has the same distributions as $\{s_0(x)\}_{x\in \ZZ}$.
\ed

% \note{this comes into the definition above}
 Note that the distribution of $\{s_0(x)\}_{x\in \ZZ}$ is spatially stationary up to parity, i.e. one point distributions $s_0(0)$ (resp. $s_0(1)$) is same as $s_{0}(2k)$ (resp. $s_0(2k+1)$) for all $k\in \ZZ$. It is natural, therefore, to look for scaling limits in which the stationary measure has a non-trivial limit, and then to use that non-trivial limit to help guess what the limit of the entire space-time process could be. The following lemma and corollary are proved in Appendix \ref{sec:stationary}.

\bl\label{Lem1}
Let $\{s_0(x)\}_{x\in \ZZ}$ be distributed according to the stationary initial data (Definition~\ref{Stationary}),  and let $\hat{s}^{\eps}_0(x)= \eps^{\frac{1}{2}}(s_0(\eps^{-1}x)-\log_q \alpha)$ be extended piece-wise linearly to $x \in \R$. Then
\be
\ii $\hat{s}^{\eps}_0$ satisfies \eqref{eqs:NSofS};
\ii $\hat{s}^{\eps}_0 \Rightarrow \mathcal{Z}_0$ in $\mathcal{C}(\RR)$, as $\eps \to 0$, where $\mathcal{Z}_0$ is a stationary solution of the SDE
\begin{align}\label{eq:OU}
d\mathcal{Z}_0(x) = -\frac{1}{2}\mathcal{Z}_0(x)dx + \frac{1}{2}d\mathcal{W}(x),
\end{align}
where $x \in \R$ plays a role of a time variable, and $\mathcal{W}$ is a two sided Brownian motion. 
\ee
In particular, $\mathcal{Z}_0$ is a stationary Ornstein-Uhlenbeck process, and $\mathcal{Z}_0(0)$ is a standard Gaussian random variable.
%\note{state a lemma showing that the stationary measure converges to a 1d OU process (and make a clear definition of what that is)}
\el

Applying Theorem \ref{Cor1}, we may now show that the stationary dynamic ASEP converges to the stationary solution to the space-time OU process.% \note{this needs to be defined somewhere}.

\bc\label{Cor2}
Using the notation of Theorem~\ref{Cor1}, consider dynamic ASEP started from the stationary initial data (Definition~\ref{Stationary}). Then the scaled height function $\hat{s}^{\eps}$ converges to the solution of \eqref{eq:SHEAdditive} with $A=-\frac{1}{4}$, $B\equiv \sqrt{2}$ and started from the initial data $\mathcal{Z}_0$, given by the stationary solution of \eqref{eq:OU}. %\note{reword if needed}
 \ec

\subsubsection{Periodic results}
\label{sec:periodic_intro}

%\note{clearly define the model you are considering and then state the result that is proved in section 6}
The proof of Theorem~\ref{Cor1} relies heavily on the microscopic Hopf-Cole transform for the dynamic ASEP. However, this transformation is intimately connected with how the rate functions \eqref{eq:JumpRate} are defined. In particular, it may fail to hold after a slight change in the definition of the rate functions. We will define such a more general model now.

For $\chi \in \ZZ$ and $N \in \NN$ with $\chi \equiv N\, \mathrm{mod}\, 2$ we consider the space $\Omega^N_\chi$ of functions $s : \ZZ \to \RR$ such that $s(x) - s(x+1) \in \{-1, 1\}$ and $s(x + m N) = s(x) + \chi m$ for any integer $m$. For a function $\mathfrak{f}:\RR\to \RR$, we define a generalization of the dynamic ASEP on a finite interval $x \in [0, N) \cap \ZZ$ with periodic boundary conditions and with the following update rates (assuming that they do not violate the condition on the $\pm 1$ slopes):
\begin{align}\label{eq:ModJumpRate_intro}
\begin{split}
s_t(x) &\mapsto s_t(x)-2 ~\text{ at rate }~ \frac{q(1+ q^{-\ff(s_t( x)  - \chi x / N)})}{1+  q^{-\ff(s_t( x) - \chi x / N)+1}}, \\
 s_t(x) &\mapsto s_t( x)+2 ~\text{ at rate }~ \frac{1+ q^{-\ff(s_t(x) - \chi x / N)}}{1+ q^{-\ff(s_t( x) - \chi x / N)-1}}.
 \end{split}
\end{align}
Each change of the height function is extended to all $x \in \ZZ$, so that $s_t(\bigcdot) \in \Omega_\chi^N$.
%We make the shift by $\chi x / N$ in the rate functions \eqref{eq:ModJumpRate_intro} to make sure that the rate functions are the same for values $x + m N$, for $m \in \ZZ$. One can see that if $s(0,\bigcdot) \in \Omega_\chi^N$, then $s(t,\bigcdot) \in \Omega_\chi^N$ for every $t > 0$.
We assume the function $\mathfrak{f}:\RR\to \RR$ to satisfy the following assumption.

\begin{assumption}\label{a:f}
There exist $\mathfrak{a} \geq 0$, $\gamma\in [0, \frac{1}{2})$ and $c \geq 0$ such that $|\mathfrak{f}(z)- \mathfrak{f}(0)- \mathfrak{a}z|\leq c|z|^{\gamma}$ for all $z \in \RR$.
\end{assumption}
The bound $\gamma < \frac{1}{2}$ guarantees quick convergence of $\mathfrak{f}(z)$ to $\mathfrak{a}z$ at infinity. In particular, this assumption is crucial in the estimate in Lemma~\ref{lem:term2}. Note that if $\mathfrak{f}(z)=z$ and $\chi=0$, then we get the rates \eqref{eq:JumpRate}. It is unknown if a Hopf-Cole transform exists for this general dynamic ASEP. However, using the method of \cite{DaPrato}, we can show the rescaled height function of the periodic dynamic ASEP with the rate functions \eqref{eq:ModJumpRate_intro} converges to the solution of \eqref{eq:SHEAdditive} with suitably chosen constants. In the statement of the following theorem, we use the standard H\"{o}lder spaces $\CC^\eta$ on $\RR$, for $\eta > 0$, with the norm $\| \cdot \|_{\CC^{\eta}}$.

\begin{thm}\label{thm:tight_intro}
Let $s_t(x)$ be the height function of the periodic generalized dynamic ASEP with the rates \eqref{eq:ModJumpRate_intro}, with a function $\mathfrak{f}$ satisfying Assumption~\ref{a:f}, with period $N \in \NN$ and with $\chi \in \ZZ$ such that  $\chi \equiv N\, \mathrm{mod}\, 2$, and let $s_0(\bigcdot) \in \Omega^N_\chi$. Let us denote $\eps = \frac{1}{N}$, and let $q = e^{-\eps}$ and
$$\hat s^\eps_t(x) := \eps^{\frac12} (s_{\eps^{-2} t}(\eps^{-1} x) - \chi x)$$
be the rescaled height function, extended piece-wise linearly to $x \in \R$. Note that $\hat s^\eps_t(x)$ is $1$-periodic in the variable $x$.

For some $\eta \in (\frac{1}{3}, \frac{1}{2})$ assume that there exists a $1$-periodic $\eta$-H\"{o}lder continuous function $\mathcal{Z}_0$ (defined on the same probability space as the initial data $\hat{s}^\eps_0$), such that
\begin{align}
 \limsup_{\eps \to 0} \E \|\hat s^\eps_0\|_{\CC^{\eta}} < \infty, \qquad \lim_{\eps \to 0} \E \|\hat s^\eps_0 - \mathcal{Z}_0\|_{\CC^{\eta}} = 0.
\end{align}
Then, for every $T > 0$ the following bound holds %\note{is it fine to use $t$ instead of $S$?}
\begin{align}\label{eq:s_bound}
 \limsup_{\eps \to 0} \E \Bigl[ \sup_{t \in [0, T]} \|\hat s^\eps_t\|_{\CC^{\eta}}\Bigr] < \infty,
\end{align}
and $\hat{s}^{\eps}$ converges weakly in  $D([0,\infty), \CC(\RR))$ to the $1$-periodic solution of \eqref{eq:SHEAdditive} with $A = -\frac{\fa}{4}$, $B \equiv  \sqrt{2}$, the initial data $\mathcal{Z}_0$, and a spatially $1$-periodic driving white noise $\xi$.
\end{thm}

%\note{explain the idea of the proof and where the main points are presented}
It is unknown if the dynamic ASEP with the general rates \eqref{eq:ModJumpRate_intro} has any ``integrable'' structure, in contrast to the model considered in Theorem~\ref{Cor1}. In particular, invariant measures for this model are unknown and not used in our approach. Also since there is no apparent microscopic Hopf-Cole transform, we have to work `directly' with the stochastic PDEs and their approximations, in the spirit of \cite{DaPrato}. We provide heuristics for our argument in Section~\ref{sec:heuristic}. Note that we have restricted ourselves to the periodic model, in order to avoid significant difficulties with growth of processes at infinity. There are a few instances where this sort of difficulty (in the context of regularity structures \cite{HL18} and paracontrolled distributions \cite{2018arXiv180800354P}) has been surmounted by use of suitable weighted function spaces, though it still requires case-by-case analysis. Although we expect Theorem~\ref{thm:tight_intro} to hold also on the whole line, it is not clear whether one can use such weighted spaces to prove it. Our full-line analysis used to prove Theorem \ref{Cor1} does not require such methods since we have the microscopic Hopf-Cole transform at our disposal.

%This is also a reason why we consider the model on the circle.
%If we stayed on the whole line $\R$, then we would have to consider the exponential weights on the processes, which would create extra difficulties for our analysis.

The assumption $\eta < \frac{1}{2}$ in Theorem~\ref{thm:tight_intro} is natural, since this is the spatial regularity of the solution to the linear stochastic PDE \eqref{eq:SHEAdditive}, see for example \cite{DZ92}. However, the restriction $\eta > \frac{1}{3}$ is a consequence of the method we are using to analyze the discrete stochastic PDE, governing the evolution of $\hat s^\eps$. More precisely, for regularities below $\frac{1}{3}$ we lose control on the non-linearity in this stochastic PDE (see Lemma~\ref{lem:term2}).

\subsection{Further directions}

For the standard ASEP, there are several PDE and SPDE results that may have generalizations to the dynamic setting. For instance, the standard ASEP enjoys a hydrodynamic limit (i.e. a law of large numbers for its height function) which is determined by the integrated inviscid Burgers equation (or Hamilton-Jacobi equation) with quadratic flux \cite{AV87,FKS91,BGRS}. Under very weak asymmetry scaling, the limit inviscid equation is replaced by its viscous analog \cite{MPS89,GJ88,DG91}. When the scaling of the asymmetry decreases from being weak and very weak, the limiting density field of ASEP does not feel the strength of the asymmetry and solves an Ornstein-Uhlenbeck equation with a linear drift \cite{GJ12K}. At the weakly asymmetric scaling, the limiting density field of ASEP solves the KPZ equation \cite{Bertini1997,GJ12K}. We expect that one can show similar results for the dynamic ASEP by making necessary changes in the model and suitable modifications of our arguments. It is presently unclear how the dynamic parameter (and preferred height) influences these hydrodynamic limit results. However, we expect that, at least under very weak asymmetry scaling, it should be possible to use the methods from this paper to answer this question.

The KPZ equation arises as a scaling limit of the height function fluctuations of the weak asymmetry standard ASEP \cite{Bertini1997} (weak asymmetry means $q= e^{-\sqrt{\eps}}$ versus our very weak asymmetry where $q=e^{-\eps}$). It is presently unclear what becomes of the dynamic ASEP under this weak asymmetry. Part of the challenge is that the stationary initial data simply converges to 0. This may suggest that there should be an SPDE limit of the weak asymmetry ($q=e^{-\sqrt{\eps}}$) dynamic ASEP, its solution will tend over time to 0. The exact form of this limiting SPDE is not yet clear. Another possibility is to tune the value of $\alpha$ in a time-dependent manner. This may enable us to access a KPZ equation-type limit. We leave this for future work.

\subsection{Notation}
\label{sec:notation}
%\note{Ivan: We should remind people about the difference between the $\langle\rangle$ and $[]$ quadratic variation processes here or somewhere later.}

We define here several objects which are used throughout the article. We use the standard notation $\vee$ and $\wedge$ for the maximum and minimum functions respectively. As before, the $L^n$-norm with respect to the underlying probability measure will be denoted by $\|\bigcdot \|_n := \E[|\bigcdot|^n]^{1/n}$.

For $\eta \in (0,1)$, we denote by $\CC^\eta$ the standard space of $\eta$-H\"{o}lder functions on $\RR$, equipped with the norm $\|\bigcdot\|_{\CC^\eta}$. The H\"{o}lder space $\mathcal{C}^\eta$ of non-integer regularity $\eta \geq 1$ consists of $\lfloor \eta \rfloor$ times continuously differentiable functions whose $\lfloor \eta \rfloor$-th derivative is $(\eta - \lfloor \eta \rfloor)$-H\"{o}lder continuous. When we apply all these norms to functions or distributions on the circle $\T := \RR \slash \ZZ$, we identify them with their periodic extensions. The $\eps$-discretization of $\RR$ will be denoted by $\eps \ZZ$, and we denote by  $\llangle \bigcdot, \bigcdot \rrangle$ and  $\llangle \bigcdot, \bigcdot \rrangle_\eps$ the standard and discretized pairings:
\begin{align}
  \llangle \zeta, \phi \rrangle := \int_{\RR} \zeta(x) \phi(x) dx, \qquad
 \llangle \zeta, \phi \rrangle_\eps := \eps \sum_{x \in \eps\ZZ} \zeta(x) \phi(x).
\end{align}
We prefer to use this non-standard notation for the pairing, to avoid confusions with the bracket processes of martingale.

For a test function $\phi : \RR \to \RR$, we define its $\lambda$-scaled and $x$-centered version $\phi^\lambda_x(y) := \lambda^{-1} \phi(\lambda^{-1}(y-x))$, where $\lambda \in (0,1]$ and $x \in \RR$.
Then for $\eta < 0$ and we define the space $\mathcal{C}^\eta$, which is the Besov space $\mathcal{B}^\eta_{\infty, \infty}$ of distributions $\zeta$ (see \cite{BCD11} for a definition), characterized by the bound
\begin{equation}\label{eq:norm}
|\llangle \zeta, \phi^\lambda_x\rrangle| \leq C \lambda^{\eta},
\end{equation}
uniformly in $\lambda \in (0,1]$, $x \in \RR$ and $\lceil -\eta \rceil$-H\"{o}lder continuous functions $\phi$, with the $\CC^{\lceil -\eta \rceil}$ norm bounded by $1$ and compactly supported in the unit ball, centered at the origin. For a discrete function $\zeta^\eps : \eps \ZZ \to \RR$, we define the norm $\|\zeta^\eps\|_{\CC^\eta_\eps}$ as the smallest constant $C \geq 0$, independent of $\eps$, such that the bound
\begin{equation}
|\llangle \zeta^\eps, \phi^\lambda_x\rrangle_\eps| \leq C (\lambda \vee \eps)^{\eta},\label{eq:norm_eps}
\end{equation}
with the same quantities as in \eqref{eq:norm}. Obviously, a piece-wise linearly extended function $\zeta^\eps$ satisfying $\|\zeta^\eps\|_{\CC^\eta_\eps} \leq C$ also belongs to $\mathcal{C}^\eta$.

  By $\Vert \bigcdot \Vert_{V \rightarrow W}$ we denote the operator norm of a linear map acting from the space $V$ to $W$. For time-dependent functions or distributions $\zeta_t$ we define the following norms
\begin{equation}
 \|\zeta\|_{\CC^\eta_{T, \eps}} := \sup_{t \in [0, T]} \|\zeta_t\|_{\CC^\eta_\eps}, \qquad \|\zeta\|_{L^\infty_T} := \sup_{t \in [0, T]} \| \zeta_t \|_{L^\infty}.\label{eq:norm_time}
\end{equation}

To make our notation lighter, sometimes we will write ``$\lesssim$" for a bound ``$\leq$" up to a multiplier, independent of relevant quantities. We will also use $\mathcal{O}(\eps)$ to denote a function which is bounded in absolute value by $C \eps$ as $\eps\to 0$ for some constant $C$ which does not depend on $\eps$ or any other varying parameters. The function $\mathcal{O}(\eps)$ can be random, in which case the bound holds almost surely for a non-random constant $C$.

%We also need to introduce analogous norms, which measure regularity in space and time together. More precisely, for $\alpha \in (0,1)$ and for a function $\zeta : \RR_{\geq 0} \times \RR \to \RR$, we define the H\"{o}lder norm $\|\zeta \|_{\CC^{\alpha / 2, \alpha}_{T, \eps}}$ to be the smallest constant $C$ such that the bounds
%\begin{equation}
%|\zeta(t,x) - \zeta(t,y)| \leq C |x-y|^\alpha, \qquad |\zeta(t,x) - \zeta(s,x)| \leq C \bigl( \sqrt{t-s} \vee \eps\bigr)^{\alpha}
%\end{equation}
%hold uniformly over $x \neq y \in \RR$ and $0 \leq s < t \leq T$ such that $t - s \leq 1$. Note, that we use the parabolic space, meaning that the regularity in time is half of the regularity in space. Furthermore, for $\alpha < 0$, the norm $\|\zeta \|_{\CC^{\alpha, \st}_{T, \eps}}$ (the superscript $\st$ stands for ``space-time") is the smallest constant such that the bound
%\begin{equation}
%\Bigg| \eps \sum_{x \in \eps\ZZ} \int_\R \zeta(s,y) \phi^\lambda_{z}(s,y) ds \Bigg| \leq C (\lambda \vee \eps)^{\alpha},
%\end{equation}
%holds uniformly in $z \in [-T, T] \times \R$, $\lambda \in (0,1]$ and uniformly over functions $\phi : \R^2 \to \R$, with the $\CC^{\lceil -\alpha \rceil/2, \lceil -\alpha \rceil}$ norm bounded by $1$ and compactly supported in the unit ball centered at the origin. Here, we used a rescaled test function $\phi^\lambda_{(t,x)}(s, y) := \lambda^{-3} \phi(\lambda^{-2}(s-t), \lambda^{-1}(y-x))$ and we postulated $\zeta(s,y) = 0$ for $s < 0$.

We define the discrete derivatives in the usual way: $\nabla^\pm f(x) := \pm (f(x \pm 1) - f(x))$, and we write for brevity $\nabla := \nabla^+$. The discrete Laplacian is defined as $\DDelta := \nabla^+ - \nabla^-$. The rescaled versions of these operators act on functions $\phi: \R \to \R$ as $\nablae^\pm \phi(x) := \pm (\phi(x \pm \eps) - \phi(x)) / \eps$ and $\Deltae := (\nablae^+ - \nablae^-) / \eps$. As before we sometimes write $\nablae$ in place of $\nablae^+$.

Since we are going to work with c\`{a}dl\`{a}g martingales, we will use the two bracket processes associated to them. More precisely, the predictable quadratic covariation $\langle M, N\rangle_t$ of two martingales $(M_t)_{t \geq 0}$ and $(N_t)_{t \geq 0}$ is the unique adapted process with bounded total variation, such that $M_t N_t - \langle M, N\rangle_t$ is a martingale. Furthermore, the quadratic covariation $[M, N]_t$ is defined by
\begin{equation}
[M, N]_t := M_t N_t - M_0 N_0 - \int_{0}^t M_{s-} d N_s - \int_{0}^t N_{s-} d M_s,
\end{equation}
where $M_{s-} := \lim_{r \uparrow s} M_{r}$ is the left limit of $M$ at time $s$. We refer to \cite[Ch.~I.4]{JS03} for properties of these two bracket processes. In particular the difference $[M, N]_t - \langle M, N \rangle_t$ is always a martingale.

\subsection*{Outline}
The rest of paper is organized as follows. In Section~\ref{Prelim}, we show that a modified Hopf-Cole transform of the height function of the dynamic ASEP on the full line satisfies a microscopic stochastic heat equation (SHE). The main result of Section~\ref{Prelim} is Theorem~\ref{MainTheorem}. The section also contains the proof of Theorem~\ref{Cor1} which follows fairly easily from Theorem~\ref{MainTheorem}. Sections \ref{Tightness} and \ref{MartingaleProb} are devoted to the proof of Theorem~\ref{MainTheorem}: Section  \ref{Tightness} provides moment bounds on the microscopic SHE which imply tightness of the space-time process, Section \ref{MartingaleProb} uses martingale problems to identify all limit solutions as the unique one from Theorem~\ref{MainTheorem}.  In Section~\ref{sec:GenDASEP}, we prove Theorem~\ref{thm:tight_intro}. Appendix~\ref{sec:heat} contains some of the important properties of various heat kernels. In Appendix~\ref{sec:integrals} and Appendix~\ref{sec:singular}, we include few other important inputs which are mainly needed for the proof of Theorem~\ref{thm:tight_intro}. Finally, Appendix \ref{sec:stationary} contains a proof of Lemma~\ref{Lem1} and Corollary~\ref{Cor2}.

\subsection*{Acknowledgments}

The authors would like to thank Alexei Borodin for early conversations related to this work as well as an anonymous referee for many useful and substantive comments. Ivan Corwin was partially supported by the Packard Fellowship for Science and Engineering, and by the NSF through DMS-1811143 and DMS-1664650.

\section{Microscopic SHE and proof of Theorem~\ref{Cor1}}\label{Prelim}
%\note{you should give a more complete overview of what is in this section (it contains more than what is said below)}

% - Markov duality between the dynamic ASEP and the standard ASEP. \newline
%     - state spaces and generators of the dynamic ASEP and the standard ASEP.\newline
%     - definition of the Markov duality. \newline
%     - discrete SHE with consistent notations.\newline
% - Quadratic variation result for discrete SHE. Proposition 2.4\newline
%    - Introduce the near stationary initial data. \newline
%    - State the tightness result. \newline
%    - Modify the quadratic variation result. Lemma 2.5\newline
%    - Write the convergence result. \newline
%    - Proof of Theorem 1.1 and Corollary 1.4.

The ultimate goal of this section is to prove Theorem~\ref{Cor1} (this comes at the end of the section). %Along the way, Theorem~\ref{MainTheorem} will be our main technical result.
In Proposition~\ref{thm:DSHETheo} we derive a microscopic stochastic heat equation (SHE) satisfied by a modified Hopf-Cole transform of the dynamic ASEP and compute the quadratic variation of the associated martingale. Under very weakly asymmetric scaling and some assumption on the initial data (Definition~\ref{NStDef}), we take a continuum limit of the discrete SHE in Theorem~\ref{MainTheorem}. %The proof of Theorem~\ref{MainTheorem} relies upon Propositions~\ref{ppn:Tight} and \ref{ppn:MgProb} which are proved later.

 %The duality between the dynamic ASEP and the standard ASEP is a modification of the self-duality of the standard ASEP (see \cite{BCS14, Schutz97}). \cite{CorwinPetrov2015} (see also \cite{JK17}) showed that the self-duality of the standrad ASEP extends to the top level of the stochastic higher spin vertex model. The dynamical version of the higher spin vertex models has been studied in \cite{Aggarwal2018}. However, no duality relation is yet known for those models. Duality between dynamic ASEP and the standard ASEP plays an importatnt role \cite{Borodin17a} to exhibit the moment formulas which does not depend on the dynamical parameter $\alpha$. This helps to find out the asymptotics of the SSEP ($q\to 1$ limit of the dynamic ASEP). Later, \cite{BC17} showed that the evolution of the expectation of the duality function of the dynamic ASEP is also $\alpha$-independent irrespective of the initial condition. This is the key idea behind Theorem~\ref{MainTheorem}. In the following discussion, we translate this idea into the language of a semi-discrete stochastic differential equation which will be amenable for further analysis.
%For convenience, we use the shorthand $s_t(x)$ to denote the height function $s(t,x)$, and we sometimes drop the $t$.
Dynamic ASEP depends on two parameters $q$ and $\alpha$, however, up to a shift in the height function by $\log_q \alpha$, the parameter $\alpha$ can always be set to $1$. Under this shift, the height function will live on a shift of the integer lattice. We will not labor this point further since in our scaling limits, such a shift is inconsequential. The state space of the dynamic ASEP is
\begin{align}\label{eq:StateSpace}
\mathcal{S}:=\big\{s=\{s(x)\}_{x\in \ZZ}: |s(x+1)-s(x)|=1, \forall x\in \ZZ\big\}.
\end{align}
Each site $x \in \ZZ$ has independent ``up'' and ``down'' step exponential clocks with rates given below. When the clock at $x$ rings, the height $s(x)$ is updated according to the following rule, assuming that the update does not result in $s$ exiting  $\mathcal{S}$:%\vspace{-0.1in}
\begin{align}\label{eq:rates}
\begin{split}
s(x) &\mapsto s(x)-2 \text{ at rate } a^{\downarrow}(s(x)):=\frac{q(1+ q^{-s(x)})}{1+ q^{-s(x)+1}} ,\\
s(x) &\mapsto s(x)+2 \text{ at rate } a^{\uparrow}(s(x)):=\frac{1+ q^{-s(x)}}{1+ q^{-s(x)-1}}.
\end{split}
\end{align}
The infinitesimal generator of the dynamic ASEP is denoted by $\mathcal{L}$ and acts on functions $f: \mathcal{S}\to \RR$ as
\begin{equation}\label{eq:generator1}
(\mathcal{L} f)(s) = \sum_{x\in \ZZ}\Big[\eta^{\downarrow}_x(s)\, a^{\downarrow}(s(x))\big(f(s^{\,x,-2})-f(s)\big) + \eta^{\uparrow}_x(s) \, a^{\uparrow}(s(x))\big(f(s^{\,x,+2})-f(s)\big)\Big] ,
\end{equation}
where $\eta^{\downarrow}_x(s) := \1{s(x)>s(x-1) \vee s(x+1)}$ and $\eta^{\uparrow}_x(s) := \1{s(x)<s(x-1)\wedge s(x+1)}$, and where the height functions $s^{\,x,\pm2}$ are obtained from $s$ by replacing $s(x)$ with $s(x) \pm 2$.

The following proposition associates a microscopic SHE to the dynamic ASEP.

\bp\label{thm:DSHETheo}
Consider the dynamic ASEP process $s_t(x)$ with $q = e^{-\eps}$ (as explained above, we have taken $\alpha=1$ without lost of generality). Let us define the constant $\theta := (1-\sqrt{q})^{2}$ and the function
%\note{Do you think we really need to keep the subscript 1? Do you ever use this notation with a large subscript?}
\begin{align}\label{eq:ZCal}
Z_t(x) := e^{\theta t}\big(q^{-\frac{ s_t(x)}{2}} - q^{\frac{ s_t(x)}{2}}\big) = 2 e^{\theta t} \sinh \Bigl(\frac{\eps s_t(x)}{2}\Bigr),
\end{align}
for $t\in \RR_{\geq 0}$ and $x\in \ZZ$. Let furthermore %\vspace{-0.1in}
\begin{align}\label{eq:Martingale}
M_t(x) &:=  Z_t(x) - \int^{t}_0\big(\mathcal{L}Z_r(x) + \theta Z_r(x)\big)dr,
\end{align}
where the generator $\mathcal{L}$, defined in \eqref{eq:generator1}, acts on $s_r(x)$ in $Z_r(x)$.

Then, for each $x$, the process $t \mapsto M_t(x)$ is a martingale with respect to the natural filtration of $\{s_t\}_{t\geq 0}$ and
%\note{the below lhs is a bit confusing. The notation does not make the $s$ dependence clear, so it is not so obvious how the generator is acting on this expression. Either we should include the $s$ explicitly or at least explain clearly that we suppress it (assuming this does not lead to confusion). Also, in the above equation the role of $t$ as time is not at all clear. I imagine some of these notation issues propagate and need to be fixed}
% for any $x\in \ZZ$ and $t\in \RR_{\geq 0}$, %denoting $M(x,\vec{s}(t)):=Z(t,x)- \int^{t}_0\mathcal{L} Z(t,x)dr$,
   $Z$ satisfies the following microscopic SHE\vspace{-0.1in}
 \begin{align}\label{eq:sdSHE}
dZ_t(x) = \sqrt q\DDelta Z_t(x)dt + dM_t(x),
\end{align}
where the nearest-neighbour discrete Laplacian $\DDelta$, defined in Section~\ref{sec:notation}, acts on the variable $x$. Moreover, the martingales $M_t(x)$ have the following properties:%(see Proposition \ref{DSHEcompute} for more on this martingale).
\begin{enumerate}
\item[(1)] the predictable quadratic covariation of the martingale \eqref{eq:Martingale} is given by
\begin{align}\label{eq:1}
\frac{d}{dt} \big\langle M(y), M(x)\big\rangle_t = \1{x=y}\frac{\eps^2 e^{2 \theta t}}{2}&\big(1+ q^{-s_t(x)}\big)\big(1 + q^{\frac{s_t(x+1)+s_t(x-1)}{2}}\big)\\
&\times \Big(1 - \nabla^{+}s_t(x) \nabla^{-}s_t(x) +\mathcal{E}^{\eps}_t(x) \Big), &&
\end{align}
where $\nabla^{\pm}$ are the discrete derivatives, defined in Section~\ref{sec:notation}, and the function $\mathcal{E}^{\eps}$ satisfies $|\mathcal{E}^{\eps}_t(x)| \leq C \eps$ uniformly in $t$ and $x$, for a non-random constant $C$;
%where $\nabla^{+}\tilde{s}_t(x):=\tilde{s}_{t}(x+1)-\tilde{s}_t(x)$ and $\nabla^{-}\tilde{s}_t(x):=\tilde{s}_{t}(x)-\tilde{s}_{t}(x-1)$.
\item[(2)] for some non-random constant $C > 0$, the following bound holds
\begin{equation}\label{eq:2}
\frac{d}{dt}\langle M(x), M(x)\rangle_t \leq C \eps^2 \big(Z_t(x)^2 + e^{2t \theta}\big).
\end{equation}
  %\note{$\mathfrak{B}_{\eps}$ has the $\eps$ in the subscript whereas just a bit above we had it in the superscript and it involved $x$ in the subscript and depended on $t$. Is this a different function? This is sloppy notation.}
\end{enumerate}
%\note{Not a fan of the $[]$ and $\{\}$ used below. Replace by $()$. Also, there is a general issue here that the time dependence of $s_x$ is not explicit. This will confuse people!}
%\note{The notation $\mathcal{O}(\eps)$ needs to be introduced or explained}
\ep

\begin{rem}
Taking expectation in \eqref{eq:sdSHE} kills the martingale and shows that $\EE[Z_t(x)]$ solves a semi-discrete heat equation. This is essentially the one-particle case of the duality for dynamic ASEP, proved in \cite[Thm.~2.3]{BC17}.
\end{rem}

\begin{proof}[Proof of Proposition~\ref{thm:DSHETheo}]
We will first demonstrate \eqref{eq:sdSHE}, and explicitly construct the martingale therein.  In this proof it will be convenient for us to overload the notation for $Z$ and write $Z_t(x;s_t)$ instead of $Z_t(x)$. This mean the same thing, but explicitly emphasizes the dependence of $Z$ on $s_t$. Owing to the definition of the dynamic ASEP \eqref{eq:rates}, we have
\begin{align}
dZ_t(x;s_t)& = \eta^{\downarrow}_x(s_t)(Z_t(x;s^{x,-2}_t) - Z_t(x;s_t))a^{\downarrow}(s_t(x))dP^{\downarrow}_t(x) \\
&\qquad + \eta^{\uparrow}_{x}(s_t)(Z_t(x;s^{x,+2}_t) -  Z_t(x;s_t))a^{\uparrow}(s_t(x))dP^{\uparrow}_t(x) + \theta Z_t(x;s_t)dt,\label{eq:dZ}
\end{align}
where $\{P^{\downarrow}_t(x)\}_{x\in \ZZ}$ and $\{P^{\uparrow}_t(x)\}_{x\in \ZZ}$ are independent Poisson processes of intensities $1$, and the functions $\eta^{\downarrow}$ and $\eta^{\uparrow}$ are defined below \eqref{eq:generator1}. We introduce the compensated Poisson processes, which are martingales, by setting
\begin{equation}\label{eq:Marrow}
M_t^{\downarrow}(x) := \int^{t}_{0}\eta^{\downarrow}_x(s_r)a^{\downarrow}(s_r(x))d (P^{\downarrow}_r(x) - r),
\end{equation}
and likewise with $\uparrow$ in place of $\downarrow$. Furthermore, we define the martingales $\{M_t(x)\}_{x\in \ZZ}$ via
\begin{equation}
M_t(x):= \int^{t}_0\!\big( Z_r(x;s^{x,-2}_r) - Z_r(x;s_r)\big) d M^{\downarrow}_r(x) + \int^{t}_{0}\!\big(Z_r(x;s^{x,+2}_r) -  Z_r(x;s_r)\big) dM^{\uparrow}_r(x).\quad \label{eq:BarM}
\end{equation}
With these martingales at hand, we can rewrite \eqref{eq:dZ} as
\begin{align}\label{eq:Derivative}
dZ_t(x; s_t) =\mathcal{L}Z_t(x;s_t)dt + \theta Z_t(x;s_t)dt + dM_t(x),
\end{align}
where the generator $\mathcal{L}$ is defined in \eqref{eq:generator1}. To complete the proof of \eqref{eq:sdSHE}, it remains to show
\begin{align}
\mathcal{L} Z_t(x; s_t) + \theta  Z_t(x; s_t) = \sqrt q\DDelta Z_t(x;s_t),
\end{align}
where $\DDelta$ acts on the variable $x$ in $Z_t(x;s_t)$.  This follows from the one-particle duality in \cite[Thm.~2.3]{BC17}, though it can also be easily deduced (in the spirit of the analogous transform for ASEP in \cite{Bertini1997}) by expressing both sides in terms of $Z_t(x;s_t)$, $\eta^{\downarrow}_x(s_t)$ and $\eta^{\uparrow}_x(s_t)$, and then checking that they match over all possible values of the pair $\eta^{\downarrow}_x(s_t)$ and $\eta^{\uparrow}_x(s_t)$. %\note{I cut out the details. We can add them back in, but they were messy and will need to be rewritten}
Now, we will prove properties (1) and (2) of the bracket processes of the martingales $M$. Owing to the independence of the exponential clocks, $\tfrac{d}{dt}\langle M(y), M(x)\rangle_t = 0$ for $y\neq x \in \ZZ$. When $y=x$, \eqref{eq:BarM} yields
\begin{equation}\label{e:bracket}
\frac{d}{dt}\langle M(x), M(x)\rangle_t = A_t(x) + B_t(x),
\end{equation}
where the two terms are
\begin{align}\label{eq:II1}
\begin{split}
A_t(x)&:= \big(Z_t(x;s^{x,-2}_t)- Z_t(x;s_t)\big)^2 \frac{d}{dt}\langle M^{\downarrow}(x), M^{\downarrow}(x)\rangle_t, \\  
B_t(x) &:=\big(Z_t(x;s^{x,+2}_t)- Z_t(x;s_t)\big)^2 \frac{d}{dt}\langle M^{\uparrow}(x), M^{\uparrow}(x)\rangle_t.
\end{split}
\end{align}
%
%
%
%M_t(x):= \int^{t}_0\big( Z(r,x;s^{x,-2}_r) - Z(r,x;s_r)\big) d M^{\downarrow}_r(r,x) + \int^{t}_{0}\big(Z(r,x;s^{x,+2}_r) -  Z(r,x;s_r)\big) dM^{\uparrow}(r,x)\big). \label{eq:BarM}
%
%
%and where we have used the shorthand
%\begin{equation}
%Z^{+} (t,x) := e^{\theta t}(q^{-\frac{\tilde s_t(x)}{2}-1} - q^{\frac{\tilde s_t(x)}{2}+1}), \quad Z^{-} (t,x) := e^{\theta t}( q^{-\frac{\tilde s_t(x)}{2}+1} - q^{\frac{\tilde s_t(x)}{2}-1}).
%\end{equation}
From their definitions as compensated Poisson processes,  $\frac{d}{dt}\langle M^{\downarrow}(x), M^{\downarrow}(x)\rangle_t = \eta^{\downarrow}_x(s_t)a^{\downarrow}(s_t(x))$ and likewise for $\uparrow$. Moreover, we have the following readily checked identities:
\begin{align}
 \eta^{\downarrow}_x(s_t)=\1{s_t(x)> s_t(x-1)=s_t(x+1)} &= \frac{1}{(q-1)^2}\big(q^{\frac{1 - \nabla^+ s_t(x)}{2}} -1\big)\big(q^{\frac{1+\nabla^- s_t(x)}{2}} -1\big),&&\\
\eta^{\uparrow}_x(s_t)=\1{s_t(x)< s_t(x-1)=s_t(x+1)} &= \frac{1}{(q-1)^2}\big(q^{\frac{1+\nabla^+ s_t(x)}{2}} -1\big)\big(q^{\frac{1 - \nabla^- s_t(x)}{2}} -1\big),&&
\end{align}
where $\nabla^{\pm}$ are the discrete derivatives, defined in Section~\ref{sec:notation}. Substituting this into \eqref{eq:II1} and using \eqref{eq:ZCal} and \eqref{eq:rates} yields
\begin{equation}\label{eq:A_B_def}
\begin{aligned}
A_t(x) &=  e^{2 \theta t}\big(q^{\frac{1 - \nabla^+ s_t(x)}{2}} -1\big)\big(q^{\frac{1+\nabla^- s_t(x)}{2}} -1\big)\big(1+ q^{ s_t(x)-1}\big)\big(1+ q^{- s_t(x)}\big),\\
B_t(x) &= q^{-1}e^{2 \theta t}\big(q^{\frac{1+\nabla^+ s_t(x)}{2}} -1\big)\big(q^{\frac{1 - \nabla^- s_t(x)}{2}} -1\big)\big(1+ q^{ s_t(x)+1}\big)\big(1+ q^{- s_t(x)}\big).
\end{aligned}
\end{equation}
We seek to establish that the sum of these two terms can be written as \eqref{eq:1}. 

First, we observer that the values of $s_t$, contributing to $A_t(x)$, satisfy $s_t(x) > s_t(x-1) \vee s_t(x+1)$, and we can write $1+ q^{ s_t(x)-1} = 1+ q^{(s_t(x+1)+ s_t(x-1)) / 2}$. Similarly, for $B_t(x)$ we have $ s_t(x)< s_t(x-1) \wedge s_t(x+1)$ and we can write $1+ q^{ s_t(x)+1} = 1+ q^{(s_t(x+1)+ s_t(x-1)) / 2}$.

Second, we will rewrite the terms in \eqref{eq:A_B_def} which involve gradients. For this, we recall $q=e^{-\eps} \in (0,1]$ and observe that $(1\pm \nabla^{\mp} s_t(x)) / 2 \in \{0,1\}$. Then, using the bound $|e^{a}-1 - a| \leq a^2 e / 2$, which holds for $|a| \leq 1$, we obtain
\begin{align}
\big(q^{\frac{1 - \nabla^+ s_t(x)}{2}} -1\big)\big(q^{\frac{1+\nabla^- s_t(x)}{2}} -1\big) &= \frac{\eps^2}{4} \bigl(1-\nabla^{+}s_t(x)\bigr) \bigl(1+\nabla^{-}s_t(x)\bigr) + \mathcal{E}^{\eps,1}_t(x)\\
&= \frac{\eps^2}{4}\bigl(1 - \Delta s_t(x) - \nabla^{+}s_t(x) \nabla^{-}s_t(x)\bigr) + \mathcal{E}^{\eps,1}_t(x),
\end{align}
where $|\mathcal{E}^{\eps,1}_t(x)| \leq C \eps^3$, uniformly in $t$ and $x$, and for some constant $C > 0$. Similarly, we can write 
\begin{align}
\big(q^{\frac{1+\nabla^+ s_t(x)}{2}} -1\big)\big(q^{\frac{1 - \nabla^- s_t(x)}{2}} -1\big) &= \frac{\eps^2}{4}\big(1+\nabla^+ s_t(x)\big)\big(1 - \nabla^- s_t(x)\big) + \mathcal{E}^{\eps,2}_t(x)\\
&= \frac{\eps^2}{4}\bigl(1 + \Delta s_t(x) - \nabla^{+}s_t(x) \nabla^{-}s_t(x)\bigr) + \mathcal{E}^{\eps,2}_t(x),
\end{align}
where the function $\mathcal{E}^{\eps,2}$ is bounded as $\mathcal{E}^{\eps,1}$. 

Substituting these identities into $A$ and $B$, we arrive at
\begin{equation}
A_t(x)+B_t(x) = e^{2 \theta t}\big(1+q^{-s_t(x)}\big) \bigl(1+q^{\frac{s_t(x+1)+s_t(x-1)}{2}}\bigr) D_t(x),
\end{equation}
where 
\begin{align}
D_t(x) := \frac{\eps^2}{4} (1 + q^{-1})\bigl(1 - \nabla^{+}s_t(x) \nabla^{-}s_t(x)\bigr) - \frac{\eps^2}{4} (1 - q^{-1}) \Delta s_t(x) + \mathcal{E}^{\eps,1}_t(x) + q^{-1} \mathcal{E}^{\eps,2}_t(x).
\end{align}
Using $|\nabla^{\pm}s_t(x)| = 1$ and $q^{-1} = 1 + \mathcal{O}(\eps)$, this gives the required identity \eqref{eq:1}.

Now, we turn to the bound \eqref{eq:2}, for which we need to bound the r.h.s. of \eqref{e:bracket}. Bounding the indicator functions $\eta^{\downarrow \uparrow}$ and the rates $a^{\downarrow \uparrow}$ in \eqref{eq:Marrow} by $1$, and using the definition \eqref{eq:ZCal}, we get
\begin{align}
\frac{d}{dt}\langle M(x), M(x)\rangle_t &\leq \big(Z_t(x;s^{x,-2}_t)- Z_t(x;s_t)\big)^2 + \big(Z_t(x;s^{x,+2}_t)- Z_t(x;s_t)\big)^2\\
&= 4 e^{2 \theta t} \big(\sinh \bigl(\eps s_t(x) / 2 - \eps\bigr) - \sinh \bigl(\eps s_t(x) / 2\bigr)\big)^2\\
&\qquad + 4 e^{2 \theta t} \big(\sinh \bigl(\eps s_t(x) / 2 + \eps\bigr) - \sinh \bigl(\eps s_t(x) / 2\bigr)\big)^2.\label{eq:M_bound_intermediate}
\end{align}
From differentiability of $\sinh(x)$ one can derive $|\sinh(x + y) - \sinh(x)|^2 \leq c y^2 (\sinh(x)^2 + 1)$, which holds uniformly in $x$ and $|y| \leq 1$, where the constant $c$ is independent of $x$ and $y$. Applying this bound to \eqref{eq:M_bound_intermediate}, we obtain \eqref{eq:2}.
\end{proof}

The following property of the process $Z$ will be used in Section~\ref{MartingaleProb}.

 %\note{I think we should write the martingale (which is given in the proof below) explicitly before the proof. Otherwise it is weird to just say that $M$ is a martingale and that the above semi-discrete SHE holds with it. Then we can remove the statement of the martingale quadratic variation from the below proof. Also, this does not strike me as a proposition, more of a lemma.}

\bp\label{DSHEcompute}
For $Z$ defined in Proposition~\ref{thm:DSHETheo}, we have the identity
\begin{equation}\label{eq:Approx2}
 \eps^{-2}e^{-2t\theta}\nabla^{+} Z_t(x) \nabla^{-}Z_t(x) = \nabla^{+} s_t(x) \nabla^{-}s_t(x) +  \mathfrak{B}^{\eps}_t(x),
 \end{equation}
 for a function  $\mathfrak{B}^{\eps}_t(x)$, satisfying $|\mathfrak{B}^{\eps}_t(x)|\leq C(e^{-2 \theta t}Z_t(x)^2 + \eps)$ uniformly in $t$ and $x$.
\ep
\begin{proof}
 Using the definition of $Z$ from \eqref{eq:ZCal} and the discrete derivative $\nabla^-$, we can write 
  \begin{equation}
  e^{-t\theta}\nabla^{-} Z_{t}(x) = q^{-\frac{s_t(x)}{2}} \big(1-q^{\frac{\nabla^- s_t(x)}{2}} \big) - q^{\frac{s_t(x)}{2}} \big(1-q^{- \frac{\nabla^- s_t(x)}{2}}\big).
  \end{equation} 
  Furthermore, recalling that $q = e^{-\eps}$ and using the bound $|e^{a}-1 - a| \leq a^2 e / 2$, for $|a| \leq 1$, the last expression equals 
  \begin{equation}
  \frac{\eps}{2}\big(q^{-\frac{s_t(x)}{2}} + q^{\frac{s_t(x)}{2}}\big)\big(\nabla^{-} s_t(x) + \mathfrak{B}^{\eps, -}_t(x)\big) = \eps \cosh\Bigl(\frac{s_t(x)}{2}\Bigr) \big(\nabla^{-} s_t(x) + \mathfrak{B}^{\eps, -}_t(x)\big),
  \end{equation}
    where $|\mathfrak{B}^{\eps, -}_t(x)| \leq C \eps$, uniformly in $t$ and $x$. Recalling the formula \eqref{eq:ZCal} and using $\cosh^2(x) - \sinh^2(x) = 1$, we obtain
\begin{equation}\label{eq:Nabla-Ct}
  e^{-t\theta}\nabla^{-} Z_{t}(x) = \eps \Big(1+\frac{1}{4} e^{-2 \theta t} Z_{t}(x)^2\Big)^{1/2} \big(\nabla^{-} s_t(x) + \mathfrak{B}^{\eps, -}_t(x)\big).
\end{equation}
Similarly, we can write 
  \begin{equation}\label{eq:Nabla+Ct}
  e^{-t\theta}\nabla^{+} Z_{t}(x) = \eps \Big(1+\frac{1}{4} e^{-2 \theta t} Z_{t}(x)^2\Big)^{1/2} \big(\nabla^{+} s_t(x) + \mathfrak{B}^{\eps, +}_t(x)\big),
\end{equation}
where $\mathfrak{B}^{\eps, +}$ is bounded in the same way as $\mathfrak{B}^{\eps, -}$. Multiplying \eqref{eq:Nabla-Ct} and \eqref{eq:Nabla+Ct}, and using $|\nabla^{\pm}s(x)| = 1$, we arrive at \eqref{eq:Approx2}.
\end{proof}

We aim now to state a convergence result for $Z$. This requires certain assumptions on the initial states.

 \bd\label{NStDef}
 Consider a sequence of random functions $\mathcal{Z}_0^{\eps}:\RR\to \RR$ indexed by $\eps$. We call $\{\mathcal{Z}_0^{\eps}\}_{\eps>0}$  \emph{near stationary initial data} with parameters $u>0$ and $\beta \in \big(0,\frac{1}{4}\big)$, if for all $k\in \NN$, there exist $C = C(u,\beta,k)>0$, such that for all $x, x_1, x_2\in \R$ one has the bounds
  \begin{subequations}\label{eqs:NStDef}
 \begin{align}
 \|\mathcal{Z}_0^{\eps}(x)\|_{2k}& \leq C e^{u|x|},\label{eq:NStDef1}\\
 \|\mathcal{Z}_0^{\eps}(x_1) - \mathcal{Z}_0^{\eps}(x_2)\|_{2k}&\leq C |x_1-x_2|^{2\beta} e^{u (|x_1|+|x_2|)}.\label{eq:NSteDef2}
 \end{align}
 \end{subequations}
 \ed

 With this definition we are ready to state our convergence result for $Z$.

 \bt\label{MainTheorem} 
 Consider the dynamic ASEP $s_t(x)$ with $q=e^{-\eps}$, where without loss of generality we have taken $\alpha=1$ (see the explanation at the beginning of this section). Extend $Z_t(x)$, defined in \eqref{eq:ZCal}, to non-integer $x$ by linear interpolation and define $\mathcal{Z}^{\eps}_t(x) := \eps^{-\frac{1}{2}}Z_{\eps^{-2} t}(\eps^{-1}x)$. Assume that  $\mathcal{Z}^\eps$ starts from an $\eps$-dependent sequence of \emph{near stationary initial data} $\mathcal{Z}_0^{\eps}$ (see Definition~\ref{NStDef}). If $\mathcal{Z}^{\eps}_0\Rightarrow\mathcal{Z}_0$ in $\mathcal{C}(\RR)$ as $\eps\to 0$, then %\note{Ivan: Are we allowing for random initial data here? Perhaps we need to say something more in that case?}
   $\mathcal{Z}^{\eps} \Rightarrow \mathcal{Z}$ in $D([0,\infty), \mathcal{C}(\RR))$, as $\eps\to 0$,
where $\mathcal{Z}$ is the unique solution of \eqref{eq:SHEAdditive}, with $A=0$ and $B_t=\sqrt{2}e^{t/4}$, started from the initial state $\mathcal{Z}_0$. %\note{The definition of $B=B(t)$ needs to be better explained. Also, do you mean $T$?}
 \et

 \begin{proof}
 The moment bounds in Proposition~\ref{thm:Tightnes} and an argument similar to \cite[Thm.~3.3]{Bertini1997} (see also \cite[Prop.~1.4]{Dembo2016}), readily yield tightness of $\{\mathcal{Z}^{\eps}\}_{\eps > 0}$ with respect to the weak topology of $D([0,\infty), \mathcal{C}(\RR))$. Identification of the limit with the solution of \eqref{eq:SHEAdditive} is proved in Proposition~\ref{ppn:MgProb}.
\end{proof}

% \bp\label{ppn:Tight}
% In the settings of Theorem~\ref{MainTheorem}, $\{\mathcal{Z^{\eps}}\}_{\eps > 0}$ is tight with respect to the weak topology of the Skorokhod space $D([0,\infty), \mathcal{C}(\RR))$.
% \ep

% \bp\label{ppn:MgProb}
% In the settings of Theorem~\ref{MainTheorem}, let $\mathcal{Z}^{\eps}_0(\bigcdot)\Rightarrow\mathcal{Z}_0(\bigcdot)$ in $\mathcal{C}(\RR)$ as $\eps\to 0$. Then every convergent subsequence of $\{\mathcal{Z}^{\eps}\}_{\eps > 0}$ in $D([0,\infty), \mathcal{C}(\RR))$ has the same limit, which is the unique solution of \eqref{eq:SHEAdditive}, with $A=0$ and $B_T=e^{T/4}$, started from the initial state $\mathcal{Z}_0$.
% \ep

%\note{I have made edits up to here}
With this result at hand, we close this section by proving Theorem~\ref{Cor1}. For this, it will be convenient to write a formula for $\mathcal{Z}^{\eps}$, which follows from \eqref{eq:ZCal}:
\begin{equation}\label{e:Z_eps_formula}
\mathcal{Z}^{\eps}_t(x) = 2 \eps^{-\frac{1}{2}} e^{\eps^{-2} \theta t} \sinh \Bigl( \frac{\sqrt \eps \hat{s}^{\eps}_t(x)}{2}\Bigr), 
\end{equation}
where $\hat{s}^{\eps}$ is defined in \eqref{eq:ReHeight}.

\begin{proof}[Proof of Theorem~\ref{Cor1}]
Owing to \eqref{e:Z_eps_formula}, we write $\hat{s}^{\eps}_0(x)= 2\eps^{-\frac{1}{2}}\sinh^{-1}\big(\sqrt{\eps}\mathcal{Z}^{\eps}_0(x)/2\big)$. By using the fact that $\sinh^{-1}$ is a bijective continuously differentiable function, we prove Theorem~\ref{Cor1} when $\{\mathcal{Z}^{\eps}_t(\cdot)\}_{t\geq 0}$ weakly converges as $\eps$ goes to $0$. For showing the weak convergence of $\{\mathcal{Z}^{\eps}_t(\cdot)\}_{t\geq 0}$, we apply Theorem~\ref{MainTheorem}. To this end, we first show that $\mathcal{Z}^{\eps}_0(\cdot)$ satisfies \eqref{eqs:NStDef} under \eqref{eqs:NSofS}. 

Using \eqref{e:Z_eps_formula} and the bound $|\sinh (x)| \leq |x| e^{|x|}$, we obtain
\begin{equation}\label{eq:Quad}
|\mathcal{Z}^{\eps}_0(x)| \leq |\hat{s}^{\eps}_0(x)|e^{\sqrt{\eps}|\hat{s}^{\eps}_0(x)| / 2}.
\end{equation}
Taking the $L^{2k}$-norm, applying the Cauchy-Schwarz inequality and \eqref{eq:NSofS1}, we obtain
\begin{align}
 \|\mathcal{Z}_0^{\eps}(x)\|_{2k} \leq \|\hat{s}^{\eps}_0(x)\|_{4k} \| e^{\sqrt{\eps}|\hat{s}^{\eps}_0(x)| / 2}\|_{4k} \leq C e^{u(1 + \sqrt \eps / 2)|x|} \leq C e^{3 u |x| / 2},
\end{align}
for a constant $C$, depending on $k$. The second inequality follows by combining 
\begin{align}
|\hat{s}^{\eps}_0(x)|\leq e^{u|\hat{s}^{\eps}_0(x)|}, \quad \| e^{\sqrt{\eps}|\hat{s}^{\eps}_0(x)| / 2}\|_{4k}\leq \| e^{|\hat{s}^{\eps}_0(x)|}\|^{\sqrt{\epsilon}/2}_{4k}, \quad \forall x\in \ZZ \label{eq:ReqIneq}
\end{align}
with an upper bound on $ \| e^{|\hat{s}^{\eps}_0(x)|}\|$ from \eqref{eq:NSofS1} for $\eps$ small and $x$ tends $\infty$. The first bound in \eqref{eq:ReqIneq} holds since $|x|\leq e^{|x|}$ for all $x\in \ZZ$ and the second bound is obtained via H\"older's inequality. We use $|\sinh(x) - \sinh(y)| \leq |x - y| e^{|x| + |y|}$ to get
\begin{equation}\label{eq:ResBd}
|\mathcal{Z}^{\eps}_0(x_1)- \mathcal{Z}^{\eps}_0(x_2)|\leq |\hat{s}^{\eps}_0(x_1) - \hat{s}^{\eps}_0(x_2)|e^{\sqrt{\eps} (|\hat{s}^{\eps}_0(x_1)|+|\hat{s}^{\eps}_0(x_2)| ) / 2}.
\end{equation}
Combining this with \eqref{eqs:NSofS}, we obtain \eqref{eq:NSteDef2}. Then Theorem~\ref{MainTheorem} implies that $\mathcal{Z}^{\eps}$ converges weakly to the unique solution of \eqref{eq:SHEAdditive} with $A=0$ and $B_t = \sqrt{2}e^{t/4}$ in $D([0,\infty), \mathcal{C}(\RR))$ as $\eps\to 0$.

To show the weak convergence of $\hat{s}^{\eps}$, we use the formula \eqref{e:Z_eps_formula}. The weak convergence $\lim_{\eps \to 0}\mathcal{Z}^{\eps}_t(x) = \mathcal{Z}_t(x)$ in $D([0,\infty),\mathcal{C}(\RR))$ and $\lim_{\eps \to 0} \eps^{-2}t\theta = \frac{t}{4}$ imply that 
\begin{equation}
\eps^{-\frac{1}{2}}\sinh \Bigl( \frac{\sqrt \eps \hat{s}^{\eps}_t(x)}{2}\Bigr) = \frac{1}{2} e^{- \eps^{-2} \theta t} \mathcal{Z}^{\eps}_t(x)
\end{equation}
weakly converges to $\frac{1}{2} e^{- t/4} \mathcal{Z}_t(x)$ in the same topology. Hence, continuous differentiability of the function $\sinh^{-1}(x)$ yields the weak convergence in $D([0,\infty),\mathcal{C}(\RR))$:
\begin{align}
\lim_{\eps \to 0}\hat{s}^{\eps}_t(x) &= \lim_{\eps \to 0} 2 \eps^{-\frac{1}{2}} \sinh^{-1} \Bigl(\frac{\sqrt \eps}{2} e^{- \eps^{-2} \theta t} \mathcal{Z}^{\eps}_t(x)\Bigr)\\
&= \lim_{\eps \to 0} 2 \eps^{-\frac{1}{2}} \sinh^{-1} \Bigl(\frac{\sqrt \eps}{2} e^{- \frac{t}{4}} \mathcal{Z}_t(x)\Bigr) = e^{- \frac{t}{4}} \mathcal{Z}_t(x),
\end{align}
where in the last identity we made use of $(\sinh^{-1})'(0) = 1$. Theorem \ref{MainTheorem} and the chain rule imply that $e^{-t/4}\mathcal{Z}$ solves \eqref{eq:SHEAdditive} with $A=-\frac{1}{4}$ and $B\equiv \sqrt{2}$, which completes the proof.
\end{proof}

\section{Moment bounds for solutions of the microscopic SHE}\label{Tightness}
  %One of the main ingredients for showing the convergence of the stochastic process $\{\mathcal{Z}^{\eps}(., .)\}_{\eps}$ is to prove it is a tight sequence of processes in the weak topology of the Skorokhod space $D([0,\infty), \mathcal{C}(\RR))$.
   The following moment bounds are the main result of this section.
   
   \bp\label{thm:Tightnes}
Consider the space time process $\mathcal{Z}^{\eps}$, defined in Theorem~\ref{MainTheorem}, starting from a sequence of near stationary initial data $\mathcal{Z}^{\eps}_0$ (see Definition~\ref{NStDef}) with parameters $u\in \RR_{>0}$ and $\beta \in (0,\frac{1}{4})$. Then, for any $t > 0$ and $k\in \NN$, there exists $C= C(u,\beta, k, t)$, such that
  \begin{subequations}
  \label{e:z_boundsA}
  \begin{align}
  \|\mathcal{Z}^{\eps}_t(x)\|_{2k} &\leq  C e^{u|x|},\label{eq:Tightness1A}\\
  \|\mathcal{Z}^{\eps}_t(x_1)- \mathcal{Z}^{\eps}_t(x_2)\|_{2k} &\leq  C |x_1-x_2|^{2\beta} e^{u(|x_1|+|x_2|)},\label{eq:Tightness2A}\\
  \|\mathcal{Z}^{\eps}_{t_1}(x)-\mathcal{Z}^{\eps}_{t_2}(x)\|_{2k} &\leq  C (\eps^2\vee |t_1-t_2|)^{\beta} e^{u|x|},\label{eq:Tightness3A}
  \end{align}
  \end{subequations}
  for all $x,x_1,x_2\in \RR$ and $t, t_1, t_2\in [0,t]$.
  \ep
  
  In our proof we find it easier to work with microscopic variables. For this, we define the process
  \begin{equation}\label{eq:ModDuality}
  \widetilde{\mathcal{Z}}^{\eps}_t(x) = \eps^{-\frac{1}{2}}Z_t(x),
\end{equation}
so that $\mathcal{Z}^{\eps}_T(X)=\widetilde{\mathcal{Z}}^{\eps}_{\eps^{-2}T}(\eps^{-1}X)$, where  $\mathcal{Z}^{\eps}$ is defined in Theorem \ref{MainTheorem}. Owing to \eqref{eq:sdSHE}, $\widetilde{\mathcal{Z}}^{\eps}_t(x)$ satisfies a microscopic SHE:
 \begin{equation}\label{eq:sdSHE2}
 d\widetilde{\mathcal{Z}}^{\eps}_t(x) = \sqrt q \DDelta \widetilde{\mathcal{Z}}^{\eps}_t(x) dt+ d \widetilde{\mathcal{M}}^{\eps}_t(x),
\end{equation}
where $\widetilde{\mathcal{M}}^{\eps} := \eps^{-\frac{1}{2}}M$. The mild solution to \eqref{eq:sdSHE2} is given by
   \begin{align}\label{eq:ZWeakSol}
   \widetilde{\mathcal{Z}}^{\eps}_t(x) = \sum_{y\in \ZZ}\mathfrak{p}^{\eps}_{t}(x-y)\widetilde{\mathcal{Z}}^{\eps}_0(y) + \sum_{y\in \ZZ} \int^{t}_{0} \mathfrak{p}^{ \eps}_{t-r}(x-y) d \widetilde{\mathcal{M}}^{\eps}_r(y)
   \end{align}
 (in all stochastic integrals we always integrate with respect to the time variable), where $\mathfrak{p}^{\eps}_{t}(x)$ solves the following semi-discrete PDE:\vspace{-0.1in}
  \begin{align}\label{eq:ODE}
  \partial_t \mathfrak{p}^{\eps}_{t}(x) = \sqrt q\DDelta \mathfrak{p}^{\eps}_{t}(x), \quad \mathfrak{p}^{\eps}_0(x) = \1{x=0}.
  \end{align}
  Note that $\mathfrak{p}^{\eps}_t$ is the transition kernel for a continuous time random walk  starting from $x=0$ which jumps symmetrically $\pm 1$ with rate $\sqrt q$.%, i.e. $\mathfrak{p}^{\eps}_t(x)= \mathbb{P}(\mathfrak{R}(t)=x)$ for all $x\in \ZZ$ and $t\geq 0$.
 %Note that the solution of the semi-discrete PDE $\partial_t \mathfrak{p}_{t}(x)= \frac{1}{2}\Delta_x \mathfrak{p}_t(x)$ is given by the transition kernel of a simple symmetric random walk. Via a change in time variable, we get $\mathfrak{p}^{\eps}_{t}(x) =  \mathfrak{p}_{\theta_2 t}(x)$.

In terms of that, we can rewrite the bounds in Proposition \ref{thm:Tightnes} as follows:
  \begin{subequations}
  \label{e:z_bounds}
  \begin{align}
  \|\widetilde{\mathcal{Z}}^{\eps}_t(x)\|_{2k} &\leq  C e^{u \eps |x|},\label{eq:Tightness1}\\
  \|\widetilde{\mathcal{Z}}^{\eps}_t(x_1)- \widetilde{\mathcal{Z}}^{\eps}_t(x_2)\|_{2k} &\leq  C (\eps|x_1-x_2|)^{2\beta} e^{u\eps(|x_1|+|x_2|)},\label{eq:Tightness2}\\
  \|\widetilde{\mathcal{Z}}^{\eps}_{t_1}(x)- \widetilde{\mathcal{Z}}^{\eps}_{t_2}(x)\|_{2k} &\leq  C \eps^{2 \beta} (1\vee |t_1-t_2|)^{\beta} e^{u\eps|x|},\label{eq:Tightness3}
  \end{align}
  \end{subequations}
  for all $x,x_1,x_2\in \ZZ$ and $t, t_1, t_2\in [0, \eps^{-2}T]$. Proving these bounds immediately implies that analogous bounds in Proposition \ref{thm:Tightnes}.

  To prove \eqref{e:z_bounds}, we first focus on bounding the second component of \eqref{eq:ZWeakSol}. For fixed $0\leq t_1<t_2$, and for any $t\in [t_1, t_2]$, let us define the processes
  \begin{align}
  \widetilde{M}^{\eps;t_1,t_2}_t(x) &:=  \sum_{y\in \ZZ} \int^{t}_{t_1} \mathfrak{p}^{ \eps}_{t_2-r}(x-y) d \widetilde{\mathcal{M}}^{\eps}_r(y),\label{eq:nabla1}\\
  \widetilde{M}^{\eps;\nabla, t_1,t_2}_t(x_1,x_2) &:= \sum_{y\in \ZZ} \int^{t}_{t_1} \bigl(\mathfrak{p}^{\eps}_{t_2-r}(x_1-y)-\mathfrak{p}^{\eps}_{t_2-r}(x_2-y)\bigr) d \widetilde  {\mathcal{M}}^{\eps}_r(y).\label{eq:nabla2}
  \end{align}
Furthermore, for any $r\in \RR_{>0}$ and $y,x_1, x_2\in \RR$, we define the modified kernels
  \begin{align}
  \bar{\mathfrak{p}}^{\eps}_{r}(y) := (1\wedge r^{-\frac{1}{2}})\mathfrak{p}^{\eps}_{r}(y), \qquad  \bar{\mathfrak{p}}^{\eps;\nabla}_r(x_1,x_2;y) := (1\wedge r^{-\frac{1}{2}})\mathfrak{p}^{\eps;\nabla}_r(x_1,x_2;y),
  \end{align}
  where we have set $\mathfrak{p}^{\eps;\nabla}_r(x_1,x_2;y):=\mathfrak{p}^{\eps}_{r}(x_1-y)-\mathfrak{p}^{\eps}_{r}(x_2-y)$.

    In the following lemma we will derive necessary estimates on the norms of the processes $\widetilde{M}^{\eps;t_1,t_2}$ and $\widetilde{M}^{\eps;\nabla,t_1,t_2}$ which we will use in Section~\ref{TightProof}.
    
 \bl\label{lem:NormBd}
For any $\beta\in (0,\frac{1}{4})$ and any $k\in \NN$, there exist positive constants $C_1=C_1(k)$ and $C_2=C_2(k,\beta)$, such that
 \begin{subequations}
   \label{e:z_lemma1}
  \begin{align}
  \|\widetilde{M}^{\eps;t_1,t_2}_t(x)\|^{2}_{2k} &\leq C_1 \sum_{y\in \ZZ} \int^{t}_{t_1} \bar{\mathfrak{p}}^{ \eps}_{t_2-r}(x-y) \widetilde{\mathcal{A}}^{\eps;k}_r(y) dr,  \label{eq:Iteration1}\\
  \|\widetilde{M}^{\eps;\nabla, t_1, t_2}_t(x_1,x_2)\|^{2}_{2k} &\leq  C_2 (\eps|x_1 -x_2|)^{4\beta} \sum_{y\in \ZZ} \int^{t}_{t_1} \big| \bar{\mathfrak{p}}^{\eps;\nabla}_{t_2-r}(x_1,x_2;y)\big|\, \widetilde{\mathcal{A}}^{\eps;k}_r(y) dr, \label{eq:Iteration2}
  \end{align}
  \end{subequations}
uniformly over $x,x_1,x_2\in \ZZ$ and $t\in [t_1,t_2]$, where 
\begin{align}\label{eq:Aeps}
    \widetilde{\mathcal{A}}^{\eps;k}_r(y) := \eps^2\|\widetilde{\mathcal{Z}}^{\eps}_r(y)\|^2_{2k}+  \eps e^{2r\theta}.
\end{align}
\el

\begin{proof}
 Using Burkholder-Davis-Gundy's inequality (Lemma \ref{lem:BDGone}), we can bound
\begin{subequations}  \label{e:z_BDG1}
\begin{align}
\|\widetilde{M}^{\eps;t_1,t_2}_t(x)\|^{2}_{2k}&\leq C \Bigl(\mathbb{E}\Big[\big[\widetilde{M}^{\eps;t_1, t_2}(x), \widetilde{M}^{\eps;t_1, t_2}(x)\big]_t^{k}\Big]\Bigr)^{\frac{1}{k}},\label{eq:BDG1}\\
\|\widetilde{M}^{\eps;\nabla, t_1, t_2}_t(x_1, x_2)\|^{2}_{2k}& \leq C \Bigl(\mathbb{E}\Big[\big[\widetilde{M}^{\eps;\nabla, t_1, t_2} (x_1, x_2), \widetilde{M}^{\eps;\nabla, t_1, t_2} (x_1, x_2)\big]_t^{k}\Big]\Bigr)^{\frac{1}{k}},\label{eq:BDG2}
\end{align}
\end{subequations}
where $[\bigcdot, \bigcdot]_t$ is the quadratic variation of the martingales with respect to time $t$ (suppressed in the notation for the martingales) and $C>0$ is a constant which depends only on $k$. We will bound the r.h.s.'s of \eqref{eq:BDG1} and \eqref{eq:BDG2} by the respective r.h.s.'s of \eqref{e:z_lemma1}.

%To show \eqref{eq:Iteration1} and \eqref{eq:Iteration2}, we will bound the r.h.s. of \eqref{eq:BDG1} and \eqref{eq:BDG2} respectively by %\note{for the second equation, there seems to be a factor like $|x_1-x_2|^{4\beta}$ missing.}
% \begin{subequations}   \label{e:z_BDG2}
%\begin{align}
%\big(\textrm{r.h.s. \eqref{eq:BDG1}}\big)^{\frac{1}{k}} &\leq C_1 \sum_{y\in \ZZ}\int^{t}_{t_1} \bar{\mathfrak{p}}^{\eps}_{t_2-r}(x-y) \widetilde{\mathcal{A}}^{\eps;k}_r(y)  dr, \label{eq:QVFinaBd}\\
%\big(\textrm{r.h.s. \eqref{eq:BDG2}}\big)^{\frac{1}{k}} &\leq C_2  (\eps|x_1 -x_2|)^{4\beta} \sum_{y\in \ZZ}\int^{t}_{t_1} \bar{\mathfrak{p}}^{\eps;\nabla}_{t_2-r}(x_1,x_2;y) \widetilde{\mathcal{A}}^{\eps;k}_r(y)  dr.\label{eq:QVFinaBd2}
%\end{align}
%\end{subequations}
%for some constants $C_1=C_1(k)>0$ and $C_2=C_2(k,\beta)>0$.

Expanding the quadratic variations yields % $\big[\widetilde{M}^{\eps;t_1, t_2}(x), \widetilde{M}^{\eps;t_1, t_2}(x)\big]_t$ and $\big[\widetilde{M}^{\eps;\nabla, t_1, t_2}(x_1, x_2), \widetilde{M}^{\eps;\nabla, t_1, t_2}(x_1,x_2)\big]_t$, yields
\begin{subequations} 
\begin{align}
\big[\widetilde{M}^{\eps;t_1, t_2}(x), \widetilde{M}^{\eps;t_1, t_2}(x)\big]_t &= \sum_{y\in \ZZ} \int^{t}_{t_1} (\mathfrak{p}^{\eps}_{t_2-r}(x-y))^2 d \big[\widetilde{\CM}^{\eps}(y), \widetilde{\CM}^{\eps}(y)\big]_r,\label{eq:qVExp1} \\
\big[\widetilde{M}^{\eps;\nabla, t_1, t_2}(x_1, x_2), \widetilde{M}^{\eps;\nabla, t_1, t_2}(x_1,x_2)\big]_t &= \sum_{y\in \ZZ} \int^{t}_{t_1} (\mathfrak{p}^{\eps;\nabla}_{t_2-r}(x_1,x_2;y))^2 d \big[\widetilde{\CM}^{\eps}(y), \widetilde{\CM}^{\eps}(y)\big]_r.\qquad\label{eq:qVExp2}
\end{align}
\end{subequations}
We will focus our analysis on \eqref{eq:qVExp1}, and  for \eqref{eq:qVExp2} we give only the key steps, since the bounds follow from similar arguments.

To start with, we separate out the predictable quadratic covariation by writing
\begin{equation}\label{eq:BreakIntoTwo}
\text{r.h.s. of \eqref{eq:qVExp1}} = \widetilde{\mathcal{R}}^{\eps;t_1,t_2}_r(x) +\sum_{y\in \ZZ} \int^{t}_{t_1} (\mathfrak{p}^{ \eps}_{t_2-r}(x-y))^2 d \langle \widetilde{\CM}^{\eps}(y), \widetilde{\CM}^{\eps}(y)\rangle_r,
\end{equation}
where 
\begin{equation}
\widetilde{\mathcal{R}}^{\eps;t_1,t_2}_t(x) :=\sum_{y\in \ZZ} \int^{t}_{t_1} \big(\mathfrak{p}^{\eps}_{t_2-r}(x-y)\big)^2 d \Big([\widetilde{\CM}^{\eps}(y), \widetilde{\CM}^{\eps}(y)]_r -\langle \widetilde{\CM}^{\eps}(y), \widetilde{\CM}^{\eps}(y)\rangle_r\Big).\qquad\label{eq:Reps}
\end{equation}
 For any $y\in \ZZ$, the process $r \mapsto [\widetilde{\CM}^{\eps}(y), \widetilde{\CM}^{\eps}(y)]_r -\langle \widetilde{\CM}^{\eps}(y), \widetilde{\CM}^{\eps}(y)\rangle_r$ is a martingale with respect to the natural filtration of $\{s_r\}_{r\geq 0}$ \cite[Ch.~I.4]{JS03}, and hence so is $t \mapsto \widetilde{\mathcal{R}}^{\eps;t_1,t_2}_t(y)$.
 Combining \eqref{eq:qVExp1} and \eqref{eq:BreakIntoTwo}, and applying the triangle inequality for the $L^k$-norm and the Burkholder-Davis-Gundy inequality (Lemma \ref{lem:BDGone}) for the martingale $\widetilde{\mathcal{R}}^{\eps;t_1,t_2}_t(y)$, we obtain
\begin{align}\label{eq:BDGBd}
\big\| \big[\widetilde{M}^{\eps;t_1, t_2}(x), \widetilde{M}^{\eps;t_1, t_2}(x)\big]_t\big\|_k &\leq C \big\|\big[\widetilde{\mathcal{R}}^{\eps;t_1,t_2}(x), \widetilde{\mathcal{R}}^{\eps;t_1,t_2}(x)\big]_t^{\frac{1}{2}}\big\|_k\\ 
&\qquad + \Big\|\sum_{y\in \ZZ} \int^{t}_{t_1} (\mathfrak{p}^{ \eps}_{t_2-r}(x-y))^2 d\langle \widetilde{\mathcal{M}}^{\eps}(y), \widetilde{\mathcal{M}}^{\eps}(y)\rangle_r\Big\|_k.
\end{align}
We first bound the second term in the r.h.s. of \eqref{eq:BDGBd}, and then later the first one. 

From \eqref{eq:2}, we can deduce an upper bound on the derivative of the bracket process $\langle \widetilde{\CM}^{\eps}(y), \widetilde{\CM}^{\eps}(y)\rangle_r$ yielding
\begin{align}
\sum_{y\in \ZZ} \int^{t}_{t_1} &(\mathfrak{p}^{ \eps}_{t_2-r}(x-y))^2 d\langle \widetilde{\CM}^{\eps}(y), \widetilde{\CM}^{\eps}(y)\rangle_r \leq C \sum_{y\in \ZZ} \int^{t}_{t_1} \bar{\mathfrak{p}}^{\eps}_{t_2-r}(x-y) 2 \eps ( |\widetilde{\mathcal{Z}}^{\eps}_r(y)|^2+ 2 e^{2r\theta}) dr,
\end{align}
for some absolute constant $C$. Here, we made use of the first estimate in \eqref{eq:Est52} for the kernel $\mathfrak{p}^{\eps}$. Taking $L^k$-norm on both sides of the above inequality and using the triangle inequality for the $L^k$-norm yields
\begin{equation}
 \Big\|\int^{t}_{t_1}\sum_{y\in \ZZ} \bar{\mathfrak{p}}^{\eps}_{t_2-r}(x-y)d\langle \widetilde{\mathcal{M}}^{\eps}(y), \widetilde{\mathcal{M}}^{\eps}(y)\rangle_r \Big\|_k \leq C \sum_{y\in \ZZ}\int^{t}_{t_1} \bar{\mathfrak{p}}^{\eps}_{t_2-r}(x-y)\widetilde{\mathcal{A}}^{\eps;k}_r(y) dr, \label{eq:BracketBd}
\end{equation}
where the function $\widetilde{\mathcal{A}}^{\eps;k}$ is defined in \eqref{eq:Aeps}.

Turning to the first term in the r.h.s. of \eqref{eq:BDGBd}, we expand the quadratic variation
 \begin{align}\label{eq:DQuarVar}
\big[\widetilde{\mathcal{R}}^{\eps;t_1,t_2}(x), \widetilde{\mathcal{R}}^{\eps;t_1,t_2}(x)\big]_t = \sum_{y\in \ZZ} \sum_{t_1\leq \tau\leq t} \big(\mathfrak{p}^{\eps}_{t_2-\tau}(x-y)\big)^4\big(\widetilde{\mathcal{Z}}^{\eps}_{\tau}(y) - \widetilde{\mathcal{Z}}^{\eps}_{\tau-}(y)\big)^4,
 \end{align}
where for each $y$, the inner sum is over all $\tau\in [t_1,t]$ which are the random times when transitions $s_{\tau}(y)\to s_{\tau}(y)-2$ or $s_{\tau}(y) \to s_{\tau}(y)+2$ occur, and the number of which is almost surely finite. In the notation above, $f(\tau-)$ refers to the limit of $f(r)$ as $r \uparrow \tau$. Denote the number of such transitions at site $y$ during the time interval $(r_1,r_2]$ by $N_y(r_1,r_2)$. Divide $[t_1, t]$ into $\ell=\lceil t-t_1\rceil$ sub-intervals $\CI_1 =(r_0, r_1]$, through $\CI_{\ell}=(r_{\ell-1}, r_{\ell}]$  where $r_0=t_1$, $r_1=r_0+1, \ldots , r_{\ell-1} =r_{0}+\lfloor t-t_1\rfloor$ and  $r_{\ell} = t$, so each interval has length at most one. Denote $N^{(i)}_y:=  N_y(\CI_{i})$. By direct computations
 \begin{align}\label{eq:OneDiffBd}
 \big(\widetilde{\mathcal{Z}}^{\eps}_{\tau}(y) - \widetilde{\mathcal{Z}}^{\eps}_{\tau-}(y)\big)^2\leq 2 \big(\eps^2\widetilde{\mathcal{Z}}^{\eps}_{\tau}(y)^2 + 2\eps e^{2\theta\tau}\big).%\qquad \textrm{where}\quad \widetilde{B}^{\eps}_r(y):=\eps^2 \big(\widetilde{\mathcal{Z}}^{\eps}_r(r_{i-1}, y) \big)^2+ 4 \eps e^{2\theta r_i}
\end{align}
Next, we show%\km{the factor $e^{|\CI_i|\theta}$ was missing!!!} 
\begin{align}\label{eq:rCI}
\max_{r\in \CI_i} \big|\widetilde{\mathcal{Z}}^{\eps}_{r}(y)\big|\leq 2\big(e^{\theta} \big|\widetilde{\mathcal{Z}}^{\eps}_{r_{i-1}}(y)\big|+ e^{\theta r_{i}}\big)\exp(2\sqrt{\eps} N^{(i)}_{y}).
\end{align}
 For any $r\in \CI_{i}$ and $y\in \ZZ$, $|s_{r}(y)-s_{r_{i-1}}(y)|$ is bounded above by $2N^{(i)}_y$. Recall from \eqref{eq:ZCal} that $\widetilde{\mathcal{Z}}^{\eps}_{r}(y)= 2 \eps^{-\frac{1}{2}}e^{\theta r}\sinh(\eps s_r(y)/2)$, which can be bounded by $2 \eps^{-\frac{1}{2}}e^{\theta r}\sinh(\eps |s_{r_{i-1}}(y)|/2+ \eps N^{(i)}_y)$. Hence, using the identity $\sinh(a+b) = \sinh(a)\cosh(b)+ \sinh(a)\cosh(b)$, for $a=\eps |s_{r_{i-1}}(y)|/2 $ and $b=\eps N^{(i)}_y$, we can bound  
\begin{align}
\big|\widetilde{\mathcal{Z}}^{\eps}_{r}(y)\big| & \leq 2 \eps^{-\frac{1}{2}}e^{\theta r}\big(\sinh(a)\cosh(b) + \cosh(a)\sinh(b)\big)  
\\& \leq 2 \Big(e^{\theta (r - r_{i-1})}\big|\widetilde{\mathcal{Z}}^{\eps}_{r_{i-1}}(y)\big|\cosh(b)+ \eps^{-\frac{1}{2}} e^{\theta r} \sinh(b)\Big).\label{eq:tildeZBd}
\end{align}
The last line follows from $|\sinh(x)| = \sinh(|x|)$ and the inequalities $\cosh(a)\leq 1+\sinh(a)$ and $\sinh(b)\leq \cosh(b)$. Using furthermore $\cosh(b)\leq \exp(2\sqrt{\eps}N^{(i)}_{y})$, $\eps^{-\frac{1}{2}}\sinh(b)\leq \exp\big(2\sqrt{\eps} N^{(i)}_{y}\big)$ and $|r-r_{i-1}|\leq 1$ for all $r\in \CI_i$, we arrive at \eqref{eq:rCI}.  

Combining \eqref{eq:OneDiffBd} and \eqref{eq:rCI} we get
\begin{equation}
\max_{r\in \CI_i}  \big(\widetilde{\mathcal{Z}}^{\eps}_{r}(y) - \widetilde{\mathcal{Z}}^{\eps}_{r-}(y)\big)^2 \leq C \big(\eps^2 \big|\widetilde{\mathcal{Z}}^{\eps}_{r_{i-1}}(y)\big|^2+ \eps e^{2 \theta r_{i}}\big)\exp(4\sqrt{\eps} N^{(i)}_{y}),
\end{equation} 
for a constant $C \geq 0$. Using this bound in \eqref{eq:DQuarVar} yields
  \begin{align}
 \text{r.h.s. of \eqref{eq:DQuarVar}}\leq \sum_{i=1}^{\ell}\sum_{y\in \ZZ} \max_{r\in \CI_{i}}\big(\mathfrak{p}^{\eps}_{t_2-r}(x-y)\big)^4  \,\,N^{(i)}_y \exp\big(8\sqrt{\eps}N^{(i)}_{y}\big) \widetilde{B}^{\eps}_i(y)^2,\label{eq:QVBd}
\end{align}
where $\widetilde{B}^{\eps}_i(y):=C (\eps^2 \widetilde{\mathcal{Z}}^{\eps}_{r_{i-1}}(y)^2+ \eps e^{2\theta r_i})$.
 We may bound the term $(\mathfrak{p}^{\eps}_{t_2-r})^4\leq C(\bar{\mathfrak{p}}^{ \eps}_{t_2-r})^2$ using the first inequality of \eqref{eq:Est52}. Taking the square root of both sides of \eqref{eq:QVBd} and using Minkowski's inequality $(a+b)^{1/2}\leq a^{1/2}+ b^{1/2}$ for any $a,b\geq 0$, we arrive at
  \begin{align}
  (\text{r.h.s. of \eqref{eq:DQuarVar}})^{\frac{1}{2}}\leq \sum_{i=1}^{\ell}\sum_{y\in \ZZ} \max_{r\in \CI_{i}}\bar{\mathfrak{p}}^{\eps}_{t_2-r}(x-y)\,(N^{(i)}_y)^{\frac{1}{2}} \exp\big(4\sqrt{\eps}N^{(i)}_{y}\big)\widetilde{B}^{\eps}_i(y).\label{eq:QVBd3}
  \end{align}
Taking $L^{k}$-norm of the both sides of \eqref{eq:QVBd3} and using the triangle and H\"{o}lder inequalities for that norm yields
\begin{align}
\big\|\big[\widetilde{\mathcal{R}}^{\eps;t_1,t_2}(x)&, \widetilde{\mathcal{R}}^{\eps;t_1,t_2}(x)\big]_t^{\frac{1}{2}}\big\|_k \leq \sum_{i=1}^{\ell}\sum_{y \in \ZZ} \max_{r\in \CI_{i}} \bar{\mathfrak{p}}^{\eps}_{t_2-r}(x-y) \big\|(N^{(i)}_y)^{\frac{1}{2}} \exp\big(4\sqrt{\eps} N^{(i)}_{y}\big) \widetilde{B}^{\eps}_i(y) \big\|_k\\
&\qquad \leq \sum_{i=1}^{\ell}\sum_{y \in \ZZ} \max_{r\in \CI_{i}} \bar{\mathfrak{p}}^{\eps}_{t_2-r}(x-y) \big\|(N^{(i)}_y)^{\frac{1}{2}} \exp\big(4\sqrt{\eps} N^{(i)}_{y}\big) \big\|_{2k}  \big\|\widetilde{B}^{\eps}_i(y) \big\|_{2k}.\label{eq:ExpandR}
\end{align}
Owing to the first inequality of \eqref{eq:Est12}, we have
   \begin{align}\label{eq:pmax}
    \max_{r\in \CI_{i}} \bar{\mathfrak{p}}^{\eps}_{t_2-r}(x-y)\leq C \bar{\mathfrak{p}}^{\eps}_{t_2-r_{i-1}}(x-y),
    \end{align}
     for all $x,y\in \ZZ$, $i=1,\ldots , \ell$ and for some absolute constant $C>0$, when $\eps$ is sufficiently small. Since $N^{(i)}_{y}$ is a Poisson random variable, whose mean is bounded uniformly in the variables $i$ and $y$ (more precisely, the mean is at most $2$), we can bound
  \begin{align}\label{eq:PoissonBd}
\big\| (N^{(i)}_y)^{\frac{1}{2}}\exp\big(4\sqrt{\eps} N^{(i)}_{y}\big)\big\|_{2k} \leq C.
  \end{align}
  Using the triangle inequality for the $L^{2 k}$-norm, we may also bound
  \begin{align}
 \|\widetilde{B}^{\eps}_i(y) \|_{2k} \leq C \Big(\eps^2 \big\|\widetilde{\mathcal{Z}}^{\eps}_{r_{i-1}}(y)\|^2_{2k}+\eps e^{2\theta r_i}\Big). \label{eq:LknormBd}
  \end{align}
  Since $\widetilde{\mathcal{Z}}^{\eps}_{r_{i-1}}(y)$ is measurable w.r.t. the $\sigma$-algebra generated by $\{s_r\}_{r\leq r_{i-1}}$, $N^{(i)}_{y}$ and  $\widetilde{\mathcal{Z}}^{\eps}_{r_{i-1}}(y)$ are independent. Applying the bounds in \eqref{eq:pmax}, \eqref{eq:PoissonBd} and \eqref{eq:LknormBd} to the r.h.s. of \eqref{eq:ExpandR} yields
  \begin{align}
  \text{r.h.s. of \eqref{eq:ExpandR}}\leq Ce^{2\theta}\sum_{i=1}^{\ell}\sum_{y\in \ZZ} \bar{\mathfrak{p}}^{\eps}_{t_2-r_{i-1}}(x-y) \Big(\eps^2\big\|\widetilde{\mathcal{Z}}^{\eps}_{r_{i-1}}(y)\|^2_{2k}+ \eps e^{2\theta r_{i-1}}\Big).
\end{align}
 Approximating the above sum over the integer values of $i$ by the corresponding integral, we arrive at
  \begin{equation}\label{eq:2ndQuadVar}
\big\|\big[\widetilde{\mathcal{R}}^{\eps;t_1,t_2}(x), \widetilde{\mathcal{R}}^{\eps;t_1,t_2}(x)\big]_t^{\frac{1}{2}}\big\|_k \leq C \sum_{y\in \ZZ} \int^{t}_{t_1} \bar{\mathfrak{p}}^{\eps}_{t_2-r}(x-y) \widetilde{\mathcal{A}}^{\eps;k}_{r}(y) dr,
\end{equation}
for a constant $C=C(k)$, where we use the function \eqref{eq:Aeps}. Applying the bounds \eqref{eq:BracketBd} and \eqref{eq:2ndQuadVar} to the r.h.s. of \eqref{eq:BDGBd} finishes the proof of \eqref{eq:Iteration1}.
%\const{use $\tilde{\mathcal A}$ here and above}
%One can now see \eqref{eq:Iteration1} by combining \eqref{eq:QVFinaBd} with \eqref{eq:BDG1}.

Now, we show the main steps in the proof of \eqref{eq:Iteration2}, though leave off the details which are similar to those described above in proving \eqref{eq:Iteration1}.
As in \eqref{eq:BreakIntoTwo}, we can decompose the r.h.s. of \eqref{eq:qVExp2} into two parts and bound then in a similar way as in \eqref{eq:BDGBd}:
\begin{align}
\Big\|\big[\widetilde{\CM}^{\eps;\nabla, t_1,t_2}(x_1,x_2), \widetilde{\CM}^{\eps;\nabla, t_1,t_2}&(x_1,x_2)\big]_t\Big\|_k \leq C\Big\|\big[\widetilde{\mathcal{R}}^{\eps;\nabla,t_1,t_2}(x_1,x_2),\widetilde{\mathcal{R}}^{\eps;\nabla,t_1,t_2}(x_1,x_2)\big]^{\frac{1}{2}}_{t}\Big\|_k \\
&+ \Big\|\sum_{y\in \ZZ} \int^{t}_{t_1} (\mathfrak{p}^{\eps;\nabla}_{t_2-s}(x_1,x_2;y))^2 d \langle \widetilde{\CM}^{\eps}(y), \widetilde{\CM}^{\eps}(y)\rangle_r\Big\|_k,\label{eq:somebound}
\end{align}
where the definition of $\widetilde{\mathcal{R}}^{\eps;\nabla, t_1,t_2}$ is similar to that of $\widetilde{\mathcal{R}}^{\eps;t_1,t_2}$ except that $\mathfrak{p}^\eps_{t_2-r}(x-y)$ is replaced by $\mathfrak{p}^{\eps; \nabla}_{t_2-r}(x_1,x_2;y)$ in \eqref{eq:Reps}.
 In order to obtain a similar bound to \eqref{eq:BracketBd}, we combine the H\"older-type estimate for the kernel $\mathfrak{p}^{\eps}_{t_2-r}$ in the variable $x$ (see the second inequality of \eqref{eq:Est22}) %\footnote{For any $t\in \RR_{\geq 0}$ and $x\in \ZZ$, $\mathfrak{p}^{\eps}_{t} (x)= \eps p^{\eps}_{\eps^2t}(\eps x)$.}
with \eqref{eq:2} and the triangle inequality for the $L^k$-norm. This yields
\begin{align}
\Big\|\sum_{y\in \ZZ} \int^{t}_{t_1} &(\mathfrak{p}^{\eps;\nabla}_{t_2-s}(x_1,x_2;y))^2 d \langle \widetilde{\CM}^{\eps}(y), \widetilde{\CM}^{\eps}(y)\rangle_r\Big\|_k \\
&\leq C (\eps|x_1-x_2|)^{4\beta} \sum_{y\in \ZZ} \int^{t}_{t_1} |\bar{\mathfrak{p}}^{\eps;\nabla}_{t_2-r}(x_1,x_2;y)| \widetilde{\mathcal{A}}^{\eps;k}_{r}(y) dr.  \label{eq:BracketBd2}
\end{align}
An argument, similar to the one which we used to prove \eqref{eq:2ndQuadVar}, yields
\begin{align}
\Big\|\big[\widetilde{\mathcal{R}}^{\eps;\nabla,t_1,t_2}(x_1,x_2)&,\widetilde{\mathcal{R}}^{\eps;\nabla,t_1,t_2}(x_1,x_2)\big]^{\frac{1}{2}}_{t}\Big\|_k\\
&\leq C(\eps|x_1-x_2|)^{4\beta} \sum_{y\in \ZZ} \int^{t}_{t_1} |\bar{\mathfrak{p}}^{\eps;\nabla}_{t_2-r}(x_1,x_2;y)| \widetilde{\mathcal{A}}^{\eps;k}_{r}(y) dr.\label{eq:2ndQuadVar2}
\end{align}
Finally, \eqref{eq:BDG2} follows by combining \eqref{eq:somebound}, \eqref{eq:BracketBd2} and \eqref{eq:2ndQuadVar2}.
\end{proof}

The following lemma develops a microscopic version of a chaos series for $\widetilde{\mathcal{Z}}^{\eps}$.

\bl\label{ChaosSeries}
For any $k\geq 1$ there exists a constant $C=C(k)>0$, such that %\note{Do we really need the range $(t_1,t_2)$ or does $t\in [0,T]$ work too?}
 \begin{subequations}
   \label{e:ChaosSeries}
\begin{align}\label{eq:SquareBdlemma}
\|\widetilde{\mathcal{Z}}^{\eps}_t(x)\|_{2k}^2 &\leq 2\sum_{y\in \ZZ}  \mathfrak{p}^{\eps}_{t}(x-y)\|\widetilde{\mathcal{Z}}^{\eps}_0(y)\|^2_{2k} + 2C\int^{t}_{0}   \sum_{y\in \ZZ}\bar{\mathfrak{p}}^{\eps}_{t-r}(x-y)\widetilde{\mathcal{A}}^{\eps;k}_r(y) dr,\qquad\\
\|\widetilde{\mathcal{Z}}^{\eps}_t(x)\|^2_{2k} &\leq 2\sum_{y\in \ZZ} \mathfrak{p}^{\eps}_{t}(x-y)\|\widetilde{\mathcal{Z}}^{\eps}_0(y)\|^2_{2k}+2\sum_{\ell=1}^{\infty} (C\eps^2)^{\ell} \int\limits_{\vec{r}\in \Delta^{(\ell)}_t} \sum_{\vec{y}\in \ZZ^{\ell}} \mathcal{K}^{\eps; \ell}_{\vec{r}}(\vec{y})\mathcal{D}^{\eps;k}_{r_1,0}(y_1) d\vec{r}, \qquad\label{eq:ChaosSeries}	
\end{align}
\end{subequations}
uniformly in $x\in \ZZ$ and $t \geq 0$, where $\Delta^{(\ell)}_t := \{(r_1,  \ldots, r_{\ell})\in \RR^{\ell}_{\geq 0}:0\leq r_{1}\leq \ldots \leq r_{\ell}\leq t\}$, the function $\widetilde{\mathcal{A}}^{\eps;k}$ is defined in \eqref{eq:Aeps}, and 
\begin{equation}\label{eq:FunctionsKD}
\begin{aligned}
\mathcal{K}^{\eps; \ell}_{\vec{r}}(\vec{y}) &:= \bar{\mathfrak{p}}^{\eps}_{t -r_{\ell}} (x-y_{\ell})\prod_{i=1}^{\ell-1}  \bar{\mathfrak{p}}^{\eps}_{r_{i+1} -r_{i}} (y_{i+1}-y_{i}), \\
 \mathcal{D}^{\eps;k}_{r_1,r_0}(y_1) &:= \sum_{y_0\in \ZZ}\mathfrak{p}^{\eps}_{r_1-r_0}(y_1-y_0)\|\widetilde{\mathcal{Z}}^{\eps}_{r_0}(y_0)\|^2_{2k}+\eps^{-1}e^{2 \theta t},
\end{aligned}
\end{equation}
 for any $\vec{y}\in \ZZ^{\ell}$, $\vec{r}\in \Delta^{(\ell)}_t$ and $r_0 \in \R$, such that $r_0 \leq r_1$.
\el

\begin{proof}
Applying the $L^{2k}$-norm  triangle inequality to the  decomposition of $\widetilde{\mathcal{Z}}^{\eps}$ in \eqref{eq:ZWeakSol} yields
\begin{align}\label{eq:TIneq}
\|\widetilde{\mathcal{Z}}^{\eps}_t(x)\|_{2k} \leq &\sum_{y\in \ZZ} \mathfrak{p}^{\eps}_{t}(x-y)\big\|\widetilde{\mathcal{Z}}^{\eps}_0(y)\big\|_{2k} + \Big\|\int^{t}_{0}\sum_{y\in \ZZ} \mathfrak{p}^{\eps}_{t-r}(x-y) d \widetilde{\mathcal{M}}^{\eps}_r(y)\Big\|_{2k}.
\end{align}
%Note that $\mathfrak{p}^{\eps}_{t}(x-\bigcdot)$ is a  probability measure on $\ZZ$.
%Thus, by Jensen's inequality,
%\begin{align}\label{eq:Jensen}
%\Big\|\sum_{y\in \ZZ} \mathfrak{p}^{\eps}_{t}(x-y)\widetilde{\mathcal{Z}}^{\eps}(0,y)\Big\|_{2k}\leq .
%\end{align}
Since $\mathfrak{p}^{\eps}_{t}(x-y)$ is a probability measure in $y$, applying Cauchy-Schwarz inequality yields
\begin{align}\label{eq:1stCom}
\Big(\sum_{y\in \ZZ} \mathfrak{p}^{\eps}_{t}(x-y)\big\|\widetilde{\mathcal{Z}}^{\eps}_0(y)\big\|_{2k}\Big)^2\leq \sum_{y\in \ZZ}  \mathfrak{p}^{\eps}_{t}(x-y)\|\widetilde{\mathcal{Z}}^{\eps}_0(y)\|^2_{2k},
\end{align}
which bounds the first term on the r.h.s. of \eqref{eq:TIneq}. Applying \eqref{eq:Iteration1}, we bound the second term
\begin{equation}\label{eq:2ndCom}
\Big\|\int^{t}_{0}\sum_{y\in \ZZ} \mathfrak{p}^{\eps}_{t-r}(x-y) d \widetilde{\mathcal{M}}^{\eps}_r(y)\Big\|^2_{2k}\leq  C \int^{t}_{0}   \sum_{y\in \ZZ}\bar{\mathfrak{p}}^{\eps}_{t-r}(x-y)\widetilde{\mathcal{A}}^{\eps;k}_r(y)  dr,
\end{equation}
where $\widetilde{\mathcal{A}}^{\eps;k}$ is defined in \eqref{eq:Aeps}.
Bounding the square of the sum of the two terms on the r.h.s. of \eqref{eq:TIneq} by twice the sum of their squares and applying \eqref{eq:1stCom} and \eqref{eq:2ndCom} gives \eqref{eq:SquareBdlemma}.
%\begin{align}
%\|\widetilde{\mathcal{Z}}^{\eps}_t(x)\|_{2k}^2 &\leq  2\sum_{y\in \ZZ}  \mathfrak{p}^{\eps}_{t}(x-y)\|\widetilde{\mathcal{Z}}^{\eps}_0(y)\|^2_{2k} + 2C\int^{t}_{0}   \sum_{y\in \ZZ}\bar{\mathfrak{p}}^{\eps}_{t-r}(x-y)\widetilde{\mathcal{A}}^{\eps;k}_r(y) dr.\qquad\label{eq:SquareBdproof}
%\end{align}
Since $\widetilde{\mathcal{A}}^{\eps;k}_r(y)$ involves $\|\widetilde{\mathcal{Z}}^{\eps}_t(x)\|_{2k}^2$, the above equation establishes a recursion which produces the series \eqref{eq:ChaosSeries}. To complete the proof we must control the tail of the series \eqref{eq:ChaosSeries}. This follows from the same bounds used in the proof of \eqref{eq:Tightness1} below, so we do not reproduce it here.
\end{proof}

\subsection{Proof of Proposition~\ref{thm:Tightnes}}\label{TightProof}

\noindent {\bf Proof of \eqref{eq:Tightness1}.} Starting with the first term on the r.h.s. of \eqref{eq:ChaosSeries}, we claim that there exists $C=C(k,\beta, u)>0$ (which may change values between lines below) such that
   \begin{align}
   \sum_{y\in \ZZ} \mathfrak{p}^{\eps}_{t} (x-y) \|\widetilde{\mathcal{Z}}^{\eps}_0(y)\|^2_{2k}  &\leq C \sum_{y\in \ZZ} \mathfrak{p}^{\eps}_{t}(x-y)  e^{2\eps u(|x-y|+|x|)}\leq C e^{2\eps u|x|}.\label{eq:T1stComp}
   \end{align}
 The first inequality follows from the  bound \eqref{eq:NStDef1} and the triangle inequality. The second bound follows from the second inequality of \eqref{eq:Est52}, with $\alpha=0$.

 Now, we turn to bound the second term on the r.h.s. of  \eqref{eq:ChaosSeries}. Due to the semigroup property of $\mathfrak{p}^{\eps}_{t}(x-\bigcdot)$, for any $\vec{r}\in \Delta^{(\ell)}_t$ we have
 \begin{align}
 \sum_{(y_1,y_2,\ldots, y_{\ell})\in \ZZ^{\ell}} \mathfrak{p}^{\eps}_{t -r_{\ell}} (x-y_{\ell})\prod_{i=1}^{\ell}  \mathfrak{p}^{\eps}_{r_{i+1} -r_{i}} (y_{i+1}-y_{i}) = \mathfrak{p}^{\eps}_{t}(x-y_0),\label{eq:SemiG2}
\end{align}
 which shows (since $\bar{\mathfrak{p}}^{\eps}(\cdot)= (1\wedge r^{-\frac{1}{2}})\mathfrak{p}^{\eps}(\cdot)$)
  \begin{align}
  \int_{\vec{r}\in \Delta^{(\ell)}_t} &\sum_{(y_0,y_1,\ldots, y_{\ell})\in \ZZ^{\ell+1}} \bar{\mathfrak{p}}^{\eps}_{t -r_{\ell}} (x-y_{\ell})\prod_{i=1}^{\ell-1}  \bar{\mathfrak{p}}^{\eps}_{r_{i+1} -r_{i}} (y_{i+1}-y_{i})\cdot \mathfrak{p}^{\eps}_{r_1-r_0}(y_1-y_0)\|\widetilde{\mathcal{Z}}^{\eps}_{r_0}(y_0)\|^2_{2k} \\&\leq \sum_{y_0\in \ZZ} \mathfrak{p}^{\eps}_{t}(x-y_0)\|\widetilde{\mathcal{Z}}^{\eps}_{0}(y)\|^2 \int\limits_{\vec{r}\in \Delta^{(\ell)}_t} \frac{1}{\sqrt{t-r_{\ell}}}\prod_{i=1}^{\ell-1}\frac{1}{\sqrt{r_{i+1}-r_{i}}} d\vec{r}.\label{eq:SemiG}
  \end{align}   
%  and 
%  \begin{align}
%  \int_{\vec{r}\in \Delta^{(\ell)}} &\sum_{(y_2,\ldots, y_{\ell})\in \ZZ^{\ell}} \bar{\mathfrak{p}}^{\eps}_{t -r_{\ell}} (x-y_{\ell})\prod_{i=1}^{\ell}  \bar{\mathfrak{p}}^{\eps}_{r_{i+1} -r_{i}} (y_{i+1}-y_{i})4\eps^{-1}e^{2\theta r_1} \\&\leq  4\eps^{-1}e^{2\theta t}\int_{\vec{r}\in \Delta^{(\ell)}} \sum_{y_0\in \ZZ} \mathfrak{p}^{\eps}_{t}(x-y_0)\|\widetilde{\mathcal{Z}}^{\eps}_{0}(y)\|^2 \int\limits_{\vec{r}\in \Delta^{(\ell)}} \frac{1}{\sqrt{t-r_{\ell}}}\prod_{i=1}^{\ell-1}\frac{1}{\sqrt{r_{i+1}-r_{i}}} d\vec{r}.
%  \end{align}
 Using this inequality and the fact that $e^{2\theta r_1} \leq e^{2\theta t}$ for any $r_1\in (0,t)$, we now show that
 \begin{align}
\int\limits_{\vec{r}\in \Delta^{(\ell)}_t}  \sum_{\vec{y}\in \ZZ^{\ell}} \mathcal{K}^{\eps; \ell}_{\vec{r}}(\vec{y})\,
\mathcal{D}^{\eps;k}_{r_1,0}(y_1)
 d\vec{r} \leq \mathcal{D}^{\eps;k}_{t,0}(x) \int\limits_{\vec{r}\in \Delta^{(\ell)}_t} \frac{1}{\sqrt{t-r_{\ell}}}\prod_{i=1}^{\ell-1}\frac{1}{\sqrt{r_{i+1}-r_{i}}} d\vec{r}, \label{eq:SeriesTerm}
\end{align}
where we use the functions $\mathcal{K}^{\eps; \ell}_{\vec{r}}(\vec{y})$ and $\mathcal{D}^{\eps;k}_{r_1,0}(y_1)$, defined in \eqref{eq:FunctionsKD}. To prove \eqref{eq:SeriesTerm}, we first express the integral on the left hand side as the sum two terms by distributing the integrals and the sum (inside the integrals) over the two summands of $\mathcal{D}^{\eps;k}_{r_1,0}(y_1)$, namely, $\sum_{y_0\in \ZZ} \mathfrak{p}^{\eps}_{r_1}(x-y)\|\widetilde{Z}^{\eps}_0(y)\|^2_{2k}$ and $\eps^{-1} e^{2\theta r_1}$. After splitting, we note that the first term on the left side of \eqref{eq:SeriesTerm} is same as the left hand side of \eqref{eq:SemiG}. Moreover, we bound the first term using \eqref{eq:SemiG}. By \eqref{eq:SemiG2} and the inequality $e^{2\theta r_1} \leq e^{2\theta t}$, the second term is bounded above by $e^{2\theta t}$ times the integral on the right hand side of \eqref{eq:SeriesTerm}. This proves \eqref{eq:SeriesTerm}.
     
The integral on the r.h.s. of \eqref{eq:SeriesTerm} equals $t^{\frac{\ell}{2}}\Gamma(\frac{1}{2})^{\ell}/\Gamma(\frac{\ell+1}{2})$.
%\note{It seems to be like the calculation below involving what was called the third term (i.e. the exponential) is not correct. First, there is no clear reason why it involves a factor of the form $ e^{2\eps u|x|}$ and second, there seems to be an extra $\eps$ around. Hopefully once we correct the chaos series, the resulting term (which should have one fewer heat kernels) will have better properties for this purpose}
Substituting this value and  applying \eqref{eq:T1stComp} (to bound the term $\sum_{y_0\in \ZZ} \mathfrak{p}^{\eps}_{t}(x-y_0)\|\widetilde{\mathcal{Z}}^{\eps}_0(y_0)\|^{2}_{2k}$ inside $\mathcal{D}^{\eps;k}_{t,0}(x)$) yields
\begin{align}
\int\limits_{\vec{r}\in \Delta^{(\ell)}_t}  \sum_{\vec{y}\in \ZZ^{\ell}} \mathcal{K}^{\eps; \ell}_{\vec{r}}(\vec{y})\, \mathcal{D}^{\eps;k}_{r_1,0}(y_1) d\vec{r}
\leq
\big(C e^{2\eps u|x|}  + \eps^{-1}  e^{2 \theta t}\big)\frac{\Gamma(\frac{1}{2})^\ell}{\Gamma(\frac{\ell+1}{2})}t^{\frac{\ell}{2}}.\label{eq:kTermBd}
\end{align}
Combining this with bounds of the form $\sum_{\ell=1}^{\infty} x^\ell/\Gamma(\ell/2) \leq e^{C x}$, for sufficiently large $C$, yields
 \begin{align}\label{eq:ThirdBd}
 \sum_{\ell=1}^{\infty} (C\eps^2)^{\ell} \int\limits_{\vec{r}\in \Delta^{(\ell)}_t} \sum_{\vec{y}\in \ZZ^{\ell}} \mathcal{K}^{\eps; \ell}_{\vec{r}}(\vec{y})\mathcal{D}^{\eps;k}_{r_1,0}(y_1) d\vec{r}\leq C \big(e^{2\eps u|x|}+ \sqrt{\eps^2 t}e^{2 \theta t}\big)e^{C\eps^2\sqrt{t}}.
 \end{align}
Recalling that $t\leq \eps^{-2}T$ and $\theta= (1-\sqrt{q})^{2}$ for $q=e^{-\eps}$, we have $\sqrt{\eps^2 t}\leq \sqrt{T}$, $t\theta\leq T/4$ and $\eps^2 \sqrt{t}\leq \eps \sqrt{T}$. Combining this bound with that on the first term in \eqref{eq:T1stComp} readily yields \eqref{eq:Tightness1}.

\smallskip
\noindent {\bf Proof of  \eqref{eq:Tightness2}.} Recalling \eqref{eq:ZWeakSol} and using the triangle inequality for the $L^{2k}$-norm, we write
 \begin{align}
 \|\widetilde{\mathcal{Z}}^{\eps}_t(x_1) - \widetilde{\mathcal{Z}}^{\eps}_t(x_2)\|_{2k} &\leq (\mathbf{I}) + (\mathbf{II}) . \label{eq:Tht21}
\end{align}
where
\begin{align}
(\mathbf{I}) &:= \Big\|\sum_{y\in \ZZ}\big(\mathfrak{p}^{\eps}_{t}(x_1-y) - \mathfrak{p}^{\eps}_{t}(x_2-y)\big)\widetilde{\mathcal{Z}}^{\eps}_0(y)\Big\|_{2k}, \nonumber\\
(\mathbf{II}) &:= \Big\|\sum_{y \in \ZZ} \int^{t}_{0} (\mathfrak{p}^{\eps}_{t-r}(x_1-y)- \mathfrak{p}^{\eps}_{t-r}(x_2-y)) d \widetilde{\CM}^{\eps}_r(y) \Big\|_{2k}.\nonumber
\end{align}

We start with bounding $(\mathbf{I})$. Due to a priori bound of $\|\widetilde{\mathcal{Z}}^{\eps}_0(\bigcdot)\|_{2k}$ from \eqref{eq:NStDef1} and the heat kernel estimate of the second inequality in \eqref{eq:Est52}, $\sum_{y \in \ZZ} \mathfrak{p}^{\eps}_t(x-y)\widetilde{\mathcal{Z}}^{\eps}_{0}(y)$ is absolutely convergent for all $x\in \ZZ$. Rearranging (as is justified by the absolute convergence ) the sum yields
 \begin{equation}\label{eq:KerDifBd}
 \sum_{y\in \ZZ} \big( \mathfrak{p}^{\eps}_{t}(x_1-y) - \mathfrak{p}^{\eps}_{t}(x_2-y)\big)\widetilde{\mathcal{Z}}^{\eps}_0( y)  = \sum_{y\in \ZZ}\mathfrak{p}^{\eps}_{t}(x_1-y)\big(\widetilde{\mathcal{Z}}^{\eps}_0(y)- \widetilde{\mathcal{Z}}_0(x_2-x_1+ y)\big).
\end{equation}
%One can rearrange the sum because $\sum_{y \in \ZZ} \mathfrak{p}^{\eps}_t(x-y)\widetilde{\mathcal{Z}}^{\eps}_{0}(y)$
%(we can do it, because the sum is absolutely convergent due to the fast decay of the heat kernel in the spatial variable \eqref{eq:Est52} and the a priori bounds on the initial data \eqref{eqs:NStDef})

 Taking $L^{2k}$-norm on both sides of \eqref{eq:KerDifBd} and using subadditivity, we find that
 \begin{align}
 (\mathbf{I})  \leq \sum_{y\in \ZZ} \mathfrak{p}^{\eps}_{t}(x_1-y)\big\|\widetilde{\mathcal{Z}}^{\eps}_0( y)- \widetilde{\mathcal{Z}}^{\eps}_0(x_2-x_1+ y)\big\|_{2k}.\label{eq:T2Ineq}
\end{align}
 As $\mathcal{Z}^{\eps}_0$ satisfies \eqref{eq:NSteDef2}, we have
 \begin{align}
 \big\|\widetilde{\mathcal{Z}}^{\eps}_0(y)- \widetilde{\mathcal{Z}}^{\eps}_0(x_2-x_1+ y)\big\|_{2k}\leq (\eps|x_1-x_2|)^{2\beta} e^{\eps u(|x_1-x_2+y|+|y|)}.\label{eq:DiffBd}
 \end{align}
Using \eqref{eq:DiffBd} and the triangle inequality in \eqref{eq:T2Ineq} yields
 \begin{align}
 (\mathbf{I})\leq e^{\eps u |x_1-x_2|} \sum_{y\in \ZZ}\mathfrak{p}^{\eps}_{t}(x_1-y) e^{2\eps u|y|}. 
 \end{align}
Bounding the sum on the right hand side using the second inequality of \eqref{eq:Est52} with $\alpha=0$ shows the desired bound
\begin{align}
(\mathbf{I})  \leq C (\eps|x_1-x_2|)^{4\beta} e^{\eps u|x_1-x_2|}. \label{eq:T21stTermFB}
\end{align}

Now, we turn to $(\mathbf{II})$. Using \eqref{eq:Iteration2} yields
\begin{align}
(\mathbf{II})^2 \leq C (\eps|x_1- x_2|)^{4\beta} \sum_{y\in \ZZ}\int^{t}_{0} | \bar{\mathfrak{p}}^{\eps;\nabla}_{t-r}(x_1,x_2;y)|\big(\eps^2\|\widetilde{\mathcal{Z}}^{\eps}_r(y)\|^2_{2k} + \eps e^{2r\theta}\big) dr. \label{eq:T22ndTerm}
\end{align}
Applying \eqref{eq:SquareBdlemma} recursively to the r.h.s. we may, in a similar way as in the proof of Lemma~\ref{ChaosSeries}, develop an infinite series bound% applying \eqref{eq:Iteration1} recursively in the r.h.s. of \eqref{eq:T22ndTerm}, we get \note{this should be checked to make sure the chaos series developed here is correct given the issue with the previous one}
\begin{align}
(\mathbf{II})^2 \leq 2(\eps|x_1-x_2|)^{4\beta} &\sum_{\ell=1}^{\infty}  (C \eps^{2})^{\ell} \int\limits_{\vec{r}\in \Delta^{(\ell)}} \sum_{\vec{y}^{(\ell)}\in \ZZ^{\ell}}\mathcal{K}^{\eps; \ell;\nabla}_{\vec{r}}(\vec{y})\mathcal{D}^{\eps;k}_{r_1,0}(y_1)d\vec{r}
%\nonumber\\
%&\qquad +2(\eps|x_1-x_2|)^{4\beta}\sum_{\ell=1}^{\infty} (4 C)^{\ell}\eps^{2\ell-1} e^{2 \theta t} \int_{\vec{r}\in \Delta^{(\ell)}} \mathcal{K}^{(\ell);\nabla}_{\vec{r}}(\vec{y}) d\vec{r}.
  \label{eq:T22ndTerm2}
\end{align}
where $\mathcal{D}^{\eps;k}_{r_1,0}(y_1)$ is defined in Lemma~\ref{ChaosSeries} and
\begin{align}
\mathcal{K}^{\eps; \ell;\nabla}_{\vec{r}}(\vec{y}) := |\bar{\mathfrak{p}}^{\eps;\nabla}_{t-r_{\ell}}(x_1,x_2;y_{\ell})| \prod_{i=1}^{\ell-1} \bar{\mathfrak{p}}^{\eps}_{r_{i+1} -r_{i}}(y_{i+1}-y_{i}). \label{eq:CalK}
\end{align}
By use of the triangle inequality $|\bar{\mathfrak{p}}^{\eps;\nabla}_{ t-r_{\ell}}(x_1,x_2;y_{\ell})| \leq \bar{\mathfrak{p}}^{\eps}_{t-r_{\ell}}(x_1-y_{\ell}) + \bar{\mathfrak{p}}^{\eps}_{t-r_{\ell}}(x_2-y_{\ell})$ we can bound the series on the r.h.s. of \eqref{eq:T22ndTerm2} in the same manner as in the proof of  \eqref{eq:Tightness2}. This eventually produces the first inequality below (second inequality uses  $\sqrt{\eps^2 t}\leq \sqrt{T}$  and  $\eps^2\sqrt{t}\leq \eps \sqrt{T}$)
%
%After using this inequality to bound the r.h.s of \eqref{eq:CalK}, in a similar way as in \eqref{eq:SeriesTerm} (see also \eqref{eq:kTermBd}),\vspace{-0.2in}
%\begin{align}
%\int_{\vec{r}\in \Delta^{(\ell)}} \sum_{\vec{y}^{(\ell)}\in \ZZ^{\ell}}\mathcal{K}^{\ell;\nabla}_{\vec{r}}(\vec{y}) \sum_{y_0\in \ZZ}\mathfrak{p}^{\eps}_{r_1-r_0}(y_1-y_0)\|\widetilde{\mathcal{Z}}^{\eps}_0(y_0)\|^2_{2k}d\vec{r} \leq C \exp(2\eps u|x|)\frac{\Gamma(\frac{1}{2})^\ell}{\Gamma(\frac{\ell+1}{2})}t^{\frac{\ell}{2}},\nonumber
%\end{align}
%for some $C= C(k)$. Substituting this bound into the r.h.s. of \eqref{eq:T22ndTerm} and taking sum over all $\ell$ from $1$ to $\infty$ of the r.h.s., we arrive at
\begin{align}\label{eq:T22ndTermFB}
(\mathbf{II})^2\leq C\sqrt{\eps^2 t}(\eps|x_1-x_2|)^{4\beta} e^{2u\eps(|x_1|+|x_2|)} e^{C \eps^2\sqrt{t}} \leq  C e^{ C\eps \sqrt{T}} (\mathbf{I})^2.
\end{align}
This and \eqref{eq:T21stTermFB} imply $(\mathbf{I})+(\mathbf{II}) \leq C (\eps|x_1-x_2|)^{2\beta} e^{u\eps(|x_1|+|x_2|)}$, finishing the proof of \eqref{eq:Tightness2}.

\smallskip
\noindent {\bf Proof of \eqref{eq:Tightness3}.} Without loss of generality, we can assume $t_2>t_1$. Using \eqref{eq:ZWeakSol} and $L^{2k}$-norm triangle inequality yields
\begin{align}\label{eq:Tht31}
\big\|\widetilde{\mathcal{Z}}^{\eps}_{t_1}(x)  - \widetilde{\mathcal{Z}}^{\eps}_{t_2}(x)\big\|_{2k} \leq (\hat{\mathbf{I}})+(\hat{\mathbf{II}}),
\end{align}
where
\begin{align}
(\hat{\mathbf{I}}) := \Big\|\sum_{y\in \ZZ} \mathfrak{p}^{\eps}_{t_2-t_1}(x-y)(\widetilde{\mathcal{Z}}^{\eps}_{t_1}(y) -\widetilde{\mathcal{Z}}^{\eps}_{t_1}( x) )\Big\|_{2k},\quad
(\hat{\mathbf{II}}) := \Big\| \sum_{y\in \ZZ} \int^{t_2}_{t_1} \mathfrak{p}^{\eps}_{t_2-r}(x-y) d \widetilde{\CM}^{\eps}_r(y) \Big\|_{2k}.
\end{align}

First, we bound the term $(\hat{\mathbf{I}})$.
By the $L^{2k}$-norm triangle inequality $$(\hat{\mathbf{I}}) \leq  \sum_{ y\in \ZZ} \mathfrak{p}^{\eps}_{t_2-t_1}(x-y)\big\|\widetilde{\mathcal{Z}}^{\eps}_{t_1}(y)- \widetilde{\mathcal{Z}}^{\eps}_{t_1}(x)\big\|_{2k}.$$ Using \eqref{eq:Tightness2} and the second inequality of \eqref{eq:Est52}, with $\alpha=2\beta$, yields
\begin{align}
(\mathbf{I})\leq (1\vee |t_2-t_1|)^{\beta}\eps^{2\beta} e^{2\eps u|x|}.\label{eq:MyRightSide}
\end{align}

Now, we turn to $(\hat{\mathbf{II}})$. Applying \eqref{eq:Iteration1} and then recursively applying \eqref{eq:SquareBdlemma} to $(\hat{\mathbf{II}})$ yields (in the same way as in \eqref{eq:T22ndTerm2})
\begin{align}
(\hat{\mathbf{II}})^2 \leq 2\sum_{\ell=1}^{\infty} (C\eps^{2})^{\ell} \int\limits_{\vec{r}\in \Delta^{(\ell)}_{t_1,t_2}} \sum_{\vec{y}\in \ZZ^{\ell}}\mathcal{K}^{\eps; \ell}_{\vec{r}}(\vec{y}) \mathcal{D}^{\eps;k}_{r_1,t_1}(y_1) d\vec{r}
%\\& +2\sum_{\ell=1}^{\infty} e^{2 \theta t}(4 C)^{\ell} \eps^{2\ell-1}\int_{\vec{r}\in \Delta^{(\ell)}_{t_1,t_2}} \sum_{\vec{y}\in \ZZ^{\ell}}\mathcal{T}^{(\ell)}_{\vec{r}}(\vec{y}) d\vec{r},
\label{eq:T32ndTermB}
\end{align}
where $\vec{r}\in \Delta^{(\ell)}_{t_1,t_2}:=\{t_1\leq r_1\ldots \leq r_{\ell}\leq t_2\}$ and the functions are defined in Lemma~\ref{ChaosSeries}.
%\begin{align}
%\mathcal{T}^{(\ell)}_{\vec{r}}(\vec{y}) := \mathfrak{p}^{\eps}_{t_2-r_{\ell}} (x-y_{\ell}) \prod_{i=0}^{\ell-1} \bar{\mathfrak{p}}^{\eps}_{r_{i+1}-r_{i}} (y_{i+1}-y_{i}).
%\end{align}
%\note{The details of what follows can probably be left out as long as we have the correct conclusion. I don't, however, see where the $\beta$ stuff comes below for this term.}
In a similar way as used to derive  \eqref{eq:kTermBd}, and using \eqref{eq:Tightness1} to bound $\|\widetilde{\mathcal{Z}}^\eps_{t_1}(y_0)\|^2_{2k}$, we arrive at
\begin{align}
\int_{\vec{r}\in \Delta^{(\ell)}_{t_1,t_2}} \sum_{\vec{y}\in \ZZ^{\ell}}\mathcal{K}^{\eps; \ell}_{\vec{r}}(\vec{y})\mathcal{D}^{\eps;k}_{r_1,t_1}(y_1)\leq
\big(C e^{2\eps u|x|}  + \eps^{-1}  e^{2(t_2-t_1)\theta}\big)\frac{\Gamma(\frac{1}{2})^\ell}{\Gamma(\frac{\ell+1}{2})}(t_2-t_1)^{\frac{\ell}{2}}.\nonumber
\end{align}
Substituting above inequality into the r.h.s. of \eqref{eq:T32ndTermB} and summing over all $\ell\in \ZZ_{\geq 1}$ yields
\begin{align}\label{eq:T32ndTermFB}
(\hat{\mathbf{II}})^2\leq 	C (1\vee|t_2-t_1|)^{2\beta}\eps^{4\beta} e^{\eps^2|t_2-t_1|^{1 / 2}} e^{2u\eps|x|}.
\end{align}
In deriving this we used the bound $\eps (t_2-t_1)^{1/2}\leq C (1\vee|t_2-t_1|)^{2\beta}\eps^{4\beta}$, which is valid for $\beta\in (0,1/4)$.
Finally, substituting \eqref{eq:T32ndTermFB} and \eqref{eq:MyRightSide} into \eqref{eq:Tht31}, we arrive at \eqref{eq:Tightness3}.

 \section{Identification of the limit for solutions of the microscopic SHE}\label{MartingaleProb}

% \bp\label{ppn:Tight}
% In the settings of Theorem~\ref{MainTheorem}, $\{\mathcal{Z^{\eps}}\}_{\eps > 0}$ is tight with respect to the weak topology of the Skorokhod space $D([0,\infty), \mathcal{C}(\RR))$.
% \ep

The following proposition is the main result of this section

 \bp\label{ppn:MgProb}
 In the setting of Theorem~\ref{MainTheorem}, let $\mathcal{Z}^{\eps}_0\Rightarrow\mathcal{Z}_0$ in $\mathcal{C}(\RR)$ as $\eps\to 0$. Then, every convergent subsequence of $\{\mathcal{Z}^{\eps}\}_{\eps > 0}$ in $D([0,\infty), \mathcal{C}(\RR))$ has the same limit, which is the unique solution of \eqref{eq:SHEAdditive}, with $A=0$ and $B_t=\sqrt{2}e^{t/4}$, started from the initial state $\mathcal{Z}_0$.
 \ep

%\begin{proof}
%... \const{we should either refer to the lemmas below, or move this proof to the very end of the section}
%\end{proof}

% The moment bounds in Proposition~\ref{thm:Tightnes} and an argument similar to \cite[Theorem~3.3]{Bertini1997} (see also \cite[Proposition~1.4]{Dembo2016}), readily yield tightness of $\{\mathcal{Z}^{\eps}\}_{\eps > 0}$ with respect to the weak topology of $D([0,\infty), \mathcal{C}(\RR))$.

% We provide the proof of this proposition at the end of this section. In the proof of Theorem~\ref{MainTheorem}, it was shown that for any fixed $T>0$, $\{\mathcal{Z}^{\eps}\}_{\eps > 0}$ is a tight sequence of stochastic processes in the space $D([0,T], \mathcal{C}(\RR))$. Thus, for any subsequence $\mathcal{Z}^{\eps_n}$ there exists a further subsequence $\mathcal{Z}^{\eps_{n_k}}$ which has a limit in the weak topology of the Skorokhod space $D([0,T], \mathcal{C}(\RR))$. To show that $\mathcal{Z}^{\eps}$ converges to the unique solution of the SPDE \eqref{eq:SHEAdditive}, we need to show that the  limit of $\mathcal{Z}^{\eps_{n_k}}$ solves equation \eqref{eq:SHEAdditive}. In what follows,

We will use the following martingale problem to uniquely identify the limiting SPDE \eqref{eq:SHEAdditive}.
% associated with the SPDE \eqref{eq:SHEAdditive}. Eventually, we show that all the subsequential limit of $\mathcal{Z}^{\eps}$ solve the following martingale problem.

 \bd\label{MartingaleProblem}
 Consider a stochastic process $\mathcal{Z}$ in $D([0,\infty), \mathcal{C}(\RR))$ such that for any $t>0$, $u>0$ and $k\in \NN$, there exists $C=C(t,u,k)>0$ satisfying%\const{should be capital letters}
 \begin{align}
 \sup_{r\in [0,t]} \sup_{x\in \RR} e^{-u|x|} \|\mathcal{Z}_r(x)\|_{2k}\leq C.
 \end{align}
Let $C^{\infty}_{b}(\RR)$ be the set all infinitely differentiable bounded functions. Then, $\mathcal{Z}$ is the solution of the martingale problem for the SPDE \eqref{eq:SHEAdditive} with $A=0$ and $B_t= \sqrt{2}e^{ct}$ started from $\mathcal{Z}_0$, if for any $\phi\in \mathcal{C}^{\infty}_{b}(\RR)\cap L^2(\RR)$ the processes%\const{probably we should write capital letters for variables here}
 \begin{align}\label{eq:Mart}
 \mathfrak{M}_t(\phi):=\mathcal{Z}_{t}(\phi) - \mathcal{Z}_{0}(\phi) - 2\int^{t}_0\mathcal{Z}_r(\phi^{\prime\prime}) dr, \quad \mathfrak{N}_{t}(\phi) := \big(\mathfrak{M}_t(\phi)\big)^2 - \frac{e^{2ct}-1}{c} \| \phi \|_{L^2}^2 \quad
 \end{align}
are local martingales, where $\mathcal{Z}_t(\phi) := \int_{-\infty}^{\infty} \mathcal{Z}_t(y)\phi(y)dy$.
 \ed

The martingale problem uniquely identities the law of the solution to the  SPDE \eqref{eq:SHEAdditive}. Therefore, in order to show Proposition \ref{ppn:MgProb} we will demonstrate a microscopic martingale problem and show that on convergent subsequences, limit laws satisfy the martingale problem in Definition \ref{MartingaleProblem}.
%  For convenience, we introduce the following class of stochastic processes:\const{do we need this space?}
%  \begin{align}
%  \mathcal{U}_{T} &:= \left\{\{U_{\eps}\}_{\eps}\subset D([0,\infty, \mathcal{C}(\RR))\big| \sup\{\|U_{\eps}(s,y)\|_{2k}:y\in \RR, s \in (0, \eps^{-2}T), \eps\in (0,1] \}<\infty \right\}.
% % \mathcal{V}_{T,X} &:= \left\{V_{\eps}(s,\zeta)\big|  \sup\{\|U_{\eps}(s,y)\|_{2k}:y\in (-\eps^{-1}X, \eps^{-1}X), s \in (0, \eps^{-2}T), \eps\in (0,1] \}\to \infty\right\}.\nonumber
%  \end{align}
%  For any $\phi\in C^{\infty}_{b}(\RR)$ and $\psi\in \mathcal{C}(\RR)$, we define
%  \begin{align}
%  \langle \phi, \psi\rangle_{\eps} := \eps\sum_{y\in \ZZ} \phi(\eps y)\psi(\eps y).
%  \end{align}
   Fix $\phi \in C^{\infty}_{b}(\RR)\cap L^2(\RR)$ and a subsequence $\mathcal{Z}^{\eps_n}$ which weakly converges to a limit $\mathcal{Z}$ in $D([0,\infty), \mathcal{C}(\RR))$. For the convenience, in this section we will drop the subscript $n$ from $\eps_n$ (though at this point we have not ruled out different limits along different subsequences) and always assume that $\mathcal{Z}^{\eps}$ converges to a limit $\mathcal{Z}$. Recalling that $\mathcal{Z}^{\eps}_t(x) = \widetilde{\mathcal{Z}}^{\eps}_{\eps^{-2}t}(\eps^{-1}x)$, let us denote $\mathcal{Z}^\eps_t(\phi) := \llangle \mathcal{Z}^{\eps}_t, \phi\rrangle_{\eps}$, where the pairing is defined in Section \ref{sec:notation}. Define furthermore 
  \begin{subequations}
   \label{e:martingales}
  \begin{align}
 \mathfrak{M}^{\eps}_t(\phi) &:=\mathcal{Z}^\eps_t(\phi) - \mathcal{Z}^\eps_0(\phi)- \sqrt q\int^{t}_{0} (\Delta_{\eps} \mathcal{Z}^{\eps})_r(\phi) dr ,\label{eq:DefM}\\
 \mathfrak{N}^{\eps}_t(\phi) &:= \big(\mathfrak{M}^{\eps}_t(\phi)\big)^2 - \int^{\eps^{-2}t}_{0} \sum_{y\in \ZZ}(\eps\phi(\eps y))^2 d\big\langle \widetilde{\CM}^{\eps}(y), \widetilde{\CM}^{\eps}(y) \big\rangle_{r}, &&\label{eq:DefN}
\end{align}
\end{subequations}
where the discrete Laplacian $\Delta_{\eps}$ acts on the variable $x$. We could have defined $\mathcal{M}^{\eps}_t(x) = \widetilde{\mathcal{M}}^{\eps}_{\eps^{-2}t}(\eps^{-1}x)$ in which case the second term in \eqref{eq:DefN} would take a slightly more appealing form. However, we find it simpler to work with this more microscopic expression below.

   Owing to Lemma~\ref{thm:DSHETheo}, it is straightforward to see that $\mathfrak{M}^{\eps}_t(\phi)$ is a local martingale with respect to the natural filtration of $\{s_{\eps^{-2}t}\}_{t\geq 0}$. To see that $\mathfrak{N}^{\eps}_t(\phi)$ is also a local martingale, we note that the second term on the r.h.s. of \eqref{eq:DefN} is the bracket process $\langle \mathfrak{M}^{\eps}(\phi), \mathfrak{M}^{\eps}(\phi)\rangle_t$. %Now, to complete the proof of Proposition~\ref{ppn:MgProb}, we need the following lemma.

Since $\mathcal{Z}^{\eps}_t(\phi)$ converges weakly to $\mathcal{Z}_t(\phi)$, by Skorokhod's representation theorem \cite[p.70]{Billingsley2} we can embed these processes onto a common probability space on which they converge almost surely. In the following lemma we provide convergence of various terms from \eqref{e:martingales} with respect to this common probability space.

\bl\label{lem:IdLemma}
 For every fixed $t \geq 0$, we have the following limits
  \begin{subequations}
   \label{e:martingales_limits}
    \begin{align}
      \lim_{\eps \to 0} \mathbb{E}&\Big[\big| \mathcal{Z}^{\eps}_t(\phi) - \mathcal{Z}_t(\phi)\big|\Big] = 0,\label{eq:Req1}\\
    \lim_{\eps \to 0} \mathbb{E}&\left[\Big|2 \sqrt q \int^{t}_{0} (\Delta_{\eps} \mathcal{Z}^{\eps})_r(\phi) dr - 4\int^{t}_{0} \mathcal{Z}_r(\phi'') dr \Big|\right] = 0, \label{eq:Req2}\\
    \lim_{\eps \to 0} \mathbb{E}&\Bigg[\bigg|\int^{\eps^{-2}t}_{0} \sum_{y\in \ZZ}\big(\eps \phi(\eps y)\big)^2 d \big\langle \widetilde{\CM}^{\eps}(y), \widetilde{\CM}^{\eps}(y) \big\rangle_{r}- 4(e^{t/2}-1) \| \phi \|_{L^2}^2 \bigg|\Bigg] = 0. &\label{eq:Req3}
\end{align}
\end{subequations}
%\const{the last identity should be changed}
 \el

\begin{proof}
  % Since $\mathcal{Z}^{\eps}_T(\phi)$ converges weakly to $\mathcal{Z}_T(\phi)$, by combining Skorokhod's representation theorem \cite[p.70]{Billingsley2} (which shows that we can embed these processes onto a common probability space on which they converge almost surely)
   %\footnote{Skorokhod's representation theorem  shows that if a sequence of random variables $\{X^{\eps}\}_{\eps>0}$ converges weakly to $X$, then, there exists a probability space containing $\{\tilde{X}^{\eps}\}_{\eps>0}$ and $\tilde{X}$ such that $\tilde{X}^{\eps}\stackrel{d}{=}X^{\eps}$ for all $\eps>0$, $\tilde{X}\stackrel{d}{=} X$ and $\tilde{X}^{\eps}\to X$ as $\eps\to 0$ a.s. See \cite[p.70]{Billingsley2} for details}
   By Fatou's lemma and the almost sure convergence of $\mathcal{Z}^{\eps}_t(\phi)$ to $\mathcal{Z}_t(\phi)$ (by Skorokhod)
  \begin{align}\label{eq:Fatou}
  \mathbb{E}\big[|\mathcal{Z}_t(\phi)|\big]\leq \liminf_{\eps \to 0}\mathbb{E}\big[|\mathcal{Z}^{\eps}_t(\phi)|\big].
\end{align}
  Owing to \eqref{eq:Tightness1} and the decay of $\phi(x)$ as $x\to \infty$ (since $\phi\in C^{\infty}_b(\RR)\cap L^2(\RR)$), %\note{does this mean functions of bounded support? This needs to be said somewhere since this is not, in fact, standard notation}
  the $L^{1}$-norm of $\{\mathcal{Z}^{\eps}_t(\phi)\}_{\eps\geq 0}$ is uniformly bounded as $\eps\to 0$, and thus $\mathbb{E}\big[|\mathcal{Z}_t(\phi)|\big]<\infty$. Since $\{\mathcal{Z}^{\eps}_t(\phi) - \mathcal{Z}_t(\phi)\}_{\eps}$ converges almost surely to $0$ and is uniformly bounded in $L^{1}$-norm as $\eps\to 0$, we obtain, by dominated convergence, \eqref{eq:Req1}.
  %  \[\lim_{\eps \to 0} \mathbb{E}\Big[\big| \mathcal{Z}^{\eps}_T(\phi) - \mathcal{Z}_T(\phi)\big|\Big] = 0.\]

%  The proof of \eqref{eq:Req1} follows by combining the continuity and the decay of $\phi$ with the pointwise convergence (in distribution and hence, in $L^{2k}$ thanks to \eqref{eq:Tightness2}) of $\mathcal{Z}^{\eps}(T,\bigcdot)$ to $\mathcal{Z}(T,\bigcdot)$.

Turning to \eqref{eq:Req2}, observe that $\sqrt q \to 1$ as $\eps\to 0$ and that, via
%Observe that
%     \begin{align}\label{eq:inner1}
%   (\Delta_{\eps} \mathcal{Z}^{\eps})_\phi(Q) = \eps^{-2}\sum_{x\in \ZZ}\eps \Big(\mathcal{Z}^{\eps}(Q,\eps x+\eps)+\mathcal{Z}^{\eps}(Q,\eps x-\eps)- 2\mathcal{Z}^{\eps}(Q,\eps x)\Big)  \phi(\eps x) . &&
%     \end{align}
summation by parts,
\begin{align}
(\Delta_{\eps} \mathcal{Z}^{\eps})_r(\phi) &= \eps\sum_{x\in \ZZ} \mathcal{Z}^{\eps}_r(\eps x) \Delta_{\eps} \phi(\eps x) = 2 \mathcal{Z}^{\eps}_r(\phi'') + \eps \llangle  \mathcal{Z}^{\eps}_r, \psi^\eps \rrangle_{\eps}, \label{eq:Taylor}
\end{align}
where $\psi^\eps(x) := \Delta_{\eps} \phi(\eps x) - \phi''(x)$. Theorem~\ref{thm:Tightnes} shows that $\mathbb{E}[\| \mathcal{Z}^{\eps}_r\|_{2k}]$ is uniformly bounded in $\eps$ and $r \in [0,t]$. This along with the $\eps$ prefactor implies that the last term in \eqref{eq:Taylor} vanishes in distribution. Since $\mathcal{Z}^{\eps}$ converges weakly to $\mathcal{Z}$ in $D([0,t), \mathcal{C}(\RR))$, we get from \eqref{eq:Taylor} the convergence in distribution
 \begin{align}\label{eq:DConv}
 2 \sqrt q \int^{t}_{0} (\Delta_{\eps} \mathcal{Z}^{\eps})_r(\phi) dr \quad \Rightarrow\quad 4\int^{t}_{0} \mathcal{Z}_r(\phi'') dr.
 %\int^{T}_0\eps^{-2}\llangle \Delta_{\eps} \mathcal{Z}^{\eps}_R, \phi \rrangle_{\eps}  dR \quad \stackrel{d}{\to} \quad 2\int^{T}_0 \mathcal{Z}_R(\phi'') dR, \qquad \text{  as  } \eps\to 0.
  \end{align}
 Due to \eqref{eq:Tightness1}, rapid decay of $\phi$ and uniform bounds on the $L^{2k}$-norm of $ \mathcal{Z}^{\eps}$, the l.h.s. of \eqref{eq:DConv} is uniformly bounded in $L^1$ as $\eps \to 0$. Combining this with the Skorokhod almost sure convergence representation and the dominated convergence theorem, as in proving \eqref{eq:Req1}, yields  \eqref{eq:Req2}.

%  Note that $\sum_{y\in \ZZ}\eps yp^{\mathrm{sym}, \eps}_{t-s}(y)=0$ because $p^{\mathrm{sym}, \eps}$ is the transition kernel of a simple symmetric random walk. Thus, the second term on the r.h.s. of \eqref{eq:Taylor} vanishes. Furthermore, one has
% \begin{align}\label{eq:Variance}
% p^{\mathrm{sym}, \eps}_{t-s}(y)(\eps y)^2 = \eps^2 (t-s).
% \end{align}
% Plugging \eqref{eq:Variance} into \eqref{eq:Taylor} and integrating, we get \eqref{eq:Req2}.

The proof of \eqref{eq:Req3} is the most involved and ultimately relies on some self-averaging, whose idea goes back to \cite{Bertini1997}. Applying the change of variable $\eps^2 r \mapsto r$ inside the integral of \eqref{eq:Req3} and then, computing the bracket process (inside that integral) by applying \eqref{eq:1} and \eqref{eq:Approx2}, we arrive at %$\widetilde{\mathcal{Z}}^{\eps}(t,x) = \eps^{-\frac{1}{2}}\widetilde{\mathcal{Z}}^{\eps}_{q, \alpha}(x, s_t)$,
 %\const{the notation is bad} we have
 \begin{align}\label{eq:quadMMs}
 &\int^{\eps^{-2}t}_{0} \sum_{y\in \ZZ}(\eps \phi(\eps y))^2 d \big\langle \widetilde{\CM}^{\eps}\big(y\big), \widetilde{\CM}^{\eps}(y) \big\rangle_{r} \\ 
 &\qquad = \int^{t}_{0} \eps \sum_{y\in \ZZ} \phi(\eps y)^2  \left(2 e^{2\eps^{-2}r\theta} - 2 \eps^{-1}\nabla^{+}\widetilde{\mathcal{Z}}^{\eps}_{\eps^{-2}r}(y)\nabla^{-}\widetilde{\mathcal{Z}}^{\eps}_{\eps^{-2}r}(y) + U^{\eps}_{\eps^{-2}r}(y)\right) dr,\nonumber
 \end{align}
 %\note{are the various powers of $\eps$ correct above? It doesn't seem to jive with what comes below}
 %where \note{what is $\xi_{\eps}$ below?}
 where $\widetilde{\mathcal{Z}}^{\eps}$ is defined in \eqref{eq:ModDuality} and 
 \begin{align}
 U^{\eps}_r(x) := \frac{1}{2} e^{2r\theta}&\big((1+q^{-s_r(x)})(1+q^{\frac{s_r(x+1)+s_r(x-1)}{2}})-4\big) \label{eq:Ueps}\\
 &\qquad \times \big(1 - \nabla^{+}s_r(x)\nabla^{-}s_r(x) + \mathcal{E}^{\eps}_r(x)\big) + 2 e^{2r\theta} \bigl( \mathcal{E}^{\eps}_r(x) + \mathfrak{B}^{\eps}_r(x) \bigr). 
 \end{align}
 Here, we use the functions $\mathcal{E}^{\eps}$ and $\mathfrak{B}^{\eps}$ which are defined in \eqref{eq:1} and \eqref{eq:Approx2} respectively.

 %where $|\tilde{\mathcal{E}}^{\eps}_r(x)| \leq C\eps$ and $|\tilde{\mathfrak{B}}^{\eps}_r(x)|\leq C \eps (e^{-2 \theta r} \widetilde{\mathcal{Z}}^{\eps}_r(x)^2 + 1)$, uniformly in $r\in [0,\eps^{-2}t]$ and $x\in \ZZ$.

There are three terms inside the parenthesis on the r.h.s. of \eqref{eq:quadMMs}. We will address the third term, then the first and then the (much harder) second term. Starting with the third term $U^{\eps}$, a direct computation shows that
 \begin{equation}
 q^{\mp s_r(x)/2} =  \sqrt{1+\frac{1}{4}\eps e^{-2r\theta}(\widetilde{\mathcal{Z}}^{\eps}_r(x))^2} \pm \frac{1}{2}\eps^{\frac{1}{2}} e^{-r\theta}\widetilde{\mathcal{Z}}^{\eps}_r(x).\label{eq:qToZ}
 \end{equation}
 Moreover, the simple bound $\sqrt{1 + x^2} \leq 1 + |x|$ yields 
  \begin{equation}
\left| q^{\mp s_r(x)/2} - 1 \right| \leq \eps^{\frac{1}{2}} e^{-r\theta}\bigl|\widetilde{\mathcal{Z}}^{\eps}_r(x)\bigr|.
 \end{equation}
 Combining this with the fact that $|\nabla^{\pm}s_r(x)|=1$ and with the upper bounds on $\mathcal{E}^{\eps}$ and $\mathfrak{B}^{\eps}$, provided in Propositions~\ref{thm:DSHETheo} and \ref{DSHEcompute} respectively, yields
 $$|U^{\eps}_r(x)| \leq \eps^{\frac{1}{2}} F(r,x),$$
where $F(r,x)$ is a polynomial of third degree in $|\widetilde{\mathcal{Z}}^{\eps}_r(x-1)|$, $|\widetilde{\mathcal{Z}}^{\eps}_r(x)|$ and $|\widetilde{\mathcal{Z}}^{\eps}_r(x+1)|$, whose coefficients are bounded uniformly in $r\in [0,\eps^{-2}t]$ and $x\in \ZZ$.
 From this and the growth bound on $\widetilde{\mathcal{Z}}^{\eps}$  in \eqref{eq:Tightness1}, it follows that for any $k\geq 1$, the expression
 $\sup_{x\in \ZZ}e^{-2u\eps|x|}\|U^{\eps}_{r}(x)\|_{2k}$ is uniformly bounded over all $r\in [0,\eps^{-2} t]$.
 By this uniform boundedness of $U^{\eps}$, the fact that $\phi$ has bounded support and the dominated convergence theorem,
$\int^{t}_{0} \eps \sum_{y\in \ZZ} \phi(\eps y)^2 U^{\eps}_{\eps^{-2}r}(y)dr$ vanishes as $\eps\to 0$.

Turning to the first term in  \eqref{eq:quadMMs}, due to the continuity of $\phi$, $\eps\sum_{y\in \ZZ}\phi(\eps y)^2\to \|\phi\|^2_{L^2}$ as $\eps\to 0$. Combining this with $2\eps^{-2} r\theta\to r/2$ and the fact that $\int_0^t e^{r/2}dr = 2(e^{t/2}-1)$ yields
\[\int^{t}_0\eps\sum_{y\in \ZZ}\phi(\eps y)^2 2 e^{2\eps^{-2}r\theta} dr \quad \to \quad 4(e^{t/2}-1) \| \phi \|_{L^2}^2, \]
as $\eps\to 0$. This limiting term is precisely what is subtracted off in \eqref{eq:Req3}.
Therefore, to complete the proof of \eqref{eq:Req3} we must show $ \lim_{\eps \to 0} \mathbb{E}[(\mathcal{R}^{\eps})^2] = 0$, where 
 \begin{align}\label{eq:Remain}
  \mathcal{R}^{\eps} :=  \Big|\int^{t}_0\sum_{y\in \ZZ} \phi(\eps y)^2\nabla^{+}\widetilde{\mathcal{Z}}^{\eps}_{\eps^{-2}r}(y)\nabla^{-}\widetilde{\mathcal{Z}}^{\eps}_{\eps^{-2}r}(y) dr \Big|.
 \end{align}
The rest of this proof is devoted to showing this.

%\note{I should just specify how all constants depend above; I can probably say early on the $C$ will always mean a constant which may depend on $T, u,\beta, k$ but not on any other parameters (check this is true)}

Expanding $(\mathcal{R}^{\eps})^2$ in terms of a double-integral in $r$ and $r'$ time variables, and introducing a conditional expectation with respect to the natural filtration $\mathcal{F}_{\eps^{-2}r'}$ up to time $\eps^{-2}r'$, yields
 \begin{align}
 \mathbb{E}[(\mathcal{R}^{\eps})^2] &= 2\mathbb{E}\bigg[	\int^{t}_{0}  \sum_{y_1\in \ZZ} \phi(\eps y_1)^2   \nabla^{+}\widetilde{\mathcal{Z}}^{\eps}_{\eps^{-2}r'}(y_1)\nabla^{-}\widetilde{\mathcal{Z}}^{\eps}_{\eps^{-2}r'}(y_1)\nonumber\\
 &\quad\qquad \times  \int^{t}_{r'}\mathbb{E}\Big[\sum_{y_2\in \ZZ} \phi(\eps y_2)^2\nabla^{+}\widetilde{\mathcal{Z}}^{\eps}_{\eps^{-2}r}(y_2)\nabla^{-}\widetilde{\mathcal{Z}}^{\eps}_{\eps^{-2}r}(y_2)\Big|\mathcal{F}_{\eps^{-2}r'}\Big] dr dr'\bigg].\nonumber
 \end{align}
Owing to this expression along with the rapid decay of $\phi$ and uniform bounds on the norms of $\widetilde{\mathcal{Z}}^{\eps}_{\eps^{-2}r'}$ for all $r'\in [0,t]$ as $\eps\to 0$, we can establish the bound
\begin{align}
\mathbb{E}[(\mathcal{R}^{\eps})^2] &\leq 2 \sum_{y_1,y_2\in \ZZ}(\phi(\eps y_1)\phi(\eps y_2))^2 \int^{t}_0 \Big\|\nabla^{+}\widetilde{\mathcal{Z}}^{\eps}_{\eps^{-2}r'}(y_1)\nabla^{-}\widetilde{\mathcal{Z}}^{\eps}_{\eps^{-2}r'}(y_1)\Big\|_{2}\\
&\qquad \times \int^{t}_{r'}\Big\|\mathbb{E}\Big[ \nabla^{+}\widetilde{\mathcal{Z}}^{\eps}_{\eps^{-2}r}(y_2)\nabla^{-}\widetilde{\mathcal{Z}}^{\eps}_{\eps^{-2}r}(y_2)\Big| \mathcal{F}_{\eps^{-2}r'}\Big]\Big\|_{2} dr dr',\label{eq:Rsquare}
\end{align}
 by interchanging the summation and expectation, as well as the expectation and integral and then applying the Cauchy-Schwarz inequality. For the first integral in \eqref{eq:Rsquare}, from \eqref{eq:Tightness1}, $q= e^{-\eps}$ and $|\nabla^{\pm} s_{r'}(x)|=1$, we obtain
\begin{align}\label{eq:KeyRes2}
\sup_{x\in \ZZ} e^{-2u\eps|x|}\big\|\nabla^{+}\widetilde{\mathcal{Z}}^{\eps}_{\eps^{-2}r'}(x)\nabla^{-}\widetilde{\mathcal{Z}}^{\eps}_{\eps^{-2}r'}(x)\big\|_{2}\leq C\eps,
\end{align}
for all $r'\in [0,t]$ where $C=C(t, u)$. For the second integral in \eqref{eq:Rsquare} we have
%, Lemma~\ref{lem:KeyEstimate} (which we state and prove below) implies
% \begin{align}\label{eq:KeyRes1}
%  \sup_{x\in \ZZ} e^{-2u\eps|x|}\Big\|\mathbb{E}\Big[ \nabla^{+}\widetilde{\mathcal{Z}}^{\eps}_{\eps^{-2}R}(x)\nabla^{-}\widetilde{\mathcal{Z}}^{\eps}_{\eps^{-2}R}(x)\Big| \mathcal{F}_{\eps^{-2}Q}\Big]\Big\|_{2k}\leq C\frac{\eps^{2}}{(R-Q)^{\frac{3}{2}}},
% \end{align}
% for any $\sqrt{\eps}<Q<R<T$ where $C=C(k,T,u)>0$.  Using this, we see that
\begin{align}
\int^{t}_{r'} \Big\|\mathbb{E}\Big[ \nabla^{+}&\widetilde{\mathcal{Z}}^{\eps}_{\eps^{-2}r}(x)\nabla^{-}\widetilde{\mathcal{Z}}^{\eps}_{\eps^{-2}r}(x)\Big| \mathcal{F}_{\eps^{-2}r'}\Big]\Big\|_{2} dr \nonumber\\
&\leq \int^{r'+\sqrt{\eps}}_{r'}\Big\|\mathbb{E}\Big[ \nabla^{+}\widetilde{\mathcal{Z}}^{\eps}_{\eps^{-2}r}(x)\nabla^{-}\widetilde{\mathcal{Z}}^{\eps}_{\eps^{-2}r}(x)\Big| \mathcal{F}_{\eps^{-2}r'}\Big]\Big\|_{2} dr\\
&\qquad +C \int^{t}_{r'+\sqrt{\eps}}e^{2(u+1)\eps|x|}\big(\eps^{2} (r-r')^{-\frac{3}{2}} +\eps^{\frac{3}{2}}\big) dr\nonumber\\
&\leq C\eps^{\frac{3}{2}}e^{2u\eps|x|} + C\eps^{\frac{5}{4}}e^{2(u+1)\eps|x|}.\label{eq:FinalIneq}
\end{align}
The first inequality comes from splitting the integral into two parts ($r'$ to $r'+\sqrt{\eps}$ and $r'+\sqrt{\eps}$ to $t$) and applying the bound of Lemma \ref{lem:KeyEstimate} into the second part by taking $a=\eps^{-2}r$ and $b=\eps^{-2}r'$. The second inequality come from applying Jensen's inequality 
\begin{align}
\Big\|\mathbb{E}\Big[ \nabla^{+}\widetilde{\mathcal{Z}}^{\eps}_{\eps^{-2}r}(x)\nabla^{-}\widetilde{\mathcal{Z}}^{\eps}_{\eps^{-2}r}(x)\Big| \mathcal{F}_{\eps^{-2}r'}\Big]\Big\|_{2}\leq \big\|\nabla^{+}\widetilde{\mathcal{Z}}^{\eps}_{\eps^{-2}r'}(x)\nabla^{-}\widetilde{\mathcal{Z}}^{\eps}_{\eps^{-2}r'}(x)\big\|_{2}, \nonumber
\end{align}
and then using \eqref{eq:KeyRes2}.
Combining \eqref{eq:FinalIneq} and \eqref{eq:KeyRes2} to control the r.h.s. of \eqref{eq:Rsquare} yields
\begin{align}
\mathbb{E}[(\mathcal{R}^{\eps})^2]\leq C\sum_{y_1,y_2\in \ZZ} (\phi(\eps y_1)\phi(\eps y_2))^2e^{2(u+1)\eps(|y_1|+|y_2|)} \Big(\eps^{\frac{5}{2}} + \eps^{ \frac{13}{4}}\Big).\nonumber
\end{align}
Due to the rapid decay of $\phi$, the r.h.s. of the above inequality vanishes as $\eps\to 0$, which completes the proof of \eqref{eq:Req3}.
\end{proof}

%  Applying \eqref{eq:KeyRes} into the r.h.s. of \eqref{eq:Rsquare} along with \eqref{eq:Tightness2} (to bound $\widetilde{\mathcal{Z}}^{\eps}(\eps^{-2}Q,y_1)\nabla^{-}\widetilde{\mathcal{Z}}^{\eps}(\eps^{-2}Q,y_1)$) and owing to the \eqref{eq:Est5}, we get the following after summing and integrating the r.h.s. of \eqref{eq:Rsquare}
% \begin{align}
% \mathbb{E}\big[(\mathcal{R}^{\eps})^2\big]\leq C\eps\to 0 \quad \text{as }\eps \to 0.
% \end{align}

%\note{we are everywhere assuming  $\mathcal{Z}^{\eps}$ started from the near stationary initial data with parameters $u$ and $\beta$ (see Definition~\ref{NStDef}). We should make sure this is clear early on, and then we don't have to repeat it}

 \bl\label{lem:KeyEstimate}
There exists a constant $C=C(k, u,T)>0$ such that
 \begin{align}
  \sup_{x\in \ZZ} e^{-2(u+1)\eps|x|}\Big\|\mathbb{E}\Big[ \nabla^{+}\widetilde{\mathcal{Z}}^{\eps}_{a}(x)\nabla^{-}\widetilde{\mathcal{Z}}^{\eps}_{a}(x)\Big| \mathcal{F}_{b}\Big]\Big\|_{2k}\leq C\Big(\frac{1}{\eps(a-b)^{3 / 2}} + \eps^{\frac{3}{2}}\Big),\label{eq:KeyREstimate}
 \end{align}
 for any $k\in \NN$ and $0<b<b+\eps^{-\frac{3}{2}}\leq a<\eps^{-2}T$,
 %\begin{align}\label{eq:KeyEstimate}
% \sup_{X\in \RR} e^{-2u|X|} \Big\|\mathbb{E}\Big[\nabla^{+}_{\eps} \mathcal{Z}^{\eps} (R,X)\nabla^{-}_{\eps} \mathcal{Z}^{\eps}(R,X)\Big| \mathcal{F}^{\eps}_{Q}\Big]\Big\|_{2k}\leq  \frac{C}{(R-Q)^{\frac{3}{2}}},
% \end{align}
% for all $\sqrt{\eps}<Q<R<T$
 where $\mathcal{F}_{b}$ is the $\sigma$-algebra generated by $\{\vec{s}_{t}\}_{t\in [0,b]}$.
 \el

 \begin{proof}
Using \eqref{eq:ZWeakSol}, we can write
 \begin{align}
 \widetilde{\mathcal{Z}}^{\eps}_{a}(x) = \sum_{y\in \ZZ} \mathfrak{p}^{\eps}_{a-b}(x-y) \widetilde{\mathcal{Z}}^{\eps}_{b}(y) + \int^{a}_{b} \sum_{y\in \ZZ}\mathfrak{p}^{\eps}_{a-c}(x-y)d\widetilde{\mathcal{M}}^{\eps}_{c}(y). &&& \label{eq:MildSol}
 \end{align}
It follows from \eqref{eq:MildSol} that $\mathbb{E}[\nabla^{+} \widetilde{\mathcal{Z}}^{\eps}_{a} (x)\nabla^{-} \widetilde{\mathcal{Z}}^{\eps}_{a}(x)| \mathcal{F}_{b}] = (\mathbf{I}) + (\mathbf{II})$,
 where
  \begin{align}
  (\mathbf{I}) & := \sum_{y_1,y_2\in \ZZ}\nabla^{+}\mathfrak{p}^{\eps}_{a-b}(x-y_1) \nabla^{-} \mathfrak{p}^{\eps}_{a-b}(x-y_2)\widetilde{\mathcal{Z}}^{\eps}_{b}(y_1)\widetilde{\mathcal{Z}}^{\eps}_{b}(y_2),\\
  (\mathbf{II}) & := \mathbb{E}\Big[\int^{a}_{b} \sum_{y\in \ZZ}\nabla^{+}\mathfrak{p}^{\eps}_{a-c}(x-y)\nabla^{-}\mathfrak{p}^{\eps}_{a-c}(x-y) d\big\langle \widetilde{\CM}^{\eps}(y), \widetilde{\CM}^{\eps}(y)\big\rangle_{c}\Big| \mathcal{F}_{b}\Big] .
  \end{align}
For any $c\in [b,a]$, $k\in \NN$ and $\kappa \in \RR_{\geq 0}$, define
\begin{align}
f^{(k)}_{\kappa}(c;b) := \sup_{x\in \ZZ}e^{-2\kappa\eps|x|} &\left\|\mathbb{E}\Big[\nabla^{+}\widetilde{\mathcal{Z}}^{\eps}_c(x)\nabla^{-}\widetilde{\mathcal{Z}}^{\eps}_c(x)\Big|\mathcal{F}_{b}\Big]\right\|_{2k}. \nonumber\\
\alpha(c):= \eps(1\wedge (c-b)^{-\frac{1}{2}})+\eps^{\frac{3}{2}}, \quad \beta_1(c):=&  (1\wedge (c-b)^{-\frac{3}{2}}), \quad \beta_2(c):=  (1\wedge (a-c)^{-\frac{3}{2}}).
\end{align}
  In what follows, we will show that
\begin{align}\label{eq:BdI}
\sup_{x\in \ZZ}e^{-2(u+1)\eps|x|}\|(\mathbf{I})\|_{2k}\leq C_1\eps^{-1}\beta_1(a)
\end{align}
  and
  \begin{align}\label{eq:BdII}
 \sup_{x\in \ZZ} e^{-2(u+1)\eps|x|}\|(\mathbf{II})\|_{2k}\leq C_2 \alpha(a) +  C_3\int^{a}_{b} \beta_2(c) f^{(k)}_{u+1}(c;b) dc, &&
\end{align}
 for some $C_1= C_1(k,u,T)>0$, $C_2= C_2(k,u,T)>0$ and $C_3= C_3(k,u, T)>0$. Assuming \eqref{eq:BdI} and \eqref{eq:BdII}, we first complete the proof of \eqref{eq:KeyREstimate}. Since $f^{(k)}_{u+1}(a;b)$ is less than the sum of l.h.s. of \eqref{eq:BdI} and \eqref{eq:BdII}, summing both sides of the inequalities in \eqref{eq:BdI} and \eqref{eq:BdII} yields
 \begin{align}
f^{(k)}_{u+1}(a;b)\leq & C_1\eps^{-1}\beta_1(a) +C_2 \alpha(a)+ C_3\int^{a}_{b} \beta_2(c) f^{(k)}_{u+1}(c;b)dc. &&\label{eq:MyIteration}
\end{align}
 Note that \eqref{eq:MyIteration} verifies the condition of the Gronwall's inequality\footnote{Gronwall's inequality says that for any interval $I$ of the form $[b,\infty)$, or $[b,a]$, or $[b,a)$ with $b<a$ and any real valued functions $f, g$ and $h$ with the negative part of $f$ being integrable on every closed and bounded subinterval $I$, if $h$ satisfies $h(c)\leq f(c) + \int^{c}_{b}g (r) h(r) dr$ for all $c\in I$, then, one has $h(a)\leq f(a) + \int^{a}_{b}f(r)g (r) \exp(\int^{r}_{b} g(w)dw) dr$.} inside the interval $[a,b]$. Applying Gronwall's inequality, we write
\begin{align}\label{eq:Gron}
f^{(k)}_{u+1}(a;b)&\leq \Pi_{\alpha, \beta_1}(a)+C_3\int^{a}_{b}\Pi_{\alpha, \beta_1}(c)\beta_2 (c) \exp\Big(C_3\int^{c}_{b} \beta_2(w)dw\Big) dc,
\end{align}
 where $\Pi_{\alpha, \beta_1}(c)$ is the shorthand for $C_1\eps^{-1}\beta_1(c) + C_2\alpha(c)$. Since $\eps^{-1}(a-b)^{-\frac{3}{2}}\geq T^{-1}\eps(a-b)^{-\frac{1}{2}}$ for all $0<b<a<\eps^{-2}T$,  $\alpha(a)$ is less than $(1\wedge \eps^{-1}(a-b)^{-\frac{3}{2}})+ \eps^{\frac{3}{2}}$. Therefore, $\Pi_{\alpha, \beta_1}(a)$ is bounded above by $C^{\prime}_1(\eps^{-1}(a-b)^{-\frac{3}{2}}+ \eps^{\frac{3}{2}})$ for some $C^{\prime}_1= C^{\prime}_1(T)>0$. By a direct computation, $\int^{c}_{b}\beta_2(w)dw\leq \big(1-(c-b)^{-\frac{1}{2}}\big)\leq 1$.  Substituting the upper bounds of $\Pi_{\alpha, \beta_1}(a)$ and $\exp(\int^{c}_{b}\beta_2(w)dw)$ into the r.h.s. of \eqref{eq:Gron}, we get
 \begin{align}\label{eq:Gron2}
 f^{(k)}_{u+1}(a;b)\leq C^{\prime}_1(\eps^{-1}(a-b)^{-\frac{3}{2}} + \eps^{\frac{3}{2}}) + C^{\prime}_2 &\int^{a}_{b} \Pi_{\alpha, \beta_1}(c)\beta_2(c) dc,
\end{align}
 for some constant $C^{\prime}_2= C^{\prime}_2(T)>0$. Note that owing to \eqref{eq:Gron2}, \eqref{eq:KeyREstimate} will be proved once we show the following
 \begin{align}\label{eq:Last}
 \int^{a}_b \Pi_{\alpha, \beta_1}(c)\beta_2(c) dc \leq C\Big(\eps^{-1} (a-b)^{-\frac{3}{2}} + \eps^{\frac{3}{2}}\Big).
 \end{align}
 Thus we now proof \eqref{eq:Last}. Since $\max\{(a-c),(c-b)\}\geq (a-b)/2$ for any $c\in [a,b]$, we have
 \begin{align}
 \big(1\wedge (c-b)^{-\frac{3}{2}}\big)\big(1\wedge (a-c)^{-\frac{3}{2}}\big) \leq \big(1\wedge 2^{\frac{3}{2}} (a-b)^{-\frac{3}{2}}\big)\Big(\big(1\wedge (a-c)^{-\frac{3}{2}}\big)+ \big(1\wedge (c-b)^{-\frac{3}{2}}\big)\Big).
\end{align}
Integrating both sides of the above inequality w.r.t. $c$ yields
\begin{align}
\int^{a}_{b}\big(1\wedge (c-b)^{-\frac{3}{2}}\big)\big(1\wedge (a-c)^{-\frac{3}{2}}\big)dc\leq C (a-b)^{-\frac{3}{2}},  \label{eq:TwoWedge1}
\end{align}
for some constant $C=C(T)>0$. By using the fact that $(c-b)\leq \eps^{-2}T$ for ant $c\in [a,b]$, we have $\eps(1\wedge (c-b)^{-\frac{1}{2}}) \leq T \eps^{-1}(1\wedge (c-b)^{-\frac{3}{2}})$. Therefore, in a same way as in \eqref{eq:TwoWedge1}, we get
\begin{align}
\int^{a}_{b}\eps\big(1\wedge (c-b)^{-\frac{1}{2}}\big)\big(1\wedge (a-c)^{-\frac{3}{2}}\big)dc\leq C (a-b)^{-\frac{3}{2}}.  \label{eq:TwoWedge2}
\end{align}
Combining \eqref{eq:TwoWedge1} and \eqref{eq:TwoWedge2} with the fact that $\int^{a}_{b}(1\wedge (a-c)^{-3/2})dc\leq 1$ shows \eqref{eq:Last}.

 % Using \eqref{eq:MyIteration} repeatedly will lead to a converging series which after summing shows \eqref{eq:KeyREstimate}.

  To complete the proof of this lemma, it boils down to proving \eqref{eq:BdI} and \eqref{eq:BdII} which do as follows.
 We first show \eqref{eq:BdI}. Observe that \begin{align}
 \|(\mathbf{I})\|_{2k} &\leq \sum_{y_1, y_2\in \ZZ} |\nabla^{+} \mathfrak{p}^{\eps}_{a-b}(x-y_1)\nabla^{-} \mathfrak{p}^{\eps}_{a-b}(x-y_2) |\|\widetilde{\mathcal{Z}}^{\eps}_{b}(y_1)\|_{4k}\|\widetilde{\mathcal{Z}}^{\eps}_{b}(y_2)\|_{4k}\\
 &\leq C\Big(\sum_{y\in \ZZ}\nabla^{+} \mathfrak{p}^{\eps}_{a-b}(x-y)e^{u\eps(|y|+1)}\Big)^2
\\& \leq  C \sum_{y\in \ZZ}(\nabla^{+} \mathfrak{p}^{\eps}_{a-b}(x-y))^2 e^{2(u+1)\eps|y|}\sum_{y\in \ZZ} e^{-2\eps |y|}\\
&\leq Ce^{2(u+1)\eps|x|}\eps^{-1} (1\wedge (a-b)^{-\frac{3}{2}}),
 \label{eq:IFormBound}
 \end{align}
 for some $C=C(k,u,T)>0$.
  The inequality in the first line follows from the triangle inequality of the $L^{2k}$-norm and H\"older's inequality (which bounds $\|\widetilde{\mathcal{Z}}^{\eps}_{b}(y_1)\widetilde{\mathcal{Z}}^{\eps}_{b}(y_2)\|_{2k}$ by the product of $\|\widetilde{\mathcal{Z}}^{\eps}_{b}(y_1)\|_{4k}$ and $\|\widetilde{\mathcal{Z}}^{\eps}_{b}(y_2)\|_{4k}$). The  inequality in the second line is obtained by substituting the upper bound of $\|\widetilde{\mathcal{Z}}^{\eps}_{b}( \bigcdot)\|$ from \eqref{eq:Tightness1}. Applying Cauchy-Schwarz inequality, we obtain the inequality of the third line. The last inequality follows by applying the first inequality of \eqref{eq:Est22} to bound $|\nabla^{+} \mathfrak{p}^{\eps}_{a-b}(x-y)|$ by $C (1\wedge(a-b)^{-\frac{3}{2}})$ and the second inequality of \eqref{eq:Est52} to bound the sum $\sum_{y\in \ZZ}\mathfrak{p}^{\eps}_{a-b}(x-y)\exp(2(u+1)\eps|x-y|)$ by a constant. Furthermore, the geometric sum $\sum_{y \in \ZZ}e^{-\eps|y|}$ is bounded above by $C\eps^{-1}$ for some constant $C>0$. From \eqref{eq:IFormBound}, \eqref{eq:BdI} follows by noting that the r.h.s. of \eqref{eq:IFormBound} does not depend on $x$ after dividing both sides by $\exp(2(u+1)\eps|x|)$.
 % To use these, one should keep in mind that the microscopic semi-discrete heat kernel $\mathfrak{p}^{\eps}$ and its macroscopic counterpart $p^{\eps}$ are related to each other via $\mathfrak{p}^{\eps}_t(x)= \eps p^\eps_{\eps^2 t}(\eps x)$ for any $x\in \ZZ$ and $t\in \RR_{\geq 0}$.

  Now, we show \eqref{eq:BdII}. In a similar way as used to derive  \eqref{eq:quadMMs},  we may use \eqref{eq:1} and \eqref{eq:Approx2} to show that
%  To begin, we note
%\begin{align}
%(\widetilde{\mathbf{II}})= \mathbb{E}\Big[\int^{r}_{s}\sum_{y\in \ZZ} \nabla^{+}\mathfrak{p}^{\eps}_{r-q}(x-y)\nabla^{-}\mathfrak{p}^{ \eps}_{r-q}(x-y)d_q \big\langle \tilde M^{\eps}(\cdot,y), \tilde M^{\eps}(\cdot,y)\big\rangle_q\Big|\mathcal{F}_{s}\Big].\label{eq:IIform}
%\end{align}
 \begin{align}
 (\mathbf{II}) = 2\eps\int^{a}_{b}  \sum_{y\in \ZZ} \nabla^{+} \mathfrak{p}^{\eps}_{a-c}(x-y) \nabla^{-} \mathfrak{p}^{\eps}_{a-c}(x-y)\Big(e^{2c\theta}\!-\! \mathbb{E}\big[\eps^{-1}\nabla^{+} \widetilde{\mathcal{Z}}^{\eps}_c(y) \nabla^{-} \widetilde{\mathcal{Z}}^{\eps}_c(y)\\
 -\tfrac{1}{2}\eps^{\frac{1}{2}} U^{\eps}_{c}(x)|\mathcal{F}_{b}\big] \Big) dc, &&\nonumber%\label{eq:IIReform}
 \end{align}%\const{can we write $\mathcal{B}_{\eps}$ explicitly?}
 where $U^{\eps}_t$ is defined in \eqref{eq:Ueps}. We can write the expression above as a sum of three terms which we will denote by $\mathcal{W}_1$, $\mathcal{W}_2$ and $\mathcal{W}_3$ and those are defined as follows:
 \begin{align}
 \mathcal{W}_1  & := 2\eps\int^{a}_{b} e^{2c\theta}\sum_{y\in \ZZ} \nabla^{+} \mathfrak{p}^{\eps}_{a-c}(x-y) \nabla^{-} \mathfrak{p}^{\eps}_{a-c}(x-y)dc,\\
 \mathcal{W}_2 &:= 2\int^{a}_{b} \sum_{y \in \ZZ} \nabla^{+} \mathfrak{p}^{\eps}_{a-c}(x-y)\nabla^{-}\mathfrak{p}^{\eps}_{a-c}(x-y)\mathbb{E}\big[\nabla^{+}\widetilde{\mathcal{Z}}^{\eps}_{c}(y)\nabla^{-}\widetilde{\mathcal{Z}}^{\eps}_{c}(y)\big| \mathcal{F}_{b}\big] dc,\\
 \mathcal{W}_{3} &:= \eps^{\frac{3}{2}} \int^{a}_{b} \sum_{y \in \ZZ} \nabla^{+}\mathfrak{p}^{\eps}_{a-c}(x-y) \nabla^{-}\mathfrak{p}^{\eps}_{a-c}(x-y)\mathbb{E}\big[U^{\eps}_{c}(x)\big| \mathcal{F}_{b}\big] dc.
 \end{align}
% $$\mathcal{B}^{\eps}_t(x):=\frac{1}{4}\eps^{-\frac{1}{2}}e^{2 \theta t}\big[(1+q^{-s_t(x)})(1+q^{\frac{s_t(x+1)+s_t(x-1)}{2}})-4\big](1-\nabla^{+}s_t(x)\nabla^{-}s_t(x)+\mathcal{E}^{\eps}_{t}(x)) +\mathfrak{B}^{\eps}_t(x). $$
% We refer to Proposition~\ref{DSHEcompute} for the definition of $\mathfrak{B}^{\eps}$.
Next, we show that the following holds:
\begin{align}
\sup_{x\in \ZZ} e^{-2(u+1)|x|}|\mathcal{W}_1|\leq C\eps(1\wedge(a-b)^{-\frac{1}{2}}),\quad & \quad   \sup_{x\in \ZZ} e^{-2(u+1)\eps|x|} \|\mathcal{W}_3\|_{2k}\leq C \eps^{\frac{3}{2}},&\label{eq:W_12Bd}\\
\sup_{x\in \ZZ} e^{-2(u+1)\eps|x|}\|\mathcal{W}_{2}\|_{2k}\leq & \int^{a}_{b}(1\wedge (a-c)^{-\frac{3}{2}}) f^{(k)}_{u+1}(c;b)dc.\label{eq:W_3Bd}
\end{align}
Since $\|(\mathbf{II})\|_{2k}$ is bounded above by $\|\mathcal{W}_1\|_{2k}+ \|\mathcal{W}_2\|_{2k}+ \|\mathcal{W}_3\|_{2k}$ via triangle inequality of the $L^{2k}$-norm, combining \eqref{eq:W_12Bd} and \eqref{eq:W_3Bd} yields \eqref{eq:BdII}.

Throughout the rest, we will prove the inequalities in \eqref{eq:W_12Bd} and \eqref{eq:W_3Bd}. We start with proving the first inequality of \eqref{eq:W_12Bd}.
To prove this, we use the following `key identity' of \cite{Bertini1997} (see also \cite[Lemma~4.2]{CSL16} or \eqref{eq:K_kernel_identity} herein)
\begin{align}\label{eq:KeyIneq}
 \int^{\infty}_{0}\sum_{y\in \ZZ}\nabla^{+} \mathfrak{p}^{\eps}_{r}(x-y) \nabla^{-} \mathfrak{p}^{\eps}_{r}(x-y) dr=0.
 \end{align}
 Via the change of variable $c\mapsto a-c$ we may rewrite
 \begin{align}
 \mathcal{W}_1 &= 2\eps e^{-2a\theta}\int^{a-b}_{0} e^{2c\theta}\sum_{y\in \ZZ} \nabla^{+} \mathfrak{p}^{\eps}_{c}(x-y) \nabla^{-} \mathfrak{p}^{\eps}_{c}(x-y) dc\nonumber\\
 & =2\eps  e^{-2a\theta}\int^{a-b}_{0}(e^{2c\theta}-1)\sum_{y\in \ZZ} \nabla^{+} \mathfrak{p}^{\eps}_{c}(x-y) \nabla^{-} \mathfrak{p}^{\eps}_{c}(x-y) dc\nonumber\\
 & \qquad - 2\eps e^{-2a\theta}\int^{\infty}_{a-b}\sum_{y\in \ZZ} \nabla^{+} \mathfrak{p}^{\eps}_{c}(x-y) \nabla^{-} \mathfrak{p}^{\eps}_{c}(x-y) dc,\label{eq:W1}
 \end{align}
 where the second equality follows by first splitting the integral into two parts by writing $e^{2c\theta}$ as $(e^{2c\theta}-1)+1$ and then, using \eqref{eq:KeyIneq} for the second part. We claim that
 \begin{align}
 \eps e^{-2a\theta}\int^{a-b}_{0}|e^{2c\theta}-1|\sum_{y\in \ZZ} \nabla^{+} \big|\mathfrak{p}^{\eps}_{c}(x-y) \nabla^{-} \mathfrak{p}^{\eps}_{c}(x-y)\big| dc\leq C\eps^3 (a-b)^{\frac{1}{2}} ,&&\label{eq:IntegralBd1}
\end{align}
while
 \begin{align}\label{eq:IntegralBd2}
\eps e^{-2a\theta}\int^{\infty}_{a-b}\sum_{y\in \ZZ} \big|\nabla^{+} \mathfrak{p}^{\eps}_{c}(x-y) \nabla^{-} \mathfrak{p}^{\eps}_{c}(x-y)\big| dc\leq C \eps \big(1\wedge (a-b)^{-\frac{1}{2}}\big).
\end{align}

Note that the l.h.s. of \eqref{eq:IntegralBd1} and \eqref{eq:IntegralBd2} bound the two term on the r.h.s.  \eqref{eq:W1} respectively (displayed in the last two lines of \eqref{eq:W1}). The first inequality of \eqref{eq:W_12Bd} follows immediately from this after recalling that $b<a\in [0,\eps^{-2}T]$ and hence $\eps^{3}(a-b)^{\frac{1}{2}} + \eps(a-b)^{-\frac{1}{2}} <C \eps(a-b)^{-\frac{1}{2}}$ for a constant only depending on $T$. To show \eqref{eq:IntegralBd2} we use that for any $T>0$ there exists a constant $C$ such that the bound $|e^{\omega}-1|\leq C\omega$ holds for all $\omega\in [0,T]$ to control $|e^{2c\theta}-1|\leq C \eps^2 c$; we also  apply the first inequality of \eqref{eq:Est22}
%\note{I need to finish rewriting this proof. Its a bit annoying to go to the $A$ variables in the conclusion here. Maybe the whole Lemma should be restated into terms of microscopic (small) variables.}
\begin{align}\label{eq:nablaBd}
\big|\nabla^{+} \mathfrak{p}^{\eps}_{c}(x-y) \nabla^{-} \mathfrak{p}^{\eps}_{c}(x-y)\big|\leq  \min\{1, c^{-\frac{3}{2}}\} \big(\mathfrak{p}^{\eps}_{c}(x-y)+\mathfrak{p}^{\eps}_{c}(x-y+1)\big).  &&&
\end{align}
Substituting \eqref{eq:nablaBd} and the inequality $|e^{2c\theta}-1|\leq C\eps^2c$ into the l.h.s. of \eqref{eq:IntegralBd1}, summing over $y$ and integrating w.r.t. $c$, we get \eqref{eq:IntegralBd1}. In a similar way, we get \eqref{eq:IntegralBd2} by using \eqref{eq:nablaBd}. %Thus, substituting the inequalities of \eqref{eq:IntegralBd1} and  \eqref{eq:IntegralBd2} into \eqref{eq:W1} shows that $|\mathcal{W}_1|\leq C\eps^2(A-B)^{-\frac{1}{2}}$ for some constant $C=C(T)>0$.
\smallskip

Starting with $\mathcal{W}_3$, by using triangle inequality of the $L^{2k}$-norm, we write
\begin{align}
\|\mathcal{W}_3\|_{2k}\leq \eps^{\frac{3}{2}}\int^{a}_{b} \sum_{y\in \ZZ} |\nabla^{+} \mathfrak{p}^{\eps}_{a-c}(x-y) \nabla^{-} \mathfrak{p}^{\eps}_{a-c}(x-y) |\big\|\mathbb{E}\big[U^\eps_{c}(y)\big| \mathcal{F}_{b}\big]\big\|_{2k} dc. \label{eq:W2}
\end{align}
%\begin{align}\label{eq:W_2}
%\mathcal{W}_2:= \eps^{\frac{3}{2}}\int^{a}_{b} \sum_{y\in \ZZ} \nabla^{+} \mathfrak{p}^{\eps}_{a-c}(x-y) \nabla^{-} \mathfrak{p}^{\eps}_{a-c}(x-y) \mathbb{E}\big[U^\eps_{c}(x)\big| \mathcal{F}_{c}\big] dc,
%\end{align}
Owing to the Jensen's inequality, we may bound $\big\|\mathbb{E}\big[U^\eps_{c}(y)\big| \mathcal{F}_{b}\big]\big\|_{2k}$ by $\big\|U^\eps_{c}(y)\big\|_{2k}$. Now, we can use the fact (shown soon after \eqref{eq:Ueps}) that for any $k\geq 1$, the expectation $\mathbb{E}[\|U^{\eps}_{t}(y)\|_{2k}]$ after scaling by $ e^{-2(u+1)\eps|y|}$ is uniformly bounded in $y$ as $\eps\to 0$. Combining this last fact with the bound on $|\nabla^{+} \mathfrak{p}^{\eps}_{a-c} \nabla^{-}\mathfrak{p}^{\eps}_{a-c}|$ from \eqref{eq:nablaBd} yields that $e^{-2(u+1)\eps|x|}\|\mathcal{W}_{3}\|_{2k}$ is bounded by
\begin{align}
C \eps^{\frac{3}{2}} \int^{a}_{b}\sum_{y\in \ZZ}(1\wedge (a-c)^{-\frac{3}{2}})\bigl(\mathfrak{p}^{\eps}_{a-c}(x-y)+ \mathfrak{p}^{\eps}_{a-c}(x-y+1)\bigr)e^{2(u+1)\eps|x-y|} dc. \nonumber
\end{align}
for some constant $C$ which does not depend on $\eps,T$ or $x$.  By the second inequality of \eqref{eq:Est52}, the sum $\sum_{y\in \ZZ}\mathfrak{p}^{\eps}_{a-c}(x-y)\exp(2(u+1)\eps|x-y|)$ and $\sum_{y\in \ZZ}\mathfrak{p}^{\eps}_{a-c}(x-y+1)\exp(2(u+1)\eps|x-y|)$ is bounded above by some positive constant which only depends on $T$. Moreover, $\int^{a}_{b} (1\wedge (a-c)^{-3/2})dc$ is bounded above by  $1$. As a consequence the r.h.s. of the last display is bounded  above by $C\eps^{3/2}$ where $C$ only depends on $T$ and $u$. This proves the second inequality of \eqref{eq:W_12Bd}.

We are left to show \eqref{eq:W_3Bd} which we prove as follow.
%In order to bound the $L^{2k}$ norm of $(\mathbf{II})$, it remains to bound
%\begin{align}\label{eq:W_3}
%\mathcal{W}_3:= \int_{a}^{b} \sum_{y\in \ZZ} \nabla^{+} \mathfrak{p}^{\eps}_{a-c}(x-y)\nabla^{-}\mathfrak{p}^{\eps}_{a-c}(x-y) \mathbb{E}\Big[\nabla^{+}\widetilde{\mathcal{Z}}^{\eps}_c(y)\nabla^{-}\widetilde{\mathcal{Z}}^{\eps}_c(y)\Big|\mathcal{F}_{b}\Big] dc.
%\end{align}
%To this goal, for $\kappa\in \RR_{\geq 0}, c\in [b,a], k\in \ZZ_{\geq 1}$, define
%Now, we claim and prove that for any $k\in \NN$, there exists $C=C(u,k,T)>0$ such that  \begin{align}\label{eq:Recursive}
%\|\mathcal{W}_3\|_{2k}\leq C\int^{a}_{b} \min\{1, (a-c)^{-\frac{3}{2}}\} f^{(k)}_{u}(c,b) dc.
%\end{align}
 Via the triangle inequality of the $L^{2k}$-norm, we may write
\begin{equation}
\|\mathcal{W}_2\|_{2k} \leq 2\int^{a}_{b}\sum_{y\in \ZZ} \big|\nabla^{+} \mathfrak{p}^{\eps}_{a-c}(x-y) \nabla^{-}\mathfrak{p}^{\eps}_{a-c}(x-y)\big|\times \Big\|\mathbb{E}\Big[\nabla^{+}\widetilde{\mathcal{Z}}^{\eps}_c(y)\nabla^{-}\widetilde{\mathcal{Z}}^{\eps}_c(y)\Big|\mathcal{F}_{b}\Big]\Big\|_{2k}dc.\label{eq:JensenRes}
\end{equation}
 From the definition of $f^{(k)}_{\kappa}(\bigcdot; b)$,
$$\Big\|\mathbb{E}\Big[\nabla^{+}\widetilde{\mathcal{Z}}^{\eps}_c(y)\nabla^{-}\widetilde{\mathcal{Z}}^{\eps}_c(y)\Big|\mathcal{F}_{b}\Big]\Big\|_{2k} \leq e^{2(u+1)\eps|y|} f^{(k)}_{u+1}(c;b).$$
Combining this last inequality with \eqref{eq:nablaBd} (to bound $|\nabla^{+}\mathfrak{p}^{\eps}_{a-c}(x-y)\nabla^{-}\mathfrak{p}^{\eps}_{a-c}(x-y)|$) and applying in the r.h.s. of \eqref{eq:JensenRes} yields that the r.h.s. of \eqref{eq:JensenRes} is bounded by
\begin{align}
C\int^{a}_{b} (1\wedge (a- c)^{-\frac{3}{2}})f^{(k)}_{u+1}(c,b) \sum_{y\in \ZZ}e^{2(u+1)\eps|y|}\big(\mathfrak{p}^{\eps}_{a-c}(x-y)+\mathfrak{p}^{\eps}_{a-c}(x-y+1)\big)dc.\nonumber
\end{align}
Via triangle inequality, we bound $\exp(2(u+1)\eps|y|)$ by $\exp(2(u+1)\eps(|x-y|+|x|))$ in the r.h.s. of the above inequality. Owing to the second inequality of \eqref{eq:Est52}, one can bound $\sum_{y\in \ZZ}e^{2(u+1)\eps|x-y|}\mathfrak{p}^{\eps}_{a-c}(x-y)$ and $\sum_{y\in \ZZ}e^{2(u+1)\eps|x-y|}\mathfrak{p}^{\eps}_{a-c}(x-y+1)$ by some constant $C= C(u,T)>0$. Combining these estimates and substituting those into the last inequality in the above display we arrive at \eqref{eq:W_3Bd}. This completes the proof.
 \end{proof}

%By using the triangle inequality of the $L^{2k}$ norm to bound $f^{(k)}_{(u+1)}(a,b)$ by $e^{-2(u+1)x}(\|(\mathbf{I})\|_{2k}+\|(\mathbf{II})\|_{2k})$  and combining the upper bound on the $L^{2k}$ norm of $(\mathbf{I})$ from \eqref{eq:IFormBound} with the bounds on $|\mathcal{W}_1|$, $\|\mathcal{W}_2\|_{2k}$ and $\|\mathcal{W}_3\|_{2k}$, we arrive at
%
%for some positive constants $C_1=C_1(k,u,\beta, T)$, $C_2=C_2(k,u,\beta, T)$ and $C_3=C_3(k,u,\beta, T)$.

\subsection{Proof of Proposition~\ref{ppn:MgProb}}

Let $\mathcal{Z}$ be a limit of a subsequence $\{\mathcal{Z}^{\eps}\}_{\eps}$. Then, for a $\phi\in \mathcal{C}^{\infty}_{b}(\RR)\cap L^{2}(\RR)$, the random variables $\mathfrak{M}^{\eps}_T(\phi)$ and $\mathfrak{N}^{\eps}_{T}(\phi)$, defined in \eqref{eq:DefM}-\eqref{eq:DefN}, converge to $\mathfrak{M}_{T}(\phi)$ and $\mathfrak{N}_{T}(\phi)$ from \eqref{eq:Mart} in $L^{1}$ as $\eps \to 0$. This implies that $\mathfrak{M}(\phi)$ and $\mathfrak{N}(\phi)$ are two local martingales, and hence $\mathcal{Z}$ solves the martingale problem associated to \eqref{eq:SHEAdditive} with $A=0$ and $B_T= e^{T/4}$. By \cite[Ch.~8, Thm.~8.1.5]{SV06}, $\mathcal{Z}$ is the unique solution of \eqref{eq:SHEAdditive} started from the initial data $\mathcal{Z}_0$, which is the weak limit of the sequence $\{\mathcal{Z}^{\eps}_0\}_{\eps}$.

%\section{KPZ Limit deformed dynamic ASEP}
%In our calculation, we see that
%\begin{align}
%\partial_x s^{\eps}(t,x) &= \frac{1}{2}\partial_{x} s^{\eps}(t,x) +(\lambda_{\eps} - \frac{1}{2})\eps^{-1} + \lambda_{\eps} \nabla^{+,\eps}_{x} \nabla^{-,\eps}_{x} + M^{\eps}(dt,x), \quad \text{ when }\alpha = e^{-a\eps}, a>\frac{1}{2}\label{eq:SpecCase}\\
%\partial_x s^{\eps}(t,x) &= \frac{1}{2}\partial_{x} s^{\eps}(t,x) +(\lambda_{\eps} + \frac{1}{2})\eps^{-1} + \lambda_{\eps} \nabla^{+,\eps}_{x} \nabla^{-,\eps}_{x} + M^{\eps}(dt,x), \quad \text{ when }\alpha = e^{-a\eps}, a<-\frac{1}{2}.\label{eq:SpecCase}
%\end{align}
%where $s^{\eps}(t,x)=\sqrt{\eps} s(\eps^{-2}t, \eps^{-1}x)$.
%In \eqref{eq:SpecCase} above, $\lambda_{\eps}\to \frac{1}{2}$ and in \eqref{eq:SpecCase}, $\lambda_{\eps}\to -\frac{1}{2}$. In either of those two cases, we seek to show that $s^{\eps}(t,x)$ converges to the Hopf-Cole solution of the KPZ equation.

\section{Convergence of a generalized dynamic ASEP}\label{sec:GenDASEP}

In this section we prove Theorem~\ref{thm:tight_intro}, without relying on the duality relation of \cite[Theorem~2.3]{BC17}. To this end, we write a system of SDEs governing the evolution of the rescaled height function $\hat s^\eps$, and use the Da Prato-Debussche trick \cite{DaPrato} to prove convergence of solutions. Since this method is quite robust, we can consider a more general evolution of the height functions \eqref{eq:ModJumpRate_intro}, with a function $\ff$ satisfying Assumption~\ref{a:f}.

\subsection{Heuristics of the argument}
\label{sec:heuristic}

We start with describing heuristics of our argument. For a height function $s \in \Omega^N_\chi$, we denote by $a^{\downarrow}(s, x)$ and $a^{\uparrow}(s, x)$ the down and up jumps rates in \eqref{eq:ModJumpRate_intro} respectively. (Recall, that the set $\Omega^N_\chi$ contains periodic height functions $s$.) Let $\zeta^{\downarrow}_t(x)$ and $\zeta^{\uparrow}_t(x)$ be the processes describing down and up jumps of the height function $s_t(x)$. %(throughout this section we use the notations $s_t(x)$ and $s(t, x)$ interchangeably, depending on convenience).
Then they are solutions of the system of SDEs
\begin{equation}\label{e:jumps}
 d \zeta^{\downarrow}_t(x) = 2\, \mathbbm{1}_{\{s_t(x)>s_t(x-1)=s_t(x+1)\}} d Q^{\downarrow}_t(x), \quad d \zeta^{\uparrow}_t(x) = 2\, \mathbbm{1}_{\{s_t(x)<s_t(x-1)=s_t(x+1)\}} d Q^{\uparrow}_t(x),
\end{equation}
with the initial states $\zeta^{\downarrow}_0(x) = \zeta^{\uparrow}_0(x) = 0$, where $Q^{\downarrow}_t(x)$ and $Q^{\uparrow}_t(x)$ are Poisson processes with rates $a^{\downarrow}(s_t, x)$ and $a^{\uparrow}(s_t, x)$ respectively. To be more precise, we should use the left limits of $s$ at time $t$ on the r.h.s. of both equations in \eqref{e:jumps}. However, we prefer not to indicate it, to make our notation less cumbersome. The evolution of the height function $s$ is described by
\begin{align}\label{eq:s_growth}
 d s_t(x) = d \zeta^{\uparrow}_t(x) - d \zeta^{\downarrow}_t(x).
\end{align}
To make the SDEs martingale-driven, we define the compensated Poisson processes to be the solutions of $d \tilde Q^{\downarrow}_t(x) = d Q^{\downarrow}_t(x) - a^{\downarrow}(s_t, x) dt$ and $d \tilde Q^{\uparrow}_t(x) = d Q^{\uparrow}_t(x) - a^{\uparrow}(s_t, x) dt$, starting from zeros at time $t=0$. These new processes are c\`{a}dl\`{a}g martingales with the predictable quadratic covariations given by $\langle \tilde Q^{\downarrow}(x), \tilde Q^{\uparrow}(y)\rangle_t \equiv 0$ and
\begin{align}
\langle \tilde Q^{\downarrow}(x), \tilde Q^{\downarrow}(y)\rangle_t = \mathbbm{1}_{x=y} \int_0^t a^{\downarrow}(s_r, x) dr, \qquad \langle \tilde Q^{\uparrow}(x), \tilde Q^{\uparrow}(y)\rangle_t = \mathbbm{1}_{x=y} \int_0^t a^{\uparrow}(s_r, x) dr.
\end{align}
It is easy to check that the following two identities hold:
\begin{align}
 \mathbbm{1}_{\{s(x)<s(x-1)=s(x+1)\}} &= \frac{1}{4}(1 + \nabla^- s(x)) (1 - \nabla^+ s(x)), \\
 \mathbbm{1}_{\{s(x)>s(x-1)=s(x+1)\}} &= \frac{1}{4}(1 - \nabla^- s(x)) (1 + \nabla^+ s(x)),
\end{align}
where $\nabla^\pm$ are discrete derivatives, defined in Section~\ref{sec:notation}. We will often also use the following two functions
\begin{align}\label{e:rho_lambda}
\rho(s, x) := \frac{a^{\uparrow}(s, x) + a^{\downarrow}(s, x)}{2}, \qquad \lambda(s, x) := \frac{a^{\uparrow}(s, x) - a^{\downarrow}(s, x)}{2}.
\end{align}
Combining these identities with \eqref{eq:s_growth}, we obtain the systems of SDEs describing the evolution of the height function $s$:
\begin{align}\label{eq:growth}
 d s_t(x) = \rho(s_t(x)) \DDelta s_t(x) dt + F(s_t, x) dt + d M_t(x),
\end{align}
where $\DDelta := \nabla^+ - \nabla^-$ is the discrete Laplacian, the function $F$ is given by
\begin{align}
F(s, x) := \lambda(s, x) \bigl( 1 - \nabla^- s(x)\, \nabla^+ s(x)\bigr),
\end{align}
and $t \mapsto M_t(x)$ is a c\`{a}dl\`{a}g martingale, starting at $0$, with jumps of size $1$ and with the
bracket process satisfying $\frac{d}{d t}\langle M(x), M(y)\rangle_t = \mathbbm{1}_{x=y} C(s_t, x)$ where
\begin{align}
C(s, x) := 2 \rho(s, x) \bigl(1 - \nabla^- s(x)\, \nabla^+ s(x) \bigr) + 2 \lambda(s, x) \DDelta s(x).
\end{align}
For two different points $x \neq y$, the martingales $M_t(x)$ and $M_t(y)$ almost surely do not make jumps at the same time. Moreover, the number of jumps of $t \mapsto M_t(x)$ at every finite interval $[0,T]$ has bounded moments uniformly in $x$ and locally uniformly in $T$, which means that the martingale makes a.s. finitely many jumps on every bounded time interval.

We need to tilt the function to make it periodic. More precisely, under the diffusive scaling and after recentering define $\hat s^\eps_t(x) := \sqrt \eps (s(\eps^{-2} t, \eps^{-1} x) - \chi x)$, where $\chi$ is defined in the statement of Theorem~\ref{thm:tight_intro}. Then \eqref{eq:growth} becomes
\begin{align}\label{eq:growth_eps}
\partial_t \hat s^\eps_t(x) = \rho_\eps(\hat s^\eps_t, x) \Deltae \hat s^\eps_t(x) + F_{\!\eps}(\hat s^\eps_t, x) + \xi^\eps(t,x),
\end{align}
where $\nablae^{\pm}$ are the respective discrete derivatives and $\Deltae$ is the discrete Laplacian, defined in Section~\ref{sec:notation}. The rescaled functions in \eqref{eq:growth_eps} are
\begin{subequations}
\begin{align}
 \rho_\eps(\hat s^\eps, x) &:= \rho \bigl(\eps^{-1 / 2} \hat s^\eps(x) + \chi x, \eps^{-1}x\bigr), \label{e:rho_def}\\
 \lambda_\eps(\hat s^\eps, x) &:= \eps^{-3/2} \lambda \bigl(\eps^{-1 / 2} \hat s^\eps(x) + \chi x, \eps^{-1}x\bigr),\label{e:lambda_def}\\
F_{\!\eps}(\hat s^\eps, x) &:= \lambda_\eps(\hat s^\eps, x) \big(1 - \eps \nablae^- \hat s^\eps(x)\, \nablae^+ \hat s^\eps(x)\big).\label{e:Feps}
\end{align}
\end{subequations}
The noise in  \eqref{eq:growth_eps} is given by $\xi^\eps(t,x) := d M^\eps_t(x)$, where the rescaled martingales are $M^\eps_t(x) := \sqrt \eps M(\eps^{-2}t, \eps^{-1}x)$, have jumps of size $\sqrt \eps$ and have the predictable quadratic covariation $\frac{d}{d t} \langle M^\eps(x), M^\eps(y)\rangle_t = \eps^{-1} \mathbbm{1}_{x=y} C_{\!\eps} (\hat s^\eps_t, x)$, where
\begin{align}\label{e:Ceps}
C_{\!\eps} (\hat s^\eps, x) := 2 \rho_\eps(\hat s^\eps, x) \bigl(1 - \eps \nablae^- \hat s^\eps(x)\, \nablae^+ \hat s^\eps(x) \bigr) + 2 \eps^{3} \lambda_\eps(\hat s^\eps, x) \Deltae \hat s^\eps(x).
\end{align}
Furthermore, properties of the martingales $M$ imply that, for $x \neq y$, $M^\eps_t(x)$ and $M^\eps_t(y)$ a.s. do not jump together, and on every time interval $[0, \eps^2 T]$ the martingale $M^\eps_t(x)$ makes a.s. finitely many jumps.

%\section{Discrete equation for the growth process}

%For $s$ from the state space, let us denote
%\begin{align}\label{e:rates}
%a^{\downarrow}(s_x) = \frac{q(1+ \alpha q^{-s_x})}{1+ \alpha q^{-s_x+1}}, \qquad a^{\uparrow}(s_x) = \frac{1+\alpha q^{-s_x}}{1+\alpha q^{-s_x-1}}.
%\end{align}

Let us now take the asymmetry to be $q = e^{-\eps}$, as in the statement of Theorem~\ref{thm:tight_intro}. Then, we have the following results for the functions $\rho_\eps$ and $\lambda_\eps$.

\begin{lemma}\label{lem:rho_bound}
The functions $\rho_\eps$ and $\lambda_\eps$, defined in \eqref{e:rho_def} and \eqref{e:lambda_def}, have the properties:

\begin{enumerate}[(1)]
\item\label{it:rho_bound1} There is a constant $c_0$ such that $|\rho_\eps(\hat s, x) - 1 + \eps / 2 | \leq  c_0 \eps^2$, uniformly in $x \in \RR$, where $\gamma$ is from Assumption~\ref{a:f}.
\item\label{it:rho_bound2} Recall the constants  $\fa$ and $\gamma$, defined in Assumption~\ref{a:f}. Then one has the bounds $|\lambda_\eps(\hat s, x)| \leq c_1 \eps^{-1/2}$ and $|\lambda_\eps(\hat s, x) + \frac{\fa \hat s(x)}{4}|  \leq c_1 \eps^{1 - \gamma}(1 + \sqrt \eps |\hat s(x)|)^\gamma$ uniformly in $x \in \R$, for some constant $c_1$.
\end{enumerate}
The constants $c_0$ and $c_1$ are independent of $\hat s$, $\eps$ and $x$.
\end{lemma}

\begin{proof}
Let us denote for brevity $s = \eps^{-1 / 2} \hat s(x)$. Then, recalling the jump rates \eqref{eq:ModJumpRate_intro}, the function $\rho_\eps$ can be written as
\begin{align}
  \rho_\eps(\hat s, x) &= 1 - \frac{1 - e^{-\eps}}{2} + \frac{1 - e^{-\eps}}{2} \left(\frac{1}{1+  e^{\eps \ff(s) + \eps}} - \frac{1}{1+  e^{\eps \ff(s) - \eps}}\right)\\
   &= 1 - \frac{1 - e^{-\eps}}{2} + \frac{(1 - e^{-\eps}) (1 - e^{2\eps})}{2} \frac{1}{1+  e^{\eps \ff(s) + \eps}} \frac{ e^{\eps \ff(s) - \eps}}{1+  e^{\eps \ff(s) - \eps}}\;.
\end{align}
The last two factors are bounded uniformly in $s$ and $\eps$, hence the bound in (\ref{it:rho_bound1}) on $\rho_\eps$ follows.

Now, we will prove the bound in (\ref{it:rho_bound2}) on $\lambda_\eps$. To this end, we rewrite it as
\begin{align}
\lambda_\eps(\hat s, x) = \frac{1 - e^{-\eps}}{2 \eps^{3/2}} \left( \frac{1}{1+  e^{\eps \ff(s) + \eps}} + \frac{1}{1+  e^{\eps \ff(s) - \eps}} - 1\right).\label{e:lambda}
\end{align}
The terms in the parenthesis are bounded by a constant, yielding $|\lambda_\eps(\hat s, x)| \leq c_1 \eps^{-1/2}$. Further, using the Taylor expansion for $e^{-\eps}$, we can write
\begin{equation*}
\lambda_\eps(\hat s, x) = \frac{1}{2 \sqrt \eps} \left( \frac{1}{1+  e^{\eps \ff(s) + \eps}} + \frac{1}{1+  e^{\eps \ff(s) - \eps}} - 1\right) +  \lambda_\eps^{(1)}(\hat s, x),
\end{equation*}
where $|\lambda_\eps^{(1)}(\hat s, x)| \leq C \sqrt \eps$. Next, we will replace $\ff(s)$ by $\fa s$:
\begin{equation*}
\frac{1}{1+  e^{\eps \ff(s) \pm \eps}} = \frac{1}{1+ e^{\sqrt \eps \fa \hat s \pm \eps}} + \frac{e^{\sqrt \eps \fa \hat s \pm \eps}}{1 + e^{\sqrt \eps \fa \hat s \pm \eps}} \frac{1 -  e^{ \eps (\ff(s) - \fa s)}}{1+  e^{\eps \ff(s) \pm \eps}},
\end{equation*}
where the last term can be bounded, using Assumption~\ref{a:f}, by a multiple of $|1 -  e^{ \eps (\ff(s) - \fa s)}| \leq C\eps^{1 - \gamma}(\sqrt \eps |\hat s(x)| + 1)^\gamma$. Hence, we obtain
\begin{equation*}
\lambda_\eps(\hat s, x) = \frac{1}{2 \sqrt \eps} \left( \frac{1}{1+ e^{\sqrt \eps \fa \hat s + \eps}} + \frac{1}{1+ e^{\sqrt \eps \fa \hat s - \eps}} - 1\right) +  \lambda_\eps^{(2)}(\hat s, x),
\end{equation*}
where $|\lambda_\eps^{(2)}(\hat s, x)| \leq C\eps^{1 - \gamma}(\sqrt \eps |\hat s(x)| + 1)^\gamma$. From this the required bound on $\lambda_\eps$ follows.
\end{proof}

Now, we will investigate the limit of the functions $F_{\!\eps}$, defined in \eqref{e:Feps}. Lemma~\ref{lem:rho_bound} yields
\begin{align}\label{e:F_bound}
F_{\!\eps}(\hat s^\eps, x) = \left(- \frac{\fa \hat s^\eps(x)}{4} + \hat \lambda_\eps(\hat s^\eps, x) \right) \big( 1 - \eps \nablae^- \hat s^\eps(x)\, \nablae^+ \hat s^\eps(x)\big),
\end{align}
where $\hat \lambda_\eps(\hat s^\eps, x) := \lambda_\eps(\hat s^\eps, x) + \frac{\fa \hat s^\eps(x)}{4}$, vanishing as $\eps \to 0$ as soon as $\hat s^\eps$ is bounded uniformly in $\eps$.
%Since the limit of $\hat s^\eps$ is expected to have H\"{o}lder regularity $\frac{1}{2}-$, the product $\nablae^- \hat s^\eps\, \nablae^+ \hat s^\eps$ is expected to diverge with the speed $c / \eps$, where $c := \E [\nabla^- s_{0} \nabla^+ s_0]$ in equilibrium.\km{not sure about it anymore} Hence, one can expect to have the limit
%\begin{align}
%F_{\!\eps}(\hat s^\eps) - \frac{c - 1}{4} \hat s^\eps\; \xrightarrow{\eps \to 0}\; 0.
%\end{align}
The product $\eps \nablae^- \hat s^\eps\, \nablae^+ \hat s^\eps$ is expected to vanish in the limit $\eps \to 0$ in a space of discretized distributions, which suggests the following limit in a respective topology:
\begin{align}
F_{\!\eps}(\hat s^\eps, x) + \frac{\fa \hat s^\eps(x)}{4}\; \xrightarrow{\eps \to 0}\; 0.
\end{align}
However, this limit is difficult to prove, because the function $F_{\!\eps}$ is non-linear in $s^\eps$. This is one of the main difficulties in the proof of Theorem~\ref{thm:tight_intro}, which is resolved in Lemma~\ref{lem:term3} below.

The martingales $M^\eps$, defining the random noise $\xi^\eps$ in \eqref{eq:growth_eps}, are expected to converge to the cylindric Wiener process, which implies that the limit of $\hat s^\eps$ is the periodic solution of \eqref{eq:SHEAdditive} with $B=1$ and $A=-\frac{\fa}{4}$.

We split the actual proof of Theorem~\ref{Cor1} into several steps: we rewrite the equation \eqref{eq:growth_eps} in mild form, and then bound each term in the expression we get. Derivation of bounds on the non-linear function $F_{\!\eps}$  is the most difficult part in our analysis.

\subsection{Reformulation of the problem}
\label{sec:mild}

The non-linear part of the equation \eqref{eq:growth_eps} makes it non-trivial to bound the solution. More precisely, we expect that $\xi^\eps$ converges to the space-time white noise. Which means that for every $t > 0$ the solution $\hat s^\eps_t$ is expected to have spatial  H\"{o}lder regularity $\frac{1}{2}-\kappa$, for any $\kappa > 0$. On the other hand, the function $F_{\!\eps}$, defined in \eqref{e:F_bound}, contains the term $\nablae^- \hat s^\eps\, \nablae^+ \hat s^\eps$ which needs to be controlled as $\eps \to 0$. We show below, that the factor $\eps$ in front of this term makes it vanish in a suitable topology. This seemingly easy fact is not straightforward to prove, and for this we use the idea of \cite{DaPrato}. %Of course, we don't need to use the whole machinery of this theory and can prove convergence using a generalization of the Da Prato-Debussche trick \cite{DaPrato}, similarly to how it was done for the KPZ equation in \cite{HaiKPZ}. However, in this case we would have to reprove several results for this particular model, which we prefer to avoid.\km{no RS}

%It will be convenient to rewrite equation \eqref{eq:growth_eps} in a more general form
%\begin{align}\label{eq:growth_eps_new}
%\partial_t \hat s^\eps = \Deltae \hat s^\eps + c \hat s^\eps + \eps \hat s^\eps \bigl(\CK_\eps \star_\eps \bigl(\nablae^- \hat s^\eps\, \nablae^+ \hat s^\eps \bigr)\bigr) + \SE_\eps(\hat s^\eps) + \xi^\eps,
%\end{align}
%where $c$ is a fixed constant, $\star_\eps$ is the space-time convolution on $\R \times \eps\ZZ$ and $\CK_\eps$ is an integrable kernel. The function $\SE_\eps$ is an error term, which vanishes as $\eps \to 0$ in some space of distributions. In the following sections, we define precisely $c$, $\CK_\eps$ and $\SE_\eps$ and write \eqref{eq:growth_eps} in this form.

 We start with rewriting \eqref{eq:growth_eps} in a mild form. To this end, we need to replace the non-constant multiplier $\rho_\eps$ by $1$. Using Lemma~\ref{lem:rho_bound}, we can write $\rho_\eps(\hat s^\eps, x) = 1 + \hat \rho_\eps(\hat s^\eps, x)$, where we have a good control over $\hat \rho_\eps$.
Hence, we rewrite \eqref{eq:growth_eps} as
\begin{align}\label{eq:growth_eps_new}
\partial_t \hat s^\eps_t(x) = \Deltae \hat s^\eps_t(x) + \hat F_{\!\eps}(\hat s^\eps_t, x) + F_{\!\eps}(\hat s^\eps_t, x) + \xi^\eps(t, x),
\end{align}
with a new function
\begin{equation}\label{eq:hatF}
\hat F_{\!\eps}(\hat s^\eps, x) := \hat \rho_\eps(\hat s^\eps, x) \Deltae \hat s^\eps(x).
\end{equation}
Let $S^\eps_t := e^{t \Deltae}$ be the semigroup of the linear operator $\frac{d}{dt} - \Deltae$. Then the last equation can be written in the mild form
\begin{align}\label{eq:growth_eps_mild}
\hat s^\eps_t(x) = (S^\eps_t \hat s^\eps_0)(x) + \int_0^t (S^\eps_{t-r} (\hat F_{\!\eps} + F_{\!\eps})(\hat s^\eps_r))(x) dr + \int_0^t (S^\eps_{t-r} d M^\eps_r)(x).
\end{align}
Here, for a function $F_{\!\eps}(\hat s^\eps, x)$ depending on $\hat s^\eps$ and the spatial variable $x$, we use the notation $F_{\!\eps}(\hat s^\eps)(x) := F_{\!\eps}(\hat s^\eps, x)$. In particular, the semigroup $S^\eps$ in the middle term of \eqref{eq:growth_eps_mild} acts on the function $x \mapsto F_{\!\eps}(\hat s^\eps_r)(x)$. We will use the same notation for functions depending on some other quantity instead of $\hat s^\eps$.

We will write \eqref{eq:growth_eps_mild} in a way which gives a better control on the non-linearity $F_{\!\eps}$.

%To make our notation lighter, we will write ``$\lesssim$" for a bound ``$\leq$" up to a multiplier, independent of relevant quantities.
%We show in Appendix~\ref{sec:heat} that the continuous and discrete heat semigroups, when acting on spaces of different regularities, produce blow-ups at $t=0$. In order to describe such blow-ups of rate $\eta$ we define the norms\km{define Besov spaces and $\rightarrow$}
%\begin{align}
%\Vert f \Vert_{\mathcal{C}^{\alpha}_{\eta, T}} := \sup_{t \in (0,T]} t^{(\eta \vee 0)/2} \Vert f_t \Vert_{\mathcal{C}^{\alpha}}, \qquad \Vert f \Vert_{\mathcal{C}^{\alpha}_{\eta, T; \eps}} := \sup_{t \in (0,T]} (\sqrt t \vee \eps)^{\eta \vee 0} \Vert f_t \Vert_{\mathcal{C}^{\alpha}}\;.
%\end{align}
%For $\eps \in (0,1]$, we define the grid $\eps\ZZ := \eps \ZZ$.

\subsubsection{Definitions of auxiliary processes}\label{sec:AuxProc}

In this section we define some auxiliary processes, which enable us to obtain good control on the non-linearity in \eqref{eq:growth_eps_mild}. For the martingales $(M^\eps_t(x))_{t \geq 0}$ we define the processes
\begin{align}\label{eq:X_eps_mild}
\hat X^\eps_\tau(t, x) := \int_{0}^\tau (S^\eps_{t-r} d M^\eps_r)(x), \qquad X^\eps_t(x) = \hat X^\eps_t(t, x) + (S^\eps_t \hat s^\eps_0)(x),
\end{align}
where $\hat s^\eps_0$ is the initial data of \eqref{eq:growth_eps}.  We will write for brevity $\hat X^\eps_t(x) = \hat X^\eps_t(t, x)$. The process $\hat X^\eps_\tau(t, x)$ is a martingale for $\tau \in [0, t]$, and we can write explicitly its predictable covariation.
Then we introduce the function $u^\eps := \hat s^\eps - X^\eps$ and derive from \eqref{eq:growth_eps_mild}
\begin{align}\label{eq:u_eps}
u^\eps_t(x) = \int_0^t (S^\eps_{t-r} \hat F_{\!\eps}(X^\eps_r + u^\eps_r)) (x)dr + \int_0^t (S^\eps_{t-r} F_{\!\eps}(X^\eps_r + u^\eps_r))(x) dr.
\end{align}
The {\it ansatz} is that $u^\eps$ can be bounded in a space of higher regularity than $\hat s^\eps$. %Moreover, $u^\eps$ vanishes as $\eps \to 0$, which implies that the limit of $\hat s^\eps$, if it exists, coincides with the limit of $X^\eps$.
If we solve \eqref{eq:u_eps} for $u^\eps$, then the solution $\hat s^\eps$ of \eqref{eq:growth_eps_mild} is obtained by
\begin{equation}\label{eq:from_u_to_s}
\hat s^\eps = X^\eps + u^\eps.
\end{equation}

A problem with \eqref{eq:u_eps} is that the non-linearity $F_{\!\eps}$ contains the term $\nablae^- \hat X^\eps \nablae^+ \hat X^\eps$ which has a nasty behaviour in the limit $\eps \to 0$. However, we obtain good bounds in Lemma~\ref{lem:Z_bound} on its ``renormalized" version
\begin{align}\label{eq:Z_def}
Z^\eps_t(x) := \nablae^- \hat X^\eps_t(x) \nablae^+ \hat X^\eps_t(x) - \fC^\eps_t(x),
\end{align}
where we write $\fC^\eps_t(x) := \fC^\eps_t(t,x)$ for the predictable quadratic covariation
\begin{align}\label{eq:C_curly}
\fC^\eps_\tau(t, x) := \langle \nablae^- \hat X^\eps_{\bigcdot}(t, x), \nablae^+ \hat X^\eps_{\bigcdot}(t, x)\rangle_\tau.
\end{align}
We will write a Wick-type product of the processes $\hat s^\eps$ in the following way
\begin{align}
 \bigl(\nablae^- \hat{s}^\eps \diamond \nablae^+ \hat{s}^\eps\bigr)_t(x) &:= \nablae^- \hat{s}^\eps_t(x) \nablae^+ \hat{s}^\eps_t(x) - \fC^\eps_t(x),\\
 &= Z^\eps_t(x) + \nablae^- (X^\eps + u^\eps)_t(x) \nablae^+ u^\eps_t(x) + \nablae^- u^\eps_t(x) \nablae^+ X^\eps_t(x),\label{eq:wick}
\end{align}
in which the nasty product $\nablae^- \hat X^\eps \nablae^+ \hat X^\eps$ is replaced by its renormalized version $Z^\eps$. After that we replace the product $\nablae^- \hat{s}^\eps  \nablae^+ \hat{s}^\eps$ in \eqref{e:F_bound} by the Wick-type product $\nablae^- \hat{s}^\eps \diamond \nablae^+ \hat{s}^\eps$ by adding and subtracting the function $\fC^\eps$, so that this renormalization does not change \eqref{eq:u_eps}. More precisely, we define two new functions
\begin{align}\label{eq:F_tilde_def}
F_{\!\eps}(X^\eps, Z^\eps, u^\eps)_t(x) &:= \lambda_\eps(\hat s^\eps_t, x) \big( 1 - \eps \nablae^- \hat s^\eps_t \diamond \nablae^+ \hat s^\eps_t\big), \\
\tilde F_{\!\eps}(X^\eps, \hat X^\eps, u^\eps)_t(x) &:= \eps \lambda_\eps(\hat s^\eps_t, x) \fC^\eps_t(x),
\end{align}
where on the r.h.s. we use the function $\hat s^\eps$, which is defined in terms of $X^\eps$ and $u^\eps$ by \eqref{eq:from_u_to_s}. We note that the new function $F_{\!\eps}$ depends on $Z^\eps$ by \eqref{eq:wick}, and $\tilde F_{\!\eps}$ depends on $\hat X^\eps$ by \eqref{eq:C_curly}.
Then \eqref{eq:u_eps} can be written as
\begin{equation}\label{eq:u_eps_new}
u^\eps_t(x) = \int_0^t (S^\eps_{t-r} \hat F_{\!\eps}(\hat s^\eps_r)) (x)dr + \int_0^t (S^\eps_{t-r} \tilde F_{\!\eps}(X^\eps, \hat X^\eps, u^\eps)_r) (x)dr + \int_0^t (S^\eps_{t-r} F_{\!\eps}(X^\eps, Z^\eps, u^\eps)_r)(x) dr,
\end{equation}

This equation should be understood in the following way: we have defined two auxiliary processes, $X^\eps$ and $Z^\eps$. Fixing them, we solve \eqref{eq:u_eps_new} for $u^\eps$, and after that we recover the solution \eqref{eq:from_u_to_s} for the initial problem \eqref{eq:growth_eps_mild}. Adding these two auxiliary processes into the equation allows to obtain a better control on the non-linear term in \eqref{eq:growth_eps_mild}. More precisely, for $\eta$ as in Theorem~\ref{thm:tight_intro}, we fix any $\kappa_\star \in (0, \frac{1}{2})$ and $\hat \kappa \in (0, \frac{1}{2} - \kappa_\star)$, whose precise values will be specified in the proof of Theorem~\ref{thm:tight_intro}. Then, for any $T > 0$, we expect the following bounds to hold:
\begin{equation}\label{eq:bounds_eps}
\|X^\eps\|_{\CC^{\eta}_{T}} \leq L, \qquad \|\hat X^\eps\|_{\CC^{\eta}_{T}} \leq L, \qquad \|u^\eps\|_{\CC^{2 \eta}_{T}} \leq L, \qquad \| Z^\eps\|_{\CC^{-1/2 + \kappa_\star}_{T, \eps}} \leq \eps^{-1/2 - \kappa_\star - \hat \kappa} L,
\end{equation}
where the constant $L > 0$ is independent of $\eps$ and $T$, and where the norms are defined in Section~\ref{sec:notation}. Moreover, we expect the processes $X^\eps$, $\hat X^\eps$ and $u^\eps$ to converge in these topologies, as $\eps \to 0$. Since we expect the regularity of $\hat X^\eps$ be close to $\frac{1}{2}$, the definition \eqref{eq:Z_def} suggests that the regularity of $Z^\eps$ is close to $-1$. However, the process $Z^\eps$ is a discretization of a space-time distribution, i.e. the limit as $\eps \to 0$ is not a function in the time variable. This explains why the sup-norm of $Z^\eps$ in the time variable is expected to explode in the limit. It will be advantageous to measure regularity of $Z^\eps$ close to $-\frac{1}{2}$, which makes the divergence in \eqref{eq:bounds_eps} stronger. Since the term $Z^\eps$ is multiplied by $\eps$ in \eqref{eq:u_eps_new}, this divergence will be compensated (see the proof of Lemma~\ref{lem:term3}).

In order to derive the claimed bounds and prove limits, we will follow the usual strategy: we first derive the respective result for a local in time solution, and then patch local solutions together to get the required result for the global solution. To this end, for any constant $L > 0$ we define the stopping time
\begin{equation}\label{eq:sigma_L}
 \sigma_{L, \eps} := \inf \bigl\{T \geq 0 : \|X^\eps\|_{\CC^{\eta}_{T}} \geq L\; \text{or}\; \|u^\eps\|_{\CC^{2 \eta}_{T}} \geq L,\; \text{or}\;\| Z^\eps\|_{\CC^{-1/2 + \kappa_\star}_{T, \eps}} \geq \eps^{-1/2 - \kappa_\star - \hat \kappa} L\bigr\},
\end{equation}
which guarantees that on the random time interval $[0, \sigma_{L, \eps}]$ the a priori bounds \eqref{eq:bounds_eps} hold. %Considering the time intervals from $0$ to $T_{\!K} := T \wedge \sigma_{L, \eps}$ guarantees boundedness of the involved processes.

In the following exposition we prefer to use `$\lesssim$', which means a bound `$\leq$' with a constant multiplier, independent of the relevant quantities. By these relevant quantities we will usually mean $\eps$ and the space-time variables.

\subsubsection{Bounds on the auxiliary processes}

%In the following lemmas we bound the terms on the right-hand side of \eqref{eq:u_eps}.
In this section we prove bounds on the auxiliary processes $\hat X^{\eps}$, $X^\eps$ and $Z^\eps$, defined in \eqref{eq:X_eps_mild} and \eqref{eq:Z_def} respectively. We start by bounding the processes $\hat X^{\eps}$ and $X^{\eps}$.

\begin{lemma}\label{lem:X_bound}
In the setting of Theorem~\ref{thm:tight_intro}, for every $p \geq 1$ and $T > 0$ there is a constant $C = C(p) > 0$, such that the processes $\hat X^{\eps}$ and $X^{\eps}$, defined in \eqref{eq:X_eps_mild}, are bounded in the following way:
\begin{align}\label{eq:X_bound}
\bigl(\E \bigl[ \|\hat X^\eps\|_{\CC^{\eta}_{T}}^p\bigr]\bigr)^{\frac{1}{p}} \leq C, \qquad \qquad \bigl(\E \bigl[\|X^\eps\|_{\CC^{\eta}_{T}}^p\bigl]\bigr)^{\frac{1}{p}} \leq C + C \bigl(\E\bigl[\|s^\eps_0\|^p_{\CC^\eta}\bigl]\bigr)^{\frac{1}{p}}.
\end{align}
\end{lemma}

\begin{proof}
We note that the second bound in \eqref{eq:X_bound} follows from the first one and the properties of the heat semigroup provided in Lemma~\ref{lem:heat_S}. Indeed, from the triangle inequality we get
\begin{align}
\|X^\eps\|_{\CC^{\eta}_{T}} \leq \|\hat X^\eps\|_{\CC^{\eta}_{T}} + \|S^\eps \hat s^\eps_0\|_{\CC^{\eta}_{T}} \leq C (1 + \|s^\eps_0\|_{\CC^\eta}),
\end{align}
and the claim follows from the Minkowski inequality.

We now prove the first bound in \eqref{eq:X_bound}. To this end, we will apply the Burkholder-Davis-Gundy inequality (Lemma~\ref{lem:BDG}) to the martingale $\tau \mapsto \hat X^\eps_\tau(t,x)$, for which we need to bound the quadratic covariation and jumps of the latter. Hence, the definition \eqref{eq:X_eps_mild} and the formula for the quadratic variation of $M^\eps$, provided above \eqref{e:Ceps}, yield
\begin{align}
\langle \hat X^\eps_{\bigcdot} (t,x), \hat X^\eps_{\bigcdot} (t,x)\rangle_\tau &= \eps^2 \sum_{y_1, y_2 \in \eps\ZZ} \int_{r = 0}^\tau p^\eps_{t-r}(x-y_1) p^\eps_{t-r}(x-y_2) d \langle M^\eps(y_1), M^\eps(y_2)\rangle_r\\
&= \eps \sum_{y \in \eps\ZZ} \int_{0}^\tau p^\eps_{t-r}(x-y)^2 C_{\!\eps} (\hat s^\eps_r, y) d r,
\end{align}
where $p^\eps$ is the discrete heat kernel, generated by $\Delta_\eps$. Moreover, \eqref{e:Ceps}, Lemma~\ref{lem:rho_bound} and the identity $|\nablae \hat s^\eps| = \eps^{-1/2}$ imply that $C_{\!\eps} (\hat s^\eps_r, y)$ is bounded uniformly in $\eps$, which yields
\begin{equation}\label{eq:X_bound_cov}
\langle \hat X^\eps_{\bigcdot} (t,x), \hat X^\eps_{\bigcdot} (t,x)\rangle_\tau \lesssim \eps \sum_{y \in \eps\ZZ} \int_{0}^\tau p^\eps_{t-r}(x-y)^2 d r.
\end{equation}
Now we turn to jumps of the martingales $\hat X^\eps_\tau (t,x)$. The definition \eqref{eq:X_eps_mild} yields
\begin{align}
\Delta_r \hat X^\eps_{\bigcdot} (t,x) = \eps \sum_{y \in \eps\ZZ} p^\eps_{t-r}(x-y) \Delta_r M^\eps(y),
\end{align}
where $\Delta_r M^\eps(y) := M^\eps_r(y) - M^\eps_{r-}(y)$ is the jump of $M^\eps(y)$ at time $r$ (and likewise for $\hat X^\eps$). Since the martingales $M^\eps_{r}(y)$ have jumps of size $\sqrt \eps$, and $M^\eps_{r}(y)$ and $M^\eps_{r}(y')$, with $y \neq y'$, a.s. do not jump simultaneously, we obtain the simple bound
\begin{align}\label{eq:X_bound_jump}
|\Delta_r \hat X^\eps_{\bigcdot} (t,x)| \leq \eps^{3/2} \sup_{y \in \eps\ZZ} p^\eps_{t-r}(x-y).
\end{align}
Applying now the Burkholder-Davis-Gundy inequality (Lemma~\ref{lem:BDG}) to the martingales $\tau \mapsto \hat X^\eps_\tau(t,x)$, and using \eqref{eq:X_bound_cov} and \eqref{eq:X_bound_jump}, we obtain
\begin{align}\label{eq:bound_1}
\left(\E |\hat  X^\eps(t,x)|^p\right)^{\frac{2}{p}} \lesssim \eps \sum_{y \in \eps\ZZ} \int_0^t p^\eps_{t-r}(x-y)^2 dr + \eps^{3} \sup_{r \in [0,t]} \sup_{y \in \eps\ZZ} p^\eps_{t-r}(x-y)^2.
\end{align}
Using the first bound in \eqref{eq:Est5}, the last term in the expression \eqref{eq:bound_1} can be bounded by $c \eps$, where $c$ is independent of $\eps$, $x$ and $t$. To bound the first term in \eqref{eq:bound_1}, we apply the first estimate in \eqref{eq:Est5} and the fact that the heat kernel sums up to $1$ in the spatial variable:
\begin{align}
\eps \sum_{y \in \eps\ZZ} \int_0^t p^\eps_{t-r}(x-y)^2 dr \lesssim \eps \sum_{y \in \eps\ZZ} \int_0^t \frac{p^\eps_{t-r}(x-y)}{\sqrt{t-r}} dr = \int_0^t \frac{dr}{\sqrt{t-r}} = 2 \sqrt t.
\end{align}
Hence, we have the following bound, where the constant $C$ depends on $p$ and $t$:
\begin{align}
 \left(\E |\hat  X^\eps(t,x)|^p\right)^{\frac{1}{p}} \leq C.\label{eq:bound_11}
\end{align}

Similarly, for two different points $x_1$ and $x_2$, the process $\tau \mapsto \hat X^\eps_\tau(t,x_1) - \hat X^\eps_\tau(t,x_2)$ is a martingale. Applying again the Burkholder-Davis-Gundy inequality (Lemma~\ref{lem:BDG}), similarly to \eqref{eq:bound_1} we obtain
\begin{align}\label{e:I4_bound}
\left(\E |\hat X^\eps(t,x_1) - \hat X^\eps(t,x_2)|^p\right)^{\frac{2}{p}} &\lesssim \eps \sum_{y \in \eps\ZZ} \int_0^t \bigl|p^\eps_{t-r}(x_1-y) - p^\eps_{t-r}(x_2-y)\bigr|^2 d r\\
&\qquad + \eps^{3} \sup_{r \in [0,t]} \sup_{y \in \eps\ZZ} \bigl|p^\eps_{t-r}(x_1-y) - p^\eps_{t-r}(x_2-y)\bigr|^2.\label{eq:bound_2}
\end{align}
Using \eqref{eq:Est2}, we can bound the last term in \eqref{eq:bound_2} by a multiple of $\eps^{1 - 2(\eta + \kappa)} |x_1 - x_2|^{2(\eta + \kappa)}$, for any $\kappa > 0$ sufficiently small. Since $\eta < \frac{1}{2}$, for $\kappa > 0$ small enough we have $1 - 2(\eta + \kappa) \geq 0$, and the last power of $\eps$ can be bounded by $1$. For the first term in \eqref{eq:bound_2} we use the bound \eqref{eq:Est2} to estimate it by a multiple of
\begin{align}
\eps \sum_{y \in \eps\ZZ} &\int_0^t \bigl|p^\eps_{t-r}(x_1-y) - p^\eps_{t-r}(x_2-y)\bigr| \frac{|x_1 - x_2|^{2 \eta + \kappa}}{(t-r)^{(1 + 2 \eta + \kappa) / 2}} d r \\
&\lesssim \eps \sum_{y \in \eps\ZZ} \int_0^t \bigl(p^\eps_{t-r}(x_1-y) + p^\eps_{t-r}(x_2-y)\bigr) \frac{|x_1 - x_2|^{2 \eta + \kappa}}{(t-r)^{(1 + 2 \eta + \kappa) / 2}} d r\\
&\lesssim \int_0^t \frac{|x_1 - x_2|^{2 \eta + \kappa}}{(t-r)^{(1 + 2 \eta + \kappa) / 2}} d r \lesssim |x_1 - x_2|^{2 \eta + \kappa},
\end{align}
which holds for $\eta > 0$ and $\kappa > 0$, such that $2 \eta + \kappa < 1$. Here, we have used the fact that the heat kernel sums up to $1$ in the spatial variable. Hence, we have the following bound, with a constant $C > 0$ depending on $p$ and $t$:
\begin{align}
\left(\E |\hat X^\eps(t,x_1) - \hat X^\eps(t,x_2)|^p\right)^{\frac{1}{p}} \leq C |x_1 - x_2|^{\eta + \kappa}.\label{eq:bound_23}
\end{align}

Now, we turn to the proof of time continuity. For two time points $0 \leq t_1 < t_2$, such that $t_2 - t_2 \leq 1$, we can write
\begin{align}
\hat X^\eps(t_2,x) - \hat X^\eps(t_1,x) &= \eps \sum_{y \in \eps\ZZ} \int_{0}^{t_1} \bigl(p^\eps_{t_2-r}(x-y) - p^\eps_{t_1-r}(x-y)\bigr) d M^\eps_r(y)\\
 & \qquad + \eps \sum_{y \in \eps\ZZ} \int_{t_1}^{t_2} p^\eps_{t_2-r}(x-y) d M^\eps_r(y) =: \CI^\eps_1(t_1) + \CI^\eps_2(t_2),
\end{align}
where the variables $t_1$ and $t_2$ in $\CI^\eps_1$ and $\CI^\eps_2$ refer to the upper bounds of the intervals of integration. The processes $\tau_1 \mapsto \CI^\eps_1(\eps^{-2} 1)$ and $\tau_2 \mapsto \CI^\eps_2(\tau_2)$ are martingales, for $\tau_1 \in [0, t_1]$ and $\tau_2 \in [t_1, t_2]$, and we are going to bound them separately.

Applying the Burkholder-Davis-Gundy inequality (Lemmas~\ref{lem:BDG}) to $\CI^\eps_1$, similarly to \eqref{eq:bound_1}
\begin{align}
 \bigl(\E |\CI^\eps_1(t_1)|^p\bigr)^{\frac{2}{p}} &\lesssim \eps \sum_{y \in \eps\ZZ} \int_0^{t_1} \bigl|p^\eps_{t_1-r}(x-y) - p^\eps_{t_2-r}(x-y)\bigr|^2 d r\\
&\qquad + \eps^{3} \sup_{r \in [0,t_1]} \sup_{y \in \eps\ZZ} \bigl|p^\eps_{t_1-r}(x-y) - p^\eps_{t_2-r}(x-y)\bigr|^2.\label{eq:bound_21}
\end{align}
Using \eqref{eq:Est1}, the last term in \eqref{eq:bound_21} is bounded by $\eps^{1 - 2(\eta + \kappa)} (t_2 - t_1)^{\eta + \kappa}$, where the power of $\eps$ is positive if $\eta < \frac{1}{2}$ and $\kappa > 0$ is sufficiently small. Furthermore, using \eqref{eq:Est1}, the first term in \eqref{eq:bound_21} is bounded by a multiple of
\begin{align}
\eps \sum_{y \in \eps\ZZ} &\int_0^{t_1} \bigl|p^\eps_{t_1-r}(x-y) - p^\eps_{t_2-r}(x-y)\bigr| \frac{(t_2 - t_1)^{\eta + \kappa}}{(t_1-r)^{(1 + 2 \eta + 2\kappa)/2}} d r \\
&\lesssim \eps \sum_{y \in \eps\ZZ} \int_0^{t_1} \bigl( p^\eps_{t_1-r}(x-y) + p^\eps_{t_2-r}(x-y)\bigr) \frac{(t_2 - t_1)^{\eta + \kappa}}{(t_1-r)^{(1 + 2 \eta + 2\kappa)/2}} d r\\
&\lesssim \int_0^{t_1} \frac{(t_2 - t_1)^{\eta + \kappa}}{(t_1-r)^{(1 + 2 \eta + 2\kappa)/2}} d r \lesssim (t_2 - t_1)^{\eta + \kappa},
\end{align}
for $\eta < \frac{1}{2}$ and $\kappa > 0$ sufficiently small. Here, as before we summed up the heat kernels to $1$.

Similarly, we apply the Burkholder-Davis-Gundy inequality (Lemmas~\ref{lem:BDG}) to $\CI^\eps_2$ and obtain
\begin{align}
  \bigl(\E |\CI^\eps_2(t_2)|^p\bigr)^{\frac{2}{p}} \lesssim \eps \sum_{y \in \eps\ZZ} \int_{t_1}^{t_2} p^\eps_{t_2-r}(x-y)^2 d r + \eps^{3} \sup_{r \in [t_1, t_2]} \sup_{y \in \eps\ZZ} p^\eps_{t_2-r}(x-y)^2.\label{eq:bound_22}
\end{align}
The last term in \eqref{eq:bound_22} can be bounded using \eqref{eq:Est5} by a multiple of $\eps$. Applying \eqref{eq:Est5} to the first term in \eqref{eq:bound_22} and using the fact that the heat kernel is summed up to $1$ in the spatial variable, we can bound it by a multiple of
\begin{align}
 \eps \sum_{y \in \eps\ZZ} \int_{t_1}^{t_2} \frac{p^\eps_{t_2-r}(x-y)}{\sqrt{t_2-r}} d r = \int_{t_1}^{t_2} \frac{d r}{\sqrt{t_2-r}} = 2 \sqrt{t_2 - t_1} \leq 2 (t_2 - t_1)^{\eta + \kappa},
\end{align}
where $\eta + \kappa \leq \frac{1}{2}$ and where we have used $t_2 - t_1 \leq 1$.

Combining the derived bounds on $\CI^\eps_1$ and $\CI^\eps_2$ with the Minkowski inequality, we obtain
\begin{align}
\left(\E |\hat X^\eps(t_2,x) - \hat X^\eps(t_1,x)|^p\right)^{\frac{1}{p}} &\leq \bigl(\E |\CI^\eps_1(t_1)|^p\bigr)^{\frac{1}{p}} + \bigl(\E |\CI^\eps_2(t_2)|^p\bigr)^{\frac{1}{p}} \\
&\leq C\bigl( |t_1 - t_2|^{\eta + \kappa} + \eps\bigr) \leq 2 C (\sqrt{|t_2 - t_1|} \vee \eps)^{2(\eta + \kappa)}.\label{eq:bound_3}
\end{align}
The required bound \eqref{eq:X_bound} now follows from \eqref{eq:bound_11}, \eqref{eq:bound_23}, \eqref{eq:bound_3} and the Kolmogorov continuity criterion \cite{Kallenberg}.
\end{proof}

Furthermore, we can derive a bound on the process $Z^\eps$ in a certain space of distributions.

\begin{lemma}\label{lem:Z_bound}
For any $\kappa_\star \in (0, \frac{1}{2})$, $\kappa > 0$ and $p \geq 1$ there is a constant $C = C(\kappa_\star, \kappa, p) > 0$ such that for any $T > 0$ the process $Z^\eps$, defined in \eqref{eq:Z_def}, satisfies the bound
\begin{align}\label{eq:Z_bound}
\bigl(\E \| Z^\eps\|^p_{\CC^{-1/2 + \kappa_\star}_{T, \eps}}\bigr)^{\frac{1}{p}} \leq C \eps^{-1/2 - \kappa_\star - \kappa},
\end{align}
where the norm is defined in Section~\ref{sec:notation}.
\end{lemma}

\begin{proof}
The process $Z^\eps$ can be written as $Z^\eps_t(x) = \hat Z^\eps_t(t,x)$, where
\begin{align}\label{eq:Z_hat_def}
\hat Z^\eps_\tau(t, x) := \nablae^- \hat X^\eps_\tau(t, x) \nablae^+ \hat X^\eps_\tau(t, x) - \fC^\eps_\tau(t, x),
\end{align}
and $\fC^\eps$ has been defined in \eqref{eq:C_curly}. The definition of the predictable quadratic covariation \cite[Ch.~I.4]{JS03} implies that $\tau \mapsto \hat Z^\eps_\tau(t, x)$ is a martingale, for $\tau \in [0, t]$. Moreover, we can use the It\^{o} formula to write
\begin{align}\label{eq:Z_Ito}
 	\hat Z^\eps_\tau = \int_{0}^\tau \nablae^{-} \hat X^\eps_{r-} d\, \nablae^+ \hat X^\eps_{r} + \int_{0}^\tau \nablae^{+} \hat X^\eps_{r-} d\, \nablae^- \hat X^\eps_{r} + D^\eps_{\!\tau},
\end{align}
where the process $D^\eps_{\!\tau}(t,x)$ is a martingale for $\tau \in [0, t]$, defined by
\begin{align}\label{e:D_def}
 D^\eps_{\!\tau}(t,x) := \bigl[ \nablae^{-} \hat X^\eps_{\bigcdot}(t, x), \nablae^{+} \hat X^\eps_{\bigcdot}(t, x)\bigr]_\tau - \langle \nablae^{-} \hat X^\eps_{\bigcdot}(t, x), \nablae^{+} \hat X^\eps_{\bigcdot}(t, x)\rangle_\tau.
 \end{align}
 Here, as in Section \ref{sec:notation}, $[\bigcdot, \bigcdot]$ refers to the quadratic covariation \cite[Thm. I.4.52]{JS03}. We will prove \eqref{eq:Z_bound} by bounding each term in \eqref{eq:Z_Ito} separately.

 We start with analysis of the first integral in \eqref{eq:Z_Ito}, which we denote by $\CY^{\eps}_\tau(t,x)$. For a rescaled test function $\phi_z^\lambda$, as in \eqref{eq:norm_eps}, we use the notation of Appendix~\ref{sec:integrals} and write $\llangle \CY^{\eps}_\tau(t), \phi_z^\lambda \rrangle_\eps = \I^\eps_2 F^\eps_{\!\lambda}(\tau)$, where $\I^\eps_2$ refers to the second order stochastic integral with the kernel
\begin{align}\label{eq:kernel_F}
F^\eps_{\!\lambda}(r_1, r_2; y_1, y_2) := \eps \sum_{ x \in \eps\ZZ} \phi^\lambda_z( x) \nablae^{-} p^\eps_{t-r_1}( x - y_1) \nablae^{+} p^\eps_{t-r_2}( x - y_2).
\end{align}
Applying Lemma~\ref{lem:integral_bound}, we obtain the moment bound
\begin{align}
\bigg(\E \Bigl[ \sup_{0 \leq \tau \leq t} \big| \I^\eps_2 F^\eps_{\!\lambda}(\tau) \big|^p\Bigr] \bigg)^{\frac{2}{p}} &\lesssim \SE^{(1)}_{\eps, \lambda} + \SE^{(2)}_{\eps, \lambda} + \SE^{(3)}_{\eps, \lambda},\label{eq:moment1}
\end{align}
 where the terms on the r.h.s. are given by
\begin{subequations}
\begin{align}
\SE^{(1)}_{\eps, \lambda} &:= \eps^{2} \sum_{y_1, y_2 \in \eps\ZZ} \int_{0}^t\! \int_{0}^{r_{1}} F^\eps_{\!\lambda}(r_1, r_2; y_1, y_2)^2 \;  d r_2\, d r_1,\label{eq:E1}\\
\SE^{(2)}_{\eps, \lambda} &:= \eps^{4} \sum_{y_2 \in \eps\ZZ} \int_{0}^t \bigg( \sup_{\substack{r_{1} \in [0, r_2]\\y_{1} \in \eps\ZZ}}  \big| F^\eps_{\!\lambda}(r_1, r_2; y_1, y_2) \big|^p \bigg)^{\frac{2}{p}} d r_2,\label{eq:E2}\\
\SE^{(3)}_{\eps, \lambda} &:= \eps^{3} \bigg( \E \Bigl[ \sup_{\substack{r_{2} \in [0, t]\\y_{2} \in \eps\ZZ}}  \big| (\I^\eps_{1} F^\eps_{\!\lambda})(r_{2}; y_{2}) \big|^p\Bigr] \bigg)^{\frac{2}{p}},\label{eq:E3}
\end{align}
\end{subequations}
because the jump sizes of all martingales are $\sqrt \eps$ and their predictable quadratic covariation \eqref{e:Ceps} is bounded.

Our aim is to bound the three terms appearing on the r.h.s. of \eqref{eq:moment1}. For this, we use the results provided in Appendix~\ref{sec:singular}, and derive the required bounds from measuring ``strength" of singularities of the involved kernels. We start with estimating the first term \eqref{eq:E1}. We may write this term as
\begin{align}\label{eq:double_sum_bound}
\SE^{(1)}_{\eps, \lambda} := \eps^2 \sum_{ x_1,  x_2 \in \eps\ZZ} \phi^\lambda_z( x_1)\phi^\lambda_z( x_2) \CK^{(\eps,1)}_t( x_2 -  x_1),
\end{align}
where the kernel $\CK^{(\eps,1)}_t$ is given by
\begin{align}
\CK^{(\eps,1)}_t( x_2 -  x_1) := \eps^{2} \sum_{y_1, y_2 \in \eps\ZZ} \int_{r_1 = 0}^t\!& \int_{r_2 = 0}^{r_{1}} \nablae^{-} p^\eps_{t-r_1}( x_1 - y_1) \nablae^{+} p^\eps_{t-r_2}( x_1 - y_2)\\
&\times \nablae^{-} p^\eps_{t-r_1}( x_2 - y_1) \nablae^{+} p^\eps_{t-r_2}( x_2 - y_2) \;  d r_2\, d r_1.
\end{align}
%Here, we have extended the integral to $[0, \infty)$ by postulating $p^\eps_t(x) = 0$, for $t < 0$.
From Lemmas~\ref{lem:heat} and \ref{lem:singular} we obtain $\CK^{(\eps,1)} \in \CS_\eps^2$, where we use the notation of Appendix~\ref{sec:singular}. We cannot apply Lemma~\ref{lem:test} directly, because the strength of singularity is too high. To overcome this difficulty, we simply notice that $\eps^{1 + \zeta} \CK^{(\eps,1)} \in \CS_\eps^{1 - \zeta}$, for any $\zeta > 0$. Taking $\zeta = 2(\kappa_\star + \kappa)$, Lemma~\ref{lem:test} yields
\begin{equation}
|\SE^{(1)}_{\eps, \lambda}| \lesssim \eps^{-1 - 2(\kappa_\star + \kappa)} \lambda^{ - 1 + 2(\kappa_\star + \kappa)},\label{eq:E1_bound}
\end{equation}
for $\lambda \geq \eps$, and for any $\kappa_\star > 0$ and $\kappa > 0$.

Now, we will bound the second term \eqref{eq:E2}. From the definition \eqref{eq:kernel_F} we obtain
\begin{align}
\SE^{(2)}_{\eps, \lambda} \leq \eps^{4} \sum_{y_2 \in \eps\ZZ} \int_{r_2=0}^t \bigg( \sup_{\substack{r_{1} \in [0, r_2]\\y_{1} \in \eps\ZZ}} | \nablae^{-} p^\eps_{t-r_1}(y_1) | \bigg)^{2} \bigg(\eps \sum_{ x \in \eps\ZZ} \phi^\lambda_z( x) \nablae^{+} p^\eps_{t-r_2}( x - y_2)\bigg)^2 d r_2.
\end{align}
Using \eqref{eq:Est5}, we can simply bound $| \nablae^{-} p^\eps_{r}(y) | \lesssim \eps^{-2}$, which yields
\begin{align}
\SE^{(2)}_{\eps, \lambda} &\lesssim \sum_{y \in \eps\ZZ} \int_{0}^t \bigg(\eps \sum_{ x \in \eps\ZZ} |\phi^\lambda_z( x)| |\nablae^{+} p^\eps_{t-r}( x - y)|\bigg)^2 d r\\
&\leq \eps^2 \sum_{ x_1,  x_2 \in \eps\ZZ} |\phi^\lambda_z( x_1)| |\phi^\lambda_z( x_2)| \CK^{(\eps,2)}_t( x_2 -  x_1),
\end{align}
with the kernel $\CK^{(\eps,2)}$, given by
\begin{align}
\CK^{(\eps,2)}_t( x_2 -  x_1) := \sum_{y \in \eps\ZZ} \int_{0}^t\! |\nablae^{+} p^\eps_{t-r}( x_1 - y) \nablae^{+} p^\eps_{t-r}( x_2 - y)| \, d r.\label{eq:K2}
\end{align}
Lemmas~\ref{lem:heat} and \ref{lem:singular} yield $\eps \CK^{(\eps,2)} \in \CS_\eps^1$. Using the same trick as in \eqref{eq:E1_bound} to ``improve" singularity, Lemma~\ref{lem:test} yields
\begin{equation}
\SE^{(2)}_{\eps, \lambda} \lesssim \eps^{- 1 - 2(\kappa_\star + \kappa)} \lambda^{ - 1 + 2(\kappa_\star + \kappa)},\label{eq:E2_bound}
\end{equation}
for $\lambda \geq \eps$, and for any $\kappa_\star > 0$ and $\kappa > 0$.

In order to bound the last term \eqref{eq:E3}, we use the definition \eqref{eq:kernel_F} and we bound as before $| \nablae^{+} p^\eps_{r}(y) | \lesssim \eps^{-2}$, which yields
\begin{equation}
 \big| (\I^\eps_{1} F^\eps_{\!\lambda})(r_{2}; y_{2}) \big| \lesssim \eps^{-2} \big| (\I^\eps_{1} \hat F^\eps_{\!\lambda})(r_2; y_2) \big|,
\end{equation}
with a new kernel $\hat F^\eps_{\!\lambda}$, given by
\begin{align}\label{eq:kernel_hatF}
\hat F^\eps_{\!\lambda}(r; y) := \eps \sum_{ x \in \eps\ZZ} |\phi^\lambda_z( x)| |\nablae^{-} p^\eps_{t-r}( x - y)|.
\end{align}
Applying this bound and then Lemma~\ref{lem:integral_bound}, we obtain
\begin{align}
\SE^{(3)}_{\eps, \lambda} \lesssim \eps^{-1} \bigg( \E \Big[ \sup_{\substack{r_2 \in [0,t] \\ y_{2} \in \eps\ZZ}} \big| (\I^\eps_{1} \hat F^\eps_{\!\lambda})(r_2; y_2) \big|^p\Big] \bigg)^{\frac{2}{p}} \lesssim \eps^{-1} \bigl(\SE^{(3,1)}_{\eps, \lambda} + \SE^{(3,2)}_{\eps, \lambda}\bigr),\label{eq:E3_bound_interm}
\end{align}
where the terms on the r.h.s. are given by
\begin{align}
\SE^{(3,1)}_{\eps, \lambda} := \eps \sum_{y \in \eps\ZZ} \int_{0}^t\! \hat F^\eps_{\!\lambda}(r; y)^2\, d r, \qquad \SE^{(3,2)}_{\eps, \lambda} := \eps^{3} \sup_{\substack{r \in [0, t]\\y \in \eps\ZZ}} \big| \hat F^\eps_{\!\lambda}(r; y) \big|^2.\label{eq:hatE2}
\end{align}
The first term $\SE^{(3,1)}_{\eps, \lambda}$ we can write in the following way:
\begin{align}
\SE^{(3,1)}_{\eps, \lambda} = \eps^3 \sum_{ x_1,  x_2 \in \eps\ZZ} |\phi^\lambda_z( x_1)| |\phi^\lambda_z( x_2)| \hat \CK^{\eps}_t( x_2 -  x_1),
\end{align}
where $\hat \CK^{\eps}$ is defined as $\CK^{(\eps,2)}$ in \eqref{eq:K2}, but using the discrete derivative $\nablae^-$ instead of $\nablae^+$. In particular, in the same way as in \eqref{eq:E2_bound} we obtain $\SE^{(3,1)}_{\eps, \lambda} \lesssim \eps^{- 2 (\kappa_\star + \kappa)} \lambda^{ - 1 + 2(\kappa_\star + \kappa)}$, which holds for any $\kappa_\star > 0$ and $\kappa > 0$.

In order to bound the term $\SE^{(3,2)}_{\eps, \lambda}$, we use Lemma~\ref{lem:heat} and get $\eps^{1 + \zeta}\nablae^- p^\eps \in \CS^{1 - \zeta}_\eps$, for any $\zeta > 0$. Using this fact, one can show that $\big| \hat F^\eps_{\!\lambda}(r; y) \big| \lesssim \eps^{-1 - \zeta} \lambda^{-1+\zeta}$, for $\lambda \geq \eps$. Taking $\zeta = 1/2 + \kappa_\star + \kappa$, from \eqref{eq:hatE2} we obtain $\SE^{(3,2)}_{\eps, \lambda} \lesssim \eps^{ - 2(\kappa_\star + \kappa)} \lambda^{ - 1 + 2(\kappa_\star + \kappa)}$. From \eqref{eq:E3_bound_interm} and the above derived bounds we conclude
\begin{equation}
\SE^{(3)}_{\eps, \lambda} \lesssim \eps^{-1 - 2(\kappa_\star + \kappa)} \lambda^{ - 1 + 2(\kappa_\star + \kappa)}.\label{eq:E3_bound}
\end{equation}

Combining \eqref{eq:moment1} with the bounds \eqref{eq:E1_bound}, \eqref{eq:E2_bound} and \eqref{eq:E3_bound}, we obtain
\begin{equation}
\bigg(\E \big| \llangle \CY^{\eps}_t(t), \phi_z^\lambda \rrangle_\eps \big|^p \bigg)^{\frac{2}{p}} \leq \bigg(\E \sup_{0 \leq \tau \leq t} \big| \llangle \CY^{\eps}_\tau(t), \phi_z^\lambda \rrangle_\eps \big|^p \bigg)^{\frac{2}{p}} \lesssim \eps^{-1 - 2(\kappa_\star + \kappa)} \lambda^{ - 1 + 2(\kappa_\star + \kappa)},
\end{equation}
which holds for any $\kappa_\star > 0$ and $\kappa > 0$. In a similar way we can prove the bound
\begin{equation}
\bigg(\E \big| \llangle \CY^{\eps}_t(t) - \CY^{\eps}_s(s), \phi_z^\lambda \rrangle_\eps \big|^p \bigg)^{\frac{2}{p}} \lesssim |t-s|^{\kappa/2} \eps^{-1 - 2(\kappa_\star + \kappa)} \lambda^{ - 1 + 2 \kappa_\star + \kappa},
\end{equation}
and by the Kolmogorov continuity criterion we conclude that the process $(t,x) \mapsto \CY^{\eps}_t(t, x)$ satisfies the bound \eqref{eq:Z_bound}. Obviously, the same result holds true for the second term in \eqref{eq:Z_Ito}.

Now, we turn to the martingale $D^\eps_{\!\tau}(t,x)$ in \eqref{eq:Z_Ito}. As before, we take a rescaled test function $\phi^\lambda_z$ and using the notation of Appendix~\ref{sec:integrals} we write $\llangle D^\eps_{\!\tau}(t), \phi^\lambda_z \rrangle_\eps = (\I^\eps_1 \tilde F^\eps_{\!\lambda})(\tau)$, where the stochastic integral is with respect to the martingale $N^\eps_t(x) := [M^\eps(x)]_t - \langle M^\eps(x)\rangle_t$ and where the kernel $\tilde F^\eps_{\!\lambda}(r; y) := \eps F^\eps_{\!\lambda}(r, r; y, y)$ is defined using \eqref{eq:kernel_F}.

Since the martingale $M^\eps_t(x)$ has bounded total variation, the jump size of $N^\eps_t(x)$ at time $t$, if it happens, is equal to $\Delta_t N^\eps(x) = (\Delta_t M^\eps(x))^2 = \eps$, because $M^\eps(x)$ has jumps of size $\sqrt \eps$ (see \cite[Thm. I.4.52]{JS03}). Moreover, the predictable quadratic covariation of $N^\eps_t(x)$ is proportional to $\eps$. Then Lemma~\ref{lem:integral_bound} yields the bound
\begin{equation}
 \bigg(\E \Bigl[\sup_{0 \leq \tau \leq t} \big| (\I^\eps_1 \tilde F^\eps_{\!\lambda})(\tau) \big|^p\Bigr] \bigg)^{\frac{2}{p}} \lesssim \eps^{2} \sum_{y \in \eps\ZZ} \int_{0}^t\! \tilde F^\eps_{\!\lambda}(r; y)^2 d r + \eps^4 \sup_{\substack{r \in [0, t]\\y \in \eps\ZZ}}  \big| \tilde F^\eps_{\!\lambda}(r; y) \big|^2 =: \tilde \SE^{(1)}_{\eps, \lambda} + \tilde \SE^{(2)}_{\eps, \lambda}.\label{eq:bracket_bound}
\end{equation}
We can write the first term in the following way:
\begin{align}
\tilde \SE^{(1)}_{\eps, \lambda} := \eps^{4} \sum_{ x_1, x_2 \in \eps\ZZ} \phi^\lambda_z( x_1)\phi^\lambda_z( x_2) \tilde \CK^{\eps}_t( x_2 -  x_1),
\end{align}
where the kernel $\tilde \CK^{\eps}_t$ is given by
\begin{align}
\tilde \CK^{\eps}_t( x_2 -  x_1) := \eps \sum_{y \in \eps\ZZ} \int_{0}^t\! \nablae^{-} p^\eps_{t-r}( x_1 - y) \nablae^{+} p^\eps_{t-r}( x_1 - y) \nablae^{-} p^\eps_{t-r}( x_2 - y) \nablae^{+} p^\eps_{t-r}( x_2 - y) \,  d r.
\end{align}
In the same way as we did above, we can ``improve" the singularity by multiplying the kernel by a positive power of $\eps$. Then Lemmas~\ref{lem:heat} and \ref{lem:singular} yield $\eps^{2(1 + \zeta)} \tilde \CK^{\eps} \in \CS_\eps^{3 - 2 \zeta}$, for any $\zeta > 0$. Applying Lemma~\ref{lem:test} with $\zeta = 1 + \kappa_\star + \kappa$, we obtain
\begin{equation}
|\tilde \SE^{(1)}_{\eps, \lambda}| \lesssim \eps^{-1 - 2(\kappa_\star + \kappa)} \lambda^{ - 1 + 2(\kappa_\star + \kappa)}.\label{eq:tildeE1_bound}
\end{equation}
Now, we turn to the second term in \eqref{eq:bracket_bound}. One can show that for any $\zeta > 0$ one has $\eps^{2 + \zeta} |\tilde F^\eps_{\!\lambda}(r; y)| \lesssim \lambda^{-1 + \zeta}$, where $\lambda \geq \eps$. Hence, taking $\zeta = 2(\kappa_\star + \kappa)$, we obtain
\begin{equation}
\tilde \SE^{(2)}_{\eps, \lambda} \lesssim \eps^{1 - 2(\kappa_\star + \kappa)} \lambda^{ - 1 + 2(\kappa_\star + \kappa)}.\label{eq:tildeE2_bound}
\end{equation}

Combining \eqref{eq:bracket_bound} with the bounds \eqref{eq:tildeE1_bound} and \eqref{eq:tildeE2_bound}, we obtain
\begin{align}
 \bigg(\E \Bigl[ \big| \llangle D^\eps_{\!t}(t), \phi^\lambda_z \rrangle_\eps \big|^p\Bigr] \bigg)^{\frac{2}{p}} \lesssim \bigg(\E \Bigl[\sup_{0 \leq \tau \leq t} \big| \I^\eps_1 \tilde F^\eps_{\!\lambda}(\tau) \big|^p\Bigr] \bigg)^{\frac{2}{p}} \lesssim \eps^{-1 - 2(\kappa_\star + \kappa)} \lambda^{ - 1 + 2(\kappa_\star + \kappa)}.
\end{align}
In a similar way we can prove the following bound
\begin{equation}
\bigg(\E \Bigl[\big| \llangle D^\eps_{\!t}(t) - D^\eps_{\!s}(s), \phi_z^\lambda \rrangle_\eps \big|^p\Bigr] \bigg)^{\frac{2}{p}} \lesssim |t-s|^{\kappa/2} \eps^{- 1 - 2(\kappa_\star + \kappa)} \lambda^{ - 1 + 2 \kappa_\star + \kappa}.
\end{equation}
The Kolmogorov continuity criterion implies that the process $(t,x) \mapsto D^{\eps}_t(t, x)$ satisfies the bound \eqref{eq:Z_bound}, and this finishes the proof.
\end{proof}

\subsubsection{Convergence of the processes $X^\eps$}

We can prove convergence of the stopped processes $X^\eps_{t \wedge \sigma_{L, \eps}}$, where the stopping time $\sigma_{L, \eps}$ is defined in \eqref{eq:sigma_L}.

\begin{prop}\label{prop:X_eps_converge}
Under the assumptions of Theorem~\ref{thm:tight_intro}, let us extend the processes $t \mapsto X^{\eps}_{t}$, defined in \eqref{eq:X_eps_mild}, piece-wise linearly in the spatial variable to $\RR$. Let furthermore $X$ be the $1$-periodic solution of \eqref{eq:SHEAdditive} with $A = 0$ and $B_T = 1$, and with the initial state $\mathcal{Z}_0$. Finally, let us define the limit $\sigma_{L} := \lim_{\eps \to 0} \sigma_{L, \eps}$ in probability. Then the process $t \mapsto X^{\eps}_{t \wedge \sigma_{L, \eps}}$ converges weakly in $D([0,\infty), \CC(\RR))$ to $X_{t \wedge \sigma_{L}}$, as $\eps \to 0$.
\end{prop}

%We split the proof of this proposition into several steps: first we prove tightness of the processes, and then we identity the limit.
%
%\begin{lemma}\label{lem:X_eps_tight}
% The processes $t \mapsto X^{\eps}_{t \wedge \sigma_{L, \eps}}$, extended piece-wise linearly to $\RR$, are tight in $D([0,T], \CC(\RR))$, and every limiting process is concentrated on $C([0,T], \CC(\RR))$.
%\end{lemma}
%
%\begin{proof}
%This lemma can be proved in the standard way, using the Aldous tightness criterion and moment bounds from Lemma~\ref{lem:X_bound}, see \cite[Prop.~4.9]{Bertini1997}.
%\end{proof}

Tightness of the processes $X^\eps$ are proved in Lemma~\ref{lem:X_bound}, and in order to identify their limit we need the following bound on the bracket process of the martingales $M^\eps$.

\begin{lemma}\label{lem:Z_bound_strong}
Under the assumptions of Theorem~\ref{thm:tight_intro}, let the initial state satisfy the bound $\| \hat s^\eps_0 \|_{\CC^\eta} \leq L$, for a constant $L > 0$. Then the function $C_\eps$, defined in \eqref{e:Ceps}, can be bounded:
\begin{align}\label{eq:C_bound_strong}
\sup_{0 \leq t \leq \sigma_{L, \eps}} \| C_\eps(\hat s^\eps_{t}) - 2 \|_{\CC^{-1/2 + \kappa_\star}_\eps} \leq C \eps^{\frac{1}{2} - \kappa_\star - \hat \kappa} L + C\eps^{3 \eta - 1} t^{- \eta / 2} L^2,
\end{align}
where the stopping time $\sigma_{L, \eps}$ is defined in \eqref{eq:sigma_L}, using the values $\kappa_\star$, $\hat \kappa$ and $\eta$.
\end{lemma}

\begin{proof}
Lemma~\ref{lem:rho_bound} yields $\rho_\eps(\hat s^\eps, x) = 1 - \eps /2 + \tilde \rho_\eps(\hat s^\eps, x)$, where $|\tilde \rho_\eps(\hat s^\eps, x)| \lesssim \eps^2$.
Using the definition \eqref{e:Ceps}, we can write
\begin{align}
C_{\!\eps} (\hat s^\eps_t) - 2 = 2 \bigl(\rho_\eps(\hat s^\eps_t) - 1 \bigr) + 2 \eps^{3} \lambda_\eps(\hat s^\eps_t) \Deltae \hat s^\eps_t + \eps \bigl(\eps - 2 \tilde \rho_\eps(\hat s^\eps_t)\bigr) \nablae^- \hat s^\eps_t\, \nablae^+ \hat s^\eps_t - 2 \eps \nablae^- \hat s^\eps_t\, \nablae^+ \hat s^\eps_t,%\label{eq:C_new}
\end{align}
and we denote these four terms by $C_{\!\eps, 1}$, $C_{\!\eps, 2}$, $C_{\!\eps, 3}$ and $C_{\!\eps, 4}$ respectively.
Lemma~\ref{lem:rho_bound} yields a bound on the first term: $|C_{\!\eps, 1}(\hat s^\eps_t, x)| \lesssim \eps$. Furthermore, the bound $|\lambda_\eps(\hat s^\eps_t, x)| \lesssim \eps^{-1/2}$ yields a bound on the second term: $|C_{\!\eps, 2}(\hat s^\eps_t, x)| \lesssim \sqrt \eps \|\hat s^\eps_t\|_{L^\infty} \lesssim \sqrt \eps L$, where we consider $t \in [0, \sigma_{L, \eps}]$. The estimate $|\nablae \hat s^\eps_t| \leq \eps^{\eta - 1}\|\hat s^\eps_t\|_{\CC^{\eta}}$ yields the bound $|C_{\!\eps, 3}(\hat s^\eps_t, x)| \lesssim \eps^{2 \eta} \|\hat s^\eps_t\|^2_{\CC^{\eta}} \lesssim \eps^{2 \eta} L^2$, for $t \in [0, \sigma_{L, \eps}]$.

Now, we turn to the most complicated term $C_{\!\eps, 4}$. Using \eqref{eq:from_u_to_s}, and replacing the product $\nablae^- \hat X^\eps \nablae^+ \hat X^\eps$ by its renormalized version \eqref{eq:Z_def}, we obtain
\begin{equation}
\nablae^- \hat{s}^\eps_t \nablae^+ \hat{s}^\eps_t = Z^\eps_t + \fC^\eps_t + \Bigl(\bigl(\nablae^- \hat{s}^\eps \nablae^+ \hat{s}^\eps\bigr)_t - \nablae^- \hat X^\eps_t \nablae^+ \hat X^\eps_t\Bigr).\label{eq:C4}
 \end{equation}
 On the time interval $t \in [0, \sigma_{L, \eps}]$, by \eqref{eq:sigma_L} the $(-1/2 + \kappa_\star)$-norm of the first term in \eqref{eq:C4} is bounded by $\eps^{-\frac{1}{2} - \kappa_\star - \hat \kappa} L$. For the last term in \eqref{eq:C4}, we use \eqref{eq:from_u_to_s} and bound the absolute value of the last term in \eqref{eq:C4} by a constant times
\begin{align}
\eps^{3 \eta - 2} \bigl(\|\hat X^\eps_t\|_{\CC^\eta} + \|S^\eps_t \hat s_0\|_{\CC^{\eta}}\bigr) \bigl(\|S^\eps_t \hat s_0\|_{\CC^{2\eta}} + \|u^\eps_t\|_{\CC^{2\eta}} \bigr) + \eps^{4 \eta - 2} \|u^\eps_t\|_{\CC^{2\eta}}^2.
\end{align}
Lemma~\ref{lem:heat_S} yields  $\|S^\eps_t \hat s_0\|_{\CC^{\eta}} \lesssim \| \hat s_0\|_{\CC^{\eta}}$ and  $\|S^\eps_t \hat s_0\|_{\CC^{2\eta}} \lesssim t^{(\eta - 2 \eta) / 2} \| \hat s_0\|_{\CC^{\eta}}$. Moreover, for $t \in [0, \sigma_{L, \eps}]$ the norms of all stochastic processes and of $\hat s_0$ are bounded by $L$. Thus, the last term in \eqref{eq:C4} is bounded by a multiple of $\eps^{3 \eta - 2} t^{(\eta - 2 \eta) / 2} L^2$.

Now we will estimate the remaining term $\fC^\eps_t$ in \eqref{eq:C4}. Define $K^\eps_t(x) := \nablae^- p^\eps_t(x) \nablae^+ p^\eps_t(x)$ (see \eqref{eq:K_kernel}) and use the definition \eqref{eq:C_curly} to write
\begin{align}\label{eq:R_bound}
\fC^\eps_t(x) = \eps \sum_{y \in \eps\ZZ} \int_0^t K^\eps_{t-r}(x-y) C_{\!\eps} (\hat s^\eps_r, y) dr,
\end{align}
where we made use of the function $C_{\!\eps}$ in \eqref{e:Ceps}. Furthermore, rewrite
\begin{align}\label{eq:some_bound_30}
 \fC^\eps_t(x) = 2 \eps \sum_{y \in \eps\ZZ} \int_0^t K^\eps_{t-r}(x-y) dr + \eps \sum_{y \in \eps\ZZ} \int_0^t K^\eps_{t-r}(y) \bigl(C_{\!\eps}(\hat s^\eps_r, x - y) - 2 \bigr) dr.
\end{align}
Applying \eqref{eq:K_kernel_identity}, the first term can be bounded by a constant, uniformly in $\eps$.

Let us denote $\hat C^{\eps}_t := \| C_\eps(\hat s^\eps_{t \wedge \sigma_{L, \eps}}) - 2 \|_{\CC^{-1/2 + \kappa_\star}_\eps}$. Then, combining all these bounds, we obtain
\begin{equation}
\hat C^{\eps}_t \leq \gamma_\eps(t) + \eps^2 \sum_{x \in \eps\ZZ} \int_0^t |K^\eps_{t-r}(x)| \hat C^{\eps}_r dr,\label{eq:hatC_bound}
\end{equation}
for the function
\begin{equation}
\gamma_\eps(t) := c \eps + c \eps^{\frac{1}{2} - \kappa_\star - \hat \kappa} L + c \sqrt \eps L + c \eps^{3 \eta - 1} t^{- \eta / 2} L^2,
\end{equation}
where the constant $c > 0$ is independent of $\eps$, $L$ and $t$. Denoting $\hat C^{\eps}_{[0,t]} := \sup_{r \in [0,t]}\hat C^{\eps}_r$, and using the first bound in \eqref{eq:K_kernel_bounds}, from \eqref{eq:hatC_bound} we get
\begin{equation}
\hat C^{\eps}_{[0,t]}  \leq \gamma_\eps(t) + c_1 \hat C^{\eps}_{[0,t]},
\end{equation}
where $c_1 \in (0,1)$. This yields $\hat C^{\eps}_{[0,t]}  \leq \gamma_\eps(t) / (1 - c_1)$, which is the required result \eqref{eq:C_bound_strong}.
\end{proof}

With these results at hand, we are ready to prove Proposition~\ref{prop:X_eps_converge}.

\begin{proof}[Proof of Proposition~\ref{prop:X_eps_converge}]
Tightness of the processes $X^\eps$ is proved in Lemma~\ref{lem:X_bound}, and it is sufficient to prove that the weak limit of every converging subsequence equals to the solution of \eqref{eq:SHEAdditive} with $A = 0$ and $B_T = \sqrt 2$, and with the initial state $\mathcal{Z}_0$.

With a little abuse of notation, let $X^\eps$ be such convergent sequence, with the limit $\mathcal{Z}$. Then \eqref{eq:X_eps_mild} implies that $X^\eps$ is the solution of
\begin{equation}
d X^\eps_t = \Deltae X^\eps_t dt + d M^\eps_t, \qquad X^\eps_0 = \hat s^\eps_0.
\end{equation}
Similarly to \eqref{e:martingales}, we can associate to this equation two martingales
  \begin{subequations}
   \label{e:martingales2}
  \begin{align}
 \mathfrak{M}^{\eps}_t(\phi) &:=X^\eps_t(\phi) - X^\eps_0(\phi)- \int^{t}_{0} X^\eps_r(\Delta_{\eps} \phi) dr ,\label{eq:DefM2}\\
 \mathfrak{N}^{\eps}_t(\phi) &:= \big(\mathfrak{M}^{\eps}_t(\phi)\big)^2 - \int^{t}_{0} \eps \sum_{y\in \eps \ZZ}\phi(y)^2 d \big\langle M^\eps(y), M^\eps(y) \big\rangle_{r}, &&\label{eq:DefN2}
\end{align}
\end{subequations}
where for a function $\phi\in \mathcal{C}^{\infty}_{b}(\RR)\cap L^{2}(\RR)$ we use the shorthand notation $X^\eps_t(\phi) := \llangle X^{\eps}_t, \phi\rrangle_{\eps}$, and where the bracket process of the martingale $M^\eps$ is given in \eqref{e:Ceps}. Using the a priori bounds \eqref{eq:bounds_eps}, which hold on the time interval $t \in [0,  \sigma_{L, \eps}]$, by analogy to Lemma~\ref{lem:IdLemma} we can prove the $L^1$ convergence of \eqref{eq:DefM2} to
\begin{align}
 \mathfrak{M}_t(\phi) :=\mathcal{Z}_t(\phi) - \mathcal{Z}_0(\phi)- \int^{t}_{0} \mathcal{Z}_r( \phi'') dr. %\mathfrak{N}_t(\phi) &:= \big(\mathfrak{M}_t(\phi)\big)^2 - t \|\phi\|_{L^2}^2.
\end{align}
Moreover, Lemma~\ref{lem:Z_bound_strong} yields the $L^1$ convergence of \eqref{eq:DefN2} to
\begin{align}
\mathfrak{N}_t(\phi) := \big(\mathfrak{M}_t(\phi)\big)^2 - 2 t \|\phi\|_{L^2}^2.
\end{align}
Therefore, the weak limit of $X^\eps$ is the unique solution of the martingale problem associated to \eqref{eq:SHEAdditive} with $A=0$ and $B_T= \sqrt 2$.
\end{proof}

Using Lemma~\ref{lem:Z_bound_strong} we can also bound the function $\fC^\eps_t(x)$, defined in \eqref{eq:C_curly}.

\begin{lemma}\label{lem:C_bound}
For the function $\fC^\eps_t(x)$ (see \eqref{eq:Z_def} and \eqref{eq:C_curly}) and the stopping time $\sigma_{L, \eps}$, defined in \eqref{eq:sigma_L}, the following bound holds:
\begin{align}\label{eq:C_bound}
\sup_{0 \leq t \leq \sigma_{L, \eps}}\| \fC^\eps_t \|_{\CC^{-1/2 + \kappa_\star}_\eps} \leq C \eps^{-\frac{1}{2} - \kappa_\star - \hat \kappa} L + C\eps^{3 \eta - 2} t^{(\eta - 2 \eta) / 2} L^2,
\end{align}
where the constant $C > 0$ is independent of $\eps$ and $t$, and where the values $\kappa_\star$, $\hat \kappa$ and $\eta$ are from the definition of the stopping time $\sigma_{L, \eps}$ in \eqref{eq:sigma_L}.
\end{lemma}

\begin{proof}
Let as before $K^\eps_t(x) := \nablae^- p^\eps_t(x) \nablae^+ p^\eps_t(x)$. %and let us define $\hat \fC^{\eps, \lambda}_t := |\langle \fC^\eps_t, \phi_{\bar x}^\lambda \rangle_\eps|$, for a test function $\phi$.
Then, using \eqref{eq:R_bound}, we obtain
\begin{align}
\fC^\eps_t (x) = 2 \eps \sum_{y \in \eps\ZZ} \int_0^t K^\eps_{t-r}(x-y) dr + \eps \sum_{y \in \eps\ZZ} \int_0^t K^\eps_{t-r}(y) \bigl( C_{\!\eps}(\hat s^\eps_r, x-y) - 2\bigr) dr.
\end{align}
Applying \eqref{eq:K_kernel_identity}, the first term can be bounded by a constant, uniformly in $\eps$. Hence, we get
\begin{align}
\| \fC^\eps_t \|_{\CC^{-1/2 + \kappa_\star}_\eps} \leq c + \eps \sum_{y \in \eps\ZZ} \int_0^t |K^\eps_{t-r}(y)|\, \| C_{\!\eps}(\hat s^\eps_r) - 2 \|_{\CC^{-1/2 + \kappa_\star}_\eps} dr.
\end{align}
The required bound \eqref{eq:C_bound} now follows from \eqref{eq:C_bound_strong} and \eqref{eq:K_kernel_bounds}.
\end{proof}

\subsubsection{Bounds on the mild solution}

In this section we derive bounds on the r.h.s. of \eqref{eq:u_eps_new}. For this, we denote the three terms on the r.h.s. of \eqref{eq:u_eps_new} by $\CI^{\eps}_{1}(t,x)$, $\CI^{\eps}_{2}(t,x)$ and $\CI^{\eps}_{3}(t,x)$ respectively. %To this end, we make our notation shorter by denoting the triplet $\CX^\eps := (X^\eps, u^\eps, Z^\eps)$ and introducing the following $\eps$-dependent norm:
%\begin{align}\label{eq:norm_eps}
%\$ \CX^\eps \$_{T, \eps} := \|X^\eps\|_{\CC^{\eta_\star}_{T}} + \|u^\eps\|_{\CC^{2 \eta}_{T}} + \eps^{-\frac{1}{2} + 2\eta_\star + \kappa_\star} \| Z^\eps\|_{\CC^{2 (\eta_\star-1)}_{T}},
%\end{align}
%where the values $\eta$, $\eta_\star$ and $\kappa_\star$ are as in \eqref{eq:sigma_L} and are fixed. In particular, for $T \leq  \sigma_{L, \eps}$ we have $\$ \CX^\eps \$_{T, \eps} \leq K$.
We start with $\CI^{\eps}_{1}(t,x)$.

%\begin{lemma}\label{lem:term1}
%The function $\CI^{\eps}_{1}(t,x)$, which was just defined, satisfies the bound
%\begin{align}\label{eq:term1}
% \|\CI^{\eps}_{1}\|_{\CC^{\alpha_\star}_T} \leq \| \hat s^\eps_0 \|_{\CC^{\alpha_\star}}.
%\end{align}
%\end{lemma}
%
%\begin{proof}
%The bound \eqref{eq:term1} follows immediately from the properties of the discrete heat semigroup. \km{make a lemma for bounds like this}
%\end{proof}
%
%We denote by $\CI^{\eps}_{2}(t,x)$ the second integral on the right-hand side of \eqref{eq:u_eps_mild}, which involves the function $\hat F_{\!\eps}$. Then the following result holds:

\begin{lemma}\label{lem:term1}
Let $\eta \in (\frac{1}{3}, \frac{1}{2})$ and let $\kappa \in (0, 1 - 2 \eta)$. Then $\CI^{\eps}_{1}(t,x)$ satisfies the bound
\begin{align}\label{eq:term1}
\sup_{0 \leq t \leq T \wedge \sigma_{L, \eps}} \|\CI^{\eps}_{1}(t)\|_{\CC^{2 \eta}} \leq C \bigl(\eps^{\kappa} T^{(1 - \kappa - 2 \eta)/2} + \eps^\eta T^{1 - \eta / 2}\bigr) L,
\end{align}
where the constant $C > 0$ is independent of $\eps$ and $T$.
\end{lemma}

\begin{proof}
We first derive a bound on the function $\hat F_{\!\eps}$, defined in \eqref{eq:hatF}. Using Lemma~\ref{lem:rho_bound} we can write $\hat \rho_\eps(\hat s^\eps, x) = - \eps / 2 + \tilde \rho_\eps(\hat s^\eps, x)$, for a function $\tilde \rho$ satisfying $|\tilde \rho(\hat s^\eps, x)| \lesssim \eps^2$. This allows to decompose $\hat F_{\!\eps} = \hat F_{\!\eps}^{(1)} + \hat F_{\!\eps}^{(2)}$, where $\hat F^{(1)}_\eps(\hat s^\eps, x) := -\eps \Deltae \hat s^\eps(x) / 2$ and $\hat F^{(2)}_\eps(\hat s^\eps, x) := \tilde \rho_\eps(\hat s^\eps, x) \Deltae \hat s^\eps(x)$.
Using the above estimate on $\tilde \rho$ and the definition of the discrete Laplacian, the second function can be simply bounded by
\begin{equation}
|\hat F^{(2)}_\eps(\hat s^\eps_t, x)| \leq |\tilde \rho_\eps(\hat s^\eps, x)| |\Deltae \hat s^\eps_t(x)| \lesssim \sup_{\substack{y \in \RR :\\ |x-y| \leq \eps}} |\hat s^\eps_t(y) - \hat s^\eps_t(x)| \lesssim \eps^\eta \| \hat s^\eps_t \|_{\CC^\eta}.
\end{equation}
 One can see that the function $\hat F_{\!\eps}^{(1)}$ is not suitable for a uniform bound. On the contrary, it can be considered as a discretization of a distribution, whose convolution with a test function has a good bound. More precisely, for a rescaled and recentered function $\phi^\lambda_z$, defined in Section~\ref{sec:notation}, we write
\begin{align}
\llangle \hat F^{(1)}_\eps(\hat s^\eps_t), \phi^\lambda_z \rrangle_\eps = \eps \sum_{y \in \eps\ZZ} \hat F^{(1)}_\eps(\hat s^\eps_t, y) \phi^\lambda_z(y) = -\frac{\eps^{2}}{2} \sum_{y \in \eps\ZZ} \Deltae \hat s^\eps_t(y) \phi^\lambda_z(y).
\end{align}
Applying summation by parts twice, the last expression can be written in the following way: $-\frac{\eps^{2}}{2} \sum_{y \in \eps\ZZ} \hat s^\eps_t(y) \Deltae \phi^\lambda_z(y)$. Thus, using the fact that $\phi$ is supported in a ball of radius $\lambda > 0$, we obtain
\begin{align}
|\llangle \hat F^{(1)}_\eps(\hat s^\eps_t), \phi^\lambda_z \rrangle_\eps| \lesssim \eps^{2} \lambda^{-3} \|\hat s^\eps_t\|_{L^\infty} \sum_{y \in \eps\ZZ} \1{|y - z| \leq \lambda} \lesssim \eps \lambda^{-2} \|\hat s^\eps_t\|_{L^\infty} \lesssim \eps^\kappa \lambda^{-1 - \kappa} \|\hat s^\eps_t\|_{L^\infty}\;,
\end{align}
for $\kappa \in [0,1]$ and $\lambda \geq \eps$. To estimate the sum in the second bound, we used the fact the number of terms in the sum is proportional to $\lambda/\eps$.

Using the operator norm $\Vert \bigcdot \Vert_{V \rightarrow W}$ for two spaces $V$ and $W$, introduced in Section~\ref{sec:notation}, we get
\begin{equation}
\|\CI^{\eps}_{1}(t)\|_{\CC^{2\eta}} \lesssim \int_0^t \!\!\| S^\eps_{t-r} \|_{\CC_\eps^{-1 - \kappa} \rightarrow \CC^{2\eta}} \| \hat F^{(1)}_\eps(\hat s^\eps_r)\|_{\CC_\eps^{-1 - \kappa}} dr + \int_0^t\!\! \| S^\eps_{t-r} \|_{L^\infty \rightarrow \CC^{2\eta}} \| \hat F^{(2)}_\eps(\hat s^\eps_r)\|_{L^\infty} dr.
\end{equation}
Furthermore, combining the derived bounds on the functions $\hat F_{\!\eps}^{(1)}$ and $\hat F_{\!\eps}^{(2)}$ with Lemma~\ref{lem:heat_S} yields
\begin{equation}
\|\CI^{\eps}_{1}(t)\|_{\CC^{2\eta}} \lesssim \int_0^t \eps^\kappa (t-r)^{- (1 + \kappa + 2\eta) / 2} \|\hat s^\eps_r\|_{L^\infty} dr + \int_0^t \eps^\eta (t-r)^{ - \eta} \| \hat s^\eps_r \|_{\CC^\eta} dr.
%&\lesssim \bigl( \eps^{{\eta_\star} + 1/2} (\sqrt{t} \vee \eps) + \eps^{2 {\eta_\star} - 1/2} t \bigr) \| \hat s^\eps\|^2_{\CC^{\eta_\star}_T},
\end{equation}
We use the bounds $\| \hat s^\eps_r\|_{L^\infty} \leq \| \hat s^\eps_r\|_{\CC^{\eta}}$ and $\| \hat s^\eps_r\|_{\CC^{\eta}} \leq \| X^\eps_r\|_{\CC^{\eta}} + \| u^\eps_r\|_{\CC^{\eta}}$. Then combining the last two bounds, we conclude
\begin{align}
\|\CI^{\eps}_{1}(t)\|_{\CC^{2\eta}} \lesssim \bigl(\eps^{\kappa} t^{(1 - \kappa - 2 \eta)/2} + \eps^\eta t^{1 - \eta / 2} \bigr) L,
\end{align}
which holds as soon as $\kappa < 1 - 2 \eta$ and $\eta < 1$. From this the required bound \eqref{eq:term1} follows.
\end{proof}

\begin{lemma}\label{lem:term2}
Let the value $\eta$ be as in Theorem~\ref{thm:tight_intro}, and let the value $\kappa_\star$ in \eqref{eq:sigma_L} satisfies $\kappa_\star > \frac{1}{2} - \eta$. Then $\CI^{\eps}_{2}(t,x)$ satisfies the bound
\begin{align}\label{eq:term2}
\sup_{0 \leq t \leq T \wedge \sigma_{L, \eps}} \|\CI^{\eps}_{2}(t)\|_{\CC^{2 \eta}} \leq C T^{\hat \beta} \eps^{\beta} L (1 + L)^2,
\end{align}
for some constant $C > 0$ independent of $\eps$ and $T$, for some value $\hat \beta > 0$, depending on $\eta$, $\eta$, $\kappa_\star$ and $\hat \kappa$, and for $\beta = (\frac{1}{2} - \kappa_\star - \hat \kappa) \wedge (3 \eta - 1) \wedge (1 - \gamma - 2\kappa_\star - \hat \kappa) \wedge (2 \eta - \gamma)$, where $\gamma$ is from Assumption~\ref{a:f}.
\end{lemma}

\begin{proof}
Using the definition \eqref{eq:F_tilde_def} and Lemma~\ref{lem:rho_bound}, we can write $\tilde F_{\!\eps} = \tilde F_{\!\eps}^{(1)} + \tilde F_{\!\eps}^{(2)}$, where $\tilde F_{\!\eps}^{(1)}(t,x) := -\frac{\eps \fa \hat s^\eps_t(x)}{4} \fC^\eps_t(x)$ and $\tilde F_{\!\eps}^{(2)}(t,x) := \eps \hat \lambda_\eps(\hat s^\eps_t, x) \fC^\eps_t(x)$, where the following bound holds
\begin{equation}\label{eq:lambda_bound}
|\hat \lambda_\eps(\hat s^\eps_t, x) | \lesssim \eps^{1 - \gamma}(1 + \sqrt \eps |\hat s^\eps_t(x)|)^\gamma.
\end{equation}

We start with estimating the function $\tilde F_{\!\eps}^{(1)}$. If we chose $\kappa_\star > \frac{1}{2} - \eta$, then applying Lemmas~\ref{lem:product} and \ref{lem:C_bound} we obtain
\begin{equation}
\| \tilde F_{\!\eps}^{(1)}(t) \|_{\CC^{-1/2 + \kappa_\star}_{\eps}} \lesssim \eps \| \hat s^\eps_t \|_{\CC^\eta} \| \fC^\eps_t \|_{\CC^{-1/2 + \kappa_\star}_{\eps}} \lesssim \eps^{\frac{1}{2} - \kappa_\star - \hat \kappa} L^2 + \eps^{3 \eta - 1} t^{- \eta / 2} L^3,\label{eq:F1}
\end{equation}
which holds for $t \in [0, \sigma_{L, \eps}]$, and where we have estimated $\| \hat s^\eps_t\|_{\CC^{\eta}} \leq \| X^\eps_t\|_{\CC^{\eta}} + \| u^\eps_t\|_{\CC^{\eta}} \leq 2 L$.

Now, we turn to the function $\tilde F_{\!\eps}^{(2)}$. The bound \eqref{eq:lambda_bound} yields
\begin{equation}
|\tilde F_{\!\eps}^{(2)}(t,x)| \leq \eps |\hat \lambda_\eps(\hat s^\eps_t, x)| |\fC^\eps_t(x)| \lesssim \eps^{2 - \gamma}(1 + \sqrt \eps |\hat s^\eps_t(x)|)^\gamma |\fC^\eps_t(x)|.
\end{equation}
Furthermore, from the identity \eqref{eq:Z_def} we obtain
\begin{align}
|\fC^\eps_t(x)| \leq |Z^\eps_t(x)| + |\nablae^- \hat X^\eps_t(x)| |\nablae^+ \hat X^\eps_t(x)| &\leq \| Z^\eps_t \|_{L^\infty} + \eps^{-2} \sup_{\substack{y \in \RR : \\ |x-y| \leq \eps}} |\hat X^\eps_t(x) - \hat X^\eps_t(y)|^2\\
&\leq \| Z^\eps_t \|_{L^\infty} + \eps^{2 (\eta-1)} \| \hat X^\eps_t \|^2_{\CC^\eta},
\end{align}
where in the last line we used the definition of the H\"{o}lder norm. Hence, we have
\begin{equation}
|\tilde F_{\!\eps}^{(2)}(t,x)| \leq \eps^{2 - \gamma}(1 + \sqrt \eps |\hat s^\eps_t(x)|)^\gamma \bigl(\| Z^\eps_t \|_{L^\infty} + \eps^{2 (\eta-1)} \| \hat X^\eps_t \|^2_{\CC^\eta}\bigr).
\end{equation}
On the time interval $t \in [0, \sigma_{L, \eps}]$ we have $\| \hat X^\eps_t \|_{\CC^\eta} \leq L$ and  $|\hat s^\eps_t(x)| \leq \leq \| X^\eps_t\|_{\CC^{\eta}} + \| u^\eps_t\|_{\CC^{\eta}} \leq 2 L$. Moreover, the definition \eqref{eq:norm_eps} yields the bound
\begin{equation}
\| Z^\eps_t \|_{L^\infty} \leq \eps^{-1/2 - \kappa_\star} \| Z^\eps_t \|_{\CC^{-1/2 + \kappa_\star}_{\eps}} \leq  \eps^{-1 - 2\kappa_\star - \hat \kappa} L.
\end{equation}
Hence, for the function $\tilde F_{\!\eps}^{(2)}$ we have the following bound
\begin{equation}\label{eq:F2}
|\tilde F_{\!\eps}^{(2)}(t,x)| \leq  \eps^{2 - \gamma}L (1 + \sqrt \eps L)^\gamma \bigl(\eps^{-1 - 2\kappa_\star - \hat \kappa} + \eps^{2(\eta-1)} L\bigr).
\end{equation}

 Combining the derived bounds \eqref{eq:F1} and \eqref{eq:F2}, we obtain
\begin{align}
\| \tilde F_{\!\eps}(t) \|_{\CC^{-1/2 + \kappa_\star}_{\eps}} &\lesssim \eps^{\frac{1}{2} - \kappa_\star - \hat \kappa} L^2 + \eps^{3 \eta - 1} t^{- \eta / 2} L^3 + \eps^{2 - \gamma}L (1 + \sqrt \eps L)^\gamma \bigl(\eps^{-1 - 2\kappa_\star - \hat \kappa} + \eps^{2(\eta-1)} L\bigr)\\
&\lesssim \eps^\beta t^{- \eta / 2} L (1 + L)^2,
\end{align}
where $\beta = (\frac{1}{2} - \kappa_\star - \hat \kappa) \wedge (3 \eta - 1) \wedge (1 - \gamma - 2\kappa_\star - \hat \kappa) \wedge (2 \eta - \gamma)$. Then Lemma~\ref{lem:heat_S} yields
\begin{align}
\|\CI^{\eps}_{2}(t)\|_{\CC^{2\eta}} &\lesssim \int_0^t \| S^\eps_{t-r} \|_{\CC_\eps^{-1/2 + \kappa_\star} \rightarrow \CC^{2\eta}} \| \tilde F_{\!\eps}(r, \bigcdot)\|_{\CC_\eps^{-1/2 + \kappa_\star}} dr \\
&\lesssim \eps^\beta L (1 + L)^2 \int_0^t (t-r)^{- (1/2 - \kappa_\star + 2\eta) / 2} r^{- \eta / 2} dr.
\end{align}
Since $(1/2 - \kappa_\star + 2\eta) / 2 < 1$ and $-\eta / 2 > -1$, the last integral is of order $t^{\hat \beta}$, for some $\hat \beta > 0$. This is exactly the required result \eqref{eq:term2}.
\end{proof}

To bound $\CI^{\eps}_{3}(t,x)$ we will compare it to
\begin{equation}\label{eq:I_hat}
\hat \CI^{\eps}_{3}(t,x) := - \frac{\fa}{4} \int_0^t (S_{t-r}^\eps \hat s^\eps_r)(x) dr.
\end{equation}

\begin{lemma}\label{lem:term3}
Let the value $\eta$ be as in Theorem~\ref{thm:tight_intro}. Then $\CI^{\eps}_{3}(t,x)$ and $\hat \CI^{\eps}_{3}(t,x)$ satisfy the bounds
\begin{subequations}
\begin{align}
\sup_{0 \leq t \leq T \wedge \sigma_{L, \eps}} \|\CI^{\eps}_{3}(t)\|_{\CC^{2 \eta}} &\leq C T^{\hat \beta}( L + \eps^\beta (1 + L)^3),\label{eq:term3_1}\\
\sup_{0 \leq t \leq T \wedge \sigma_{L, \eps}} \|(\CI^{\eps}_{3} - \hat \CI^{\eps}_{3})(t)\|_{\CC^{2 \eta}} &\leq C T^{\hat \beta} \eps^\beta (1 + L)^3,\label{eq:term3_2}
\end{align}
\end{subequations}
for some constant $C > 0$ independent of $\eps$ and $T$, for some $\hat \beta > 0$, depending on $\eta$ and $\kappa_\star$, and for the value  $\beta = (2 \eta - \gamma) \wedge (3 \eta - 1) \wedge (1/2 - \kappa_\star - \hat \kappa)$, where $\gamma$ is from Assumption~\ref{a:f}.
\end{lemma}

\begin{proof}
 Before deriving bounds on $\CI^{\eps}_{3}$, we estimate the function inside the integral in the definition of $\CI^{\eps}_{3}$. Using the expression \eqref{eq:F_tilde_def}, we can write this function as
\begin{align}
F_{\!\eps}(X^\eps, Z^\eps, u^\eps) + \frac{\fa \hat s^\eps}{4} = \hat \lambda_\eps(\hat s^\eps) \big( 1 - \eps \nablae^- \hat s^\eps \diamond \nablae^+ \hat s^\eps\big) + \frac{\eps \fa}{4} \hat s^\eps \bigl(\nablae^- \hat s^\eps \diamond \nablae^+ \hat s^\eps\bigr),
\end{align}
and we denote the two terms on the r.h.s. by $F_{\!\eps}^{(1)}$ and $F_{\!\eps}^{(2)}$ respectively. For the first term, the simple bound $|\nablae^\pm \hat s^\eps_t(x)| \leq \eps^{\eta - 1} \| \hat s^\eps_t \|_{\CC^\eta}$ and the bound on the error term $\hat \lambda_\eps$, provided in Lemma~\ref{lem:rho_bound}, yields
\begin{align}
|F_{\!\eps}^{(1)}(t, x)| \lesssim \eps^{1 - \gamma} \bigl(1 + \sqrt \eps \|\hat s^\eps_t\|_{L^\infty}\bigr)^\gamma \big( 1 + \eps^{2 \eta - 1} \|\hat s^\eps_t\|_{\CC^{\eta}}^2\big) \lesssim \eps^{1 - \gamma} \bigl(1 + \sqrt \eps L\bigr)^\gamma \big( 1 + \eps^{2 \eta - 1} L^2\big),
\end{align}
where the value $\gamma$ is from Assumption~\ref{a:f}, and where we consider $t \in [0,\sigma_{L, \eps}]$. Here, we made use of $\| \hat s^\eps \|_{\CC^{\eta}_t} \leq \| X^\eps \|_{\CC^{\eta}_t} + \| u^\eps \|_{\CC^{2\eta}_t}$.

Now, we turn to the function $F_{\!\eps}^{(2)}$. Using the definition of the Wick-type product \eqref{eq:wick} we can write $F_{\!\eps}^{(2)} = F_{\!\eps}^{(3)} + F_{\!\eps}^{(4)}$, where
\begin{align}
F_{\!\eps}^{(3)} := \frac{\eps \fa}{4} \hat s^\eps \left(\nablae^- X^\eps\, \nablae^+ u^\eps + \nablae^- u^\eps\, \nablae^+ \hat s^\eps\right), \qquad F_{\!\eps}^{(4)} := \frac{\eps \fa}{4}\hat s^\eps Z^\eps,
\end{align}
where we use the process $Z^\eps$, defined in \eqref{eq:Z_def}. For the function $F_{\!\eps}^{(3)}$ we have the bound
\begin{align}
|F_{\!\eps}^{(3)}(t,x)| \lesssim \|\hat s^\eps_t\|_{L^\infty} \left( \eps^{3 \eta - 1} \|X^\eps_t\|_{\CC^{\eta}} \|u^\eps_t\|_{\CC^{2 \eta}} + \eps^{4 \eta - 1} \|u^\eps_t\|^2_{\CC^{2 \eta}} \right) \lesssim \eps^{3 \eta - 1} L^3,
\end{align}
on the time interval $t \in [0,\sigma_{L, \eps}]$.
%In order to derive a bound on the term $F_{\!\eps}^{(5)}$, we take any $\kappa \in (0,\eta)$ and write $F_{\!\eps}^{(5)} := \eps^\kappa\hat s^\eps (\eps^{1 - \kappa}Z^\eps) / 4$. By assumption, for every fixed $t \geq 0$ we have $s^\eps_t \in \CC^\alpha$. Moreover, Lemma~\ref{lem:Z_bound} yields $(\eps^{1 - \kappa}Z^\eps_t) \in \CC^{-\kappa}$, where the positive power of $\eps$ improves regularity of $Z^\eps_t$. Thus, applying Lemmas~\ref{lem:product} and \ref{lem:Z_bound} we obtain
For the term $F_{\!\eps}^{(4)}$, we use Lemma~\ref{lem:product} and obtain
\begin{align}
\| F_{\!\eps}^{(4)}(t) \|_{\CC^{-1/2 + \kappa_\star}_{\eps}} \lesssim \eps \| \hat s^\eps_t \|_{\CC^\eta} \| Z^\eps_t \|_{\CC^{-1/2 + \kappa_\star}_{\eps}} \lesssim \eps^{1/2 - \kappa_\star - \hat \kappa} L^2,
\end{align}
which holds on the time interval $t \in [0,\sigma_{L, \eps}]$.

Combining these bounds on the functions $F_{\!\eps}^{(1)}$ and $F_{\!\eps}^{(2)}$, we conclude
\begin{align}
\| (F_{\!\eps} + \fa \hat s^\eps / 4)(t) \|_{\CC^{-1/2 + \kappa_\star}_{\eps}} \leq \| F_{\!\eps}^{(1)}(t) \|_{\CC^{-1/2 + \kappa_\star}_{\eps}} + \| F_{\!\eps}^{(3)}(t) \|_{\CC^{-1/2 + \kappa_\star}_{\eps}} + \| F_{\!\eps}^{(4)}(t) \|_{\CC^{-1/2 + \kappa_\star}_{\eps}}.
%&\lesssim \eps^{\beta} ( 1 + L)^3,
\end{align}
The definition \eqref{eq:norm_eps} yields $\| \cdot \|_{\CC^{-1/2 + \kappa_\star}_{\eps}} \leq \| \cdot \|_{L^\infty}$, which yields
\begin{align}
\| (F_{\!\eps} + \fa \hat s^\eps / 4)(t) \|_{\CC^{-1/2 + \kappa_\star}_{\eps}} \leq \| F_{\!\eps}^{(1)}(t) \|_{L^\infty} + \| F_{\!\eps}^{(3)}(t) \|_{L^\infty} + \| F_{\!\eps}^{(4)}(t) \|_{\CC^{-1/2 + \kappa_\star}_{\eps}} \lesssim \eps^{\beta} ( 1 + L)^3,
\end{align}
where $\beta = (2 \eta - \gamma) \wedge (3 \eta - 1) \wedge (1/2 - \kappa_\star - \hat \kappa)$. Then Lemma~\ref{lem:heat_S} gives
\begin{align}
\|(\CI^{\eps}_{3} - \hat \CI^{\eps}_{3})(t)\|_{\CC^{2\eta}} &\lesssim \int_0^t \| S^\eps_{t-r} \|_{\CC_\eps^{-1/2 + \kappa_\star} \rightarrow \CC^{2\eta}} \| (F_{\!\eps} + \fa \hat s^\eps / 4)(r)\|_{\CC_\eps^{-1/2 + \kappa_\star}} dr \\
&\lesssim \eps^\beta (1 + L)^3 \int_0^t (t-r)^{- (1/2 - \kappa_\star + 2\eta) / 2} dr \lesssim \eps^\beta (1 + L)^3 t^{\hat \beta},
\end{align}
which holds for some $\hat \beta > 0$, because $(1/2 - \kappa_\star + 2\eta) / 2 < 1$. This is the required bound \eqref{eq:term3_2}, and the bound \eqref{eq:term3_1} follows from the triangle inequality.
\end{proof}

\subsection{Proof of the main convergence result}\label{sec:Theorem1.6}

\begin{proof}[Proof of Theorem~\ref{thm:tight_intro}]
Taking into account our restrictions on the values $\gamma$ and $\eta$ in Assumption~\ref{a:f} and Theorem~\ref{thm:tight_intro}, we can choose the values $\kappa_\star$ and $\hat \kappa$ in \eqref{eq:sigma_L} to be such that the powers of $\eps$ in \eqref{eq:term2} and \eqref{eq:term3_1} are strictly positive. Then \eqref{eq:u_eps_new} and Lemmas~\ref{lem:term1}, \ref{lem:term2} and \ref{lem:term3} yield the bound
\begin{align}
\sup_{0 \leq t \leq T \wedge \sigma_{L, \eps}} \| u^\eps_t\|_{\CC^{2 \eta}} \leq C T^{\hat \beta} (1 + L)^3,
\end{align}
for $T \in (0,1]$, for a constant $C$, independent of $\eps$ and $T$, and for some value $\hat \beta > 0$. Taking $T = T_*$ sufficiently small and depending on $L$, we can write
\begin{equation}
\E \|u^\eps\|_{\CC^{2 \eta}_{T_* \wedge \sigma_{L, \eps}}} \leq \hat C,
\end{equation}
for a different proportionality constant $\hat C > 0$. Similarly, the second bound in Lemma~\ref{lem:term3} yields
\begin{align}
\E \|u^\eps - \hat \CI^{\eps}_{3}\|_{\CC^{2 \eta}_{T_* \wedge \sigma_{L, \eps}}} &\leq \hat C \eps^{\beta},
\end{align}
for some $\beta > 0$, where $\hat \CI^{\eps}_{3}$ is defined in \eqref{eq:I_hat}. Iterating this procedure with a new initial data $u^\eps(T_*)$, we obtain these bounds on any time interval $[0, T]$:
\begin{align}\label{eq:bound}
\E \|u^\eps\|_{\CC^{2 \eta}_{T \wedge \sigma_{L, \eps}}} \leq \tilde C, \qquad \E \|u^\eps - \hat \CI^{\eps}_{3}\|_{\CC^{2 \eta}_{T \wedge \sigma_{L, \eps}}} \leq \tilde C \eps^{\kappa},
\end{align}
with a new proportionality constant $\tilde C > 0$.
One can see that in the case when $\sigma_{L, \eps} \geq T > 0$, uniformly in $\eps$, these bounds and Proposition~\ref{prop:X_eps_converge} imply that $\hat{s}^{\eps} := X^\eps + u^\eps$ converges weakly in $D([0, T], \CC(\RR))$ to the solution of \eqref{eq:SHEAdditive} with $A = -\frac{\fa}{4}$ and $B = 1$, and with the initial state $\mathcal{Z}_0$. In order to prove convergence of $\hat{s}^{\eps}$ on $[0,\infty)$, we define a new stopping time
\begin{align}\label{eq:sigma_L_new}
 \sigma^* := \lim_{L \to \infty} \inf \bigl\{t \geq 0 : \limsup_{\eps \to 0} \|u^\eps\|_{\CC^{2 \eta}_{t}} \geq L\bigr\},
\end{align}
which together with the first bound in \eqref{eq:bound} ensures
\begin{align}\label{eq:bound_stopped}
\E \|u^\eps\|_{\CC^{2 \eta}_{T \wedge \sigma^*}} \leq \tilde C.
\end{align}

Then for any $T > 0$ and $\kappa \in (0, \beta)$ we can estimate
\begin{align}\label{eq:final_limit}
 \PP \bigl[ \|u^\eps - \hat \CI^{\eps}_{3}\|_{\CC^{2 \eta}_{T \wedge \sigma^*}} \geq \eps^{\kappa}\bigr] \leq  \PP \bigl[ \|u^\eps - \hat \CI^{\eps}_{3}\|_{\CC^{2 \eta}_{\sigma_{L, \eps}}} \geq \eps^{\kappa}\bigr] +  \PP [ \sigma_{L, \eps} < T \wedge \sigma^*].
\end{align}
The desired result follows if we prove that the limits $\lim_{L \uparrow \infty} \lim_{\eps \downarrow 0}$ of this expression vanish. The first term in \eqref{eq:final_limit} we can bound using the Chebyshev's inequality and the second estimate in \eqref{eq:bound} by
\begin{align}
  \PP \bigl[\|u^\eps - \hat \CI^{\eps}_{3}\|_{\CC^{2 \eta}_{\sigma_{L, \eps}}} \geq \eps^{\kappa}\bigr] \leq \frac{\E \|u^\eps - \hat \CI^{\eps}_{3}\|_{\CC^{2 \eta}_{\sigma_{L, \eps}}}}{\eps^{\kappa}} \leq \tilde C\eps^{\beta - \kappa},
\end{align}
which vanishes as $\eps \to 0$, for every fixed $L$. To bound the second term in \eqref{eq:final_limit}, we make use of the definition \eqref{eq:sigma_L} and obtain
\begin{equation}
\PP \bigl[\sigma_{L, \eps} < T \wedge \sigma^*\bigr] \leq \PP \bigl[\|X^\eps\|_{\CC^{\eta}_{T}} \geq L\bigr] + \PP \bigl[\| Z^\eps\|_{\CC^{-1/2 + \kappa_\star}_{T, \eps}} \geq \eps^{-1/2 - \kappa_\star - \hat \kappa} L\bigr] + \PP \bigl[\|u^\eps\|_{\CC^{2 \eta}_{T \wedge \sigma^*}} \geq L\bigr].\label{eq:final_limit1}
\end{equation}
Lemmas~\ref{lem:X_bound} and \ref{lem:Z_bound} imply that the first two probabilities in \eqref{eq:final_limit1} vanish as $L \to \infty$. The bound \eqref{eq:bound_stopped} guarantees that the last probability in \eqref{eq:final_limit1} vanishes as $L \to \infty$.

Combining these limits together we obtain $\lim_{\eps \to 0}  \PP \bigl[ \|u^\eps - \hat \CI^{\eps}_{3}\|_{\CC^{2 \eta}_{T \wedge \sigma^*}} = 0$, for any $T > 0$. Proposition~\ref{prop:X_eps_converge} implies convergence of $\hat{s}^{\eps} := X^\eps + u^\eps$ weakly in $D([0, T \wedge \sigma^*], \CC(\RR))$ to the solution $\mathcal{Z}$ of \eqref{eq:SHEAdditive} with $A = -\frac{\fa}{4}$ and $B = 1$, and with the initial state $\mathcal{Z}_0$. Since the process $\mathcal{Z}_t$ is defined for $t \in [0, \infty)$, we conclude that the required convergence holds in $D([0, \infty), \CC(\RR))$.
\end{proof}

\appendix

\section{Properties of heat kernels and semigroups}
\label{sec:heat}

In this appendix we provide regularity properties of the continuous and discrete heat semigroups. Moreover, we list bounds on the discrete heat kernel, which are used in the article.

\subsection{Bounds on heat kernels and semigroups}

Let $S^{\eps}_t = e^{t\Delta_\eps}$ be the discrete heat semigroup, generated by the discrete Laplacian $\Delta_\eps$ defined in Section~\ref{sec:notation}. This semigroup has nice regularizing properties when acting on spaces of functions and distributions introduced in Section~\ref{sec:notation}, which we provide below.

\begin{lemma}\label{lem:heat_S}
Let us restrict the domain of $S^{\eps}_t$ to periodic functions/distributions on the circle $\T = \RR / \ZZ$. Then, for any $\alpha \leq 0$ and for any $\gamma > \alpha \vee 0$, there is a constant $C > 0$, independent of $\eps \in (0,1]$ and $t > 0$, such that the following bound holds:
\begin{equation}\label{eq:S_bound}
\| S^{\eps}_t \|_{\CC^{\alpha}_\eps \rightarrow \CC^{\gamma}} \leq C t^{(\alpha - \gamma) / 2}.
\end{equation}
\end{lemma}

\begin{proof}
The bound \eqref{eq:S_bound} can be proved in the same way as a more general result \cite[Prop.~4.17]{MR3774431}. More precisely, the semigroup $S_t^\eps$ is given by convolution with the discrete heat kernel $p^\eps_t(x)$. Furthermore, the bound \eqref{eq:norm_eps} holds also for rescaled Schwarz functions $\phi$. Finally, the kernel $p^\eps_t(x)$ is Schwarz in the $x$ variable, and can be considered as a function rescaled by $\lambda = \sqrt t$. Then \eqref{eq:S_bound} follows from our definitions of the norms.
\end{proof}

The following is an analogue of the classical result \cite[Thm.~2.85]{BCD11} for our $\eps$-dependent norms:

\begin{lemma}\label{lem:product}
Let $\varphi \in \mathcal{C}^\alpha$ and $\psi \in \mathcal{C}^\beta_\eps$, where $\beta < 0 < \alpha < 1$ and $\alpha + \beta > 0$. Then there is a constant $C$, depending on $\alpha$ and $\beta$, such that
\begin{align}
\Vert \varphi \psi \Vert_{\mathcal{C}^\beta_\eps} \leq C \Vert \varphi \Vert_{\mathcal{C}^\alpha} \Vert \psi \Vert_{\mathcal{C}^\beta_\eps}\;.
\end{align}
\end{lemma}

The heat semigroups $S_t$ and $S_t^\eps$ are given by convolutions with respectively the continuous and discrete heat kernels. The latter ones are the Green's functions of the operators $\partial_t - \partial^2_x$ and $\partial_t - \Deltae$, and are defined as the unique solutions of the equations:
\begin{align}\label{eq:heat_equations}
  \partial_t p_t(x) = \partial^2_x p_t(x), \qquad  \partial_t p^\eps_t(y) = \Deltae p^\eps_t(y), \qquad \text{for all}~ t \geq 0, x \in \RR, y \in \eps \ZZ,
\end{align}
with the respective initial values $p_0(x) = \delta_x$ and $p^\eps_0(y) = \eps^{-1} \1{y = 0}$, where $\delta_x$ is the Dirac delta function. For these kernels we have the following bounds.

 \bl\label{lem:KernelEst}
For any $u>0$, any $v \in (0,\frac{1}{2})$ and any $T > 0$, there exists a constant $C$, depending on $u$ and $v$, and independent of $\eps > 0$ and $T$, such that  the following bounds hold
\begin{subequations}
 \begin{align}
    %p^\eps_{t}(x) & \leq C\big(\sqrt t \vee \eps\big)^{- 1},  \label{eq:Est3}\\
    %\sum_{x \in \eps\ZZ} p^\eps_{t}(x) e^{u |x|} &\leq C \eps^{-1}, \label{eq:Est4}\\
   p^\eps_{t}(x) \leq C\big(\sqrt t \vee \eps\big)^{- 1}, \qquad\qquad \eps \sum_{x \in \eps\ZZ} p^\eps_{t}(x) |x|^{\alpha} e^{u |x|} \leq C(\sqrt t \vee \eps)^{\alpha}, \label{eq:Est5}
 \end{align}
 uniformly in $x \in \R$, $t \in [0,T]$ and $\eps \in (0,1]$, and for any $\alpha \geq 0$. Furthermore, the following bounds hold uniformly in $t \in [0,T]$ and $x, x_1, x_2 \in \R$:
\begin{align}
 |\nabla^{+}_{\eps} p^{\eps}(x)|\leq C(\eps^2 t^{-\frac{3}{2}}\wedge \eps^{-1}), & \qquad |p^\eps_{t}(x_1)- p^\eps_{t}(x_2)| \leq C \big(\sqrt t \vee \eps\big)^{-1 - v} |x_2 - x_1|^{2v}. \label{eq:Est2}
\end{align}
Finally, for any points $x \in \R$ and $0 \leq t_1<t_2 \leq T$ one had the bounds
 \begin{align}
    p^\eps_{t_2}(x) \leq C e^{t_2-t_1} p^\eps_{t_1}(x), \qquad\qquad |p^\eps_{t_1}(x)- p^\eps_{t_2}(x)| \leq C \big(\sqrt t_1 \vee \eps\big)^{-1 - v}(t_2 -t_1)^{v/2}.\label{eq:Est1}
 \end{align}
\end{subequations}
  \el

\begin{proof}
These bounds follow from the respective bounds for the non-rescaled discrete heat kernel $p^0_t(x)$, proved in \cite{Bertini1997}, and the identity
\begin{equation}
p^\eps_t(x) = \frac{1}{\eps} p^0_{\eps^{-2}t}(\eps^{-1}x).
\end{equation}
More precisely, the first bound in \eqref{eq:Est5} follows from \cite[Eq.~4.22]{Bertini1997}, and the second one can be proved similarly to \cite[Eq.~4.14]{Bertini1997}. A proof of the bound \eqref{eq:Est2} can be found in \cite[Eq.~4.39]{Bertini1997}. The bounds in \eqref{eq:Est1} are proved in \cite[Eqs.~A.5, 4.44]{Bertini1997}.
\end{proof}

In Section~\ref{Tightness}, we often use the microscopic heat kernel $\{\mathfrak{p}^{\eps}_{t}\}_{t\in \RR_{\geq 0}}$ as defined in \eqref{eq:ODE}. It is worth noting that $p^{\eps}_{t}$ of \eqref{eq:heat_equations} is related to $\mathfrak{p}^{\eps}_{t}$ via $\mathfrak{p}^{\eps}_{t}(x)=\eps p^{\eps}_{\eps^{-2}\sqrt q t}(\eps^{-1}x)$. In the following result, we write the bounds on $\mathfrak{p}^{\eps}_{t}$ by translating Lemma~\ref{lem:KernelEst} using the relation between $\mathfrak{p}^{\eps}_t$ and $ p^{\eps}_t$.

\bl\label{ppn:KernelEst}
Fix any $u>0$, any $v \in [0,\frac{1}{2})$ and any $T > 0$, there exists a constant $C=C(u,v)>0$ such that  the following bounds hold
\begin{subequations}
 \begin{align}
    %p^\eps_{t}(x) & \leq C\big(\sqrt t \vee 1\big)^{- 1},  \label{eq:Est3}\\
    %\sum_{x \in \eps\ZZ} p^\eps_{t}(x) e^{u |x|} &\leq C \eps^{-1}, \label{eq:Est4}\\
   \mathfrak{p}^\eps_{t}(x) \leq C\big(\sqrt t \vee 1\big)^{- 1}, \qquad\qquad  \sum_{x \in \ZZ} \mathfrak{p}^\eps_{t}(x) (\eps|x|)^{\alpha} e^{u \eps|x|} \leq C(\sqrt t \vee 1)^{\alpha}, \label{eq:Est52}
 \end{align}
 uniformly in $x \in \ZZ$, $t \in [0,\eps^{-2}T]$ and $\eps \in (0,1]$, and for any $\alpha \geq 0$. Furthermore, the following bounds hold uniformly in $t \in [0,\eps^{-2}T]$ and $x, x_1, x_2 \in \ZZ$:
\begin{align}
 |\nabla^{+} \mathfrak{p}^{\eps}(x)|\leq C( t^{-\frac{3}{2}}\wedge 1), & \qquad |\mathfrak{p}^\eps_{t}(x_1)- \mathfrak{p}^\eps_{t}(x_2)| \leq C \big(\sqrt t \vee 1\big)^{-1 - v} (\eps|x_2 - x_1|)^{2v}. \label{eq:Est22}
\end{align}
Finally, for any points $x \in \ZZ$ and $0 \leq t_1<t_2 \leq \eps^{-2}T$ one had the bounds
 \begin{align}
    \mathfrak{p}^\eps_{t_2}(x) \leq C e^{\eps^{2}(t_2-t_1)} \mathfrak{p}^\eps_{t_1}(x), \quad\p |\mathfrak{p}^\eps_{t_1}(x)- \mathfrak{p}^\eps_{t_2}(x)| \leq C \big(\sqrt t_1 \vee 1\big)^{-1 - v}\big(\eps^2(t_2 -t_1)\big)^{v/2}.\label{eq:Est12}
 \end{align}
\end{subequations}
\el

\subsection{Properties of some discrete kernels}

Let $\nablae^{\pm}$ be the discrete derivatives, defined in Section \ref{sec:notation}, $p^\eps_t$ from \eqref{eq:heat_equations}, and define
\begin{align}\label{eq:K_kernel}
 K^\eps_t(x) := \nablae^- p^\eps_t(x) \nablae^+ p^\eps_t(x),
\end{align}
We collect important properties of this kernel in the following lemma.

\begin{lemma}\label{lem:K_kernel}
There is a constant $c_0 > 0$, such that for every $T > 0$ the following hold
\begin{align}\label{eq:K_kernel_identity}
\Bigl|\sum_{x \in \eps\ZZ} \int_{0}^T K^\eps_t(x) dt\Bigr| \leq \frac{c_0 \eps^{-1}}{\sqrt{T \vee \eps^2}}, \qquad\qquad \sum_{x \in \eps\ZZ} \int_{0}^\infty K^\eps_t(x) dt = 0.
\end{align}
Moreover, for every $T > 0$ and every $a \geq 0$, there exist values $\eps_0 > 0$, $c_1 \in (0,1)$ and $c_2 > 0$ such that the following bounds hold uniformly in $\eps \in (0,\eps_0)$:
\begin{align}\label{eq:K_kernel_bounds}
 \sum_{x \in \eps\ZZ} \int_{0}^T |K^\eps_t(x)| e^{a |x|} dt \leq \frac{c_1}{\eps^2}, \qquad \sum_{x \in \eps\ZZ} \int_{0}^T |K^\eps_t(x)| e^{a |x|} (T-t)^{-\frac{1}{2}} dt \leq \frac{c_2}{\eps^2}.
\end{align}
\end{lemma}

\begin{proof}
The second identity in \eqref{eq:K_kernel_identity} follows from the first bound in the limit $T \to +\infty$. In order to prove the first bound in \eqref{eq:K_kernel_identity} we write $K^\eps_t(x) = \eps^{-4} K_{\eps^{-2}t}(\eps^{-1}x)$, where the kernel $K$ equals $K^\eps$ with $\eps = 1$. Furthermore, using \eqref{eq:heat_equations} and \eqref{eq:K_kernel}, we can write the Fourier integral
\begin{align}
 K_t(x) = \frac{1}{(2\pi)^2} \int_{-\pi}^\pi dk \int_{-\pi}^\pi d\ell\, e^{i(k + \ell) x} \bigl(1 - e^{-i k}\bigr) \bigl(e^{i \ell} - 1\bigr) e^{-t (2 - \cos k - \cos \ell)}.
\end{align}
Using the identity $\sum_{x \in \ZZ} e^{i k x} = 2 \pi \delta_{k}$, we conclude that
\begin{align}
 \sum_{x \in \eps\ZZ} \int_{0}^T K^\eps_t(x) dt = \frac{1}{2\pi} \int_{-\pi}^\pi dk\,e^{i k} \bigl(1 - e^{-2 T (1 - \cos k)}\bigr) = - \frac{1}{2\pi} \int_{-\pi}^\pi dk\,e^{i k} e^{-2 T (1 - \cos k)}.
\end{align}
The absolute value of the last expression can be bounded by
\begin{align}
 \frac{1}{2\pi} \int_{-\pi}^\pi e^{-2 T (1 - \cos k)} dk \leq \frac{1}{2\pi} \int_{-\pi}^\pi e^{-T k^2} dk = \frac{1}{\pi \sqrt T} \int_{0}^{\pi \sqrt T} e^{-k^2} dk \leq \frac{c_0}{\sqrt{T \vee 1}},
\end{align}
for some constant $c_0 > 0$, independent of $T$. The required bound \eqref{eq:K_kernel_identity} follows now from the last one after rescaling.

The bounds \eqref{eq:K_kernel_bounds} follow immediately from \cite[Lem.~A.3]{Bertini1997} and \cite[Lem.~A.4]{Bertini1997} respectively.
\end{proof}

%...
%
%\begin{lemma}\label{lem:K_bound}
% For the just defined kernel $K^\eps_t(x)$ the following bound holds
%\begin{align}
%\left| \int_0^T \langle K^\eps_t, \phi_y^\lambda \rangle_\eps dt \right| \leq C \frac{1}{\lambda} \left(\frac{1}{\lambda} + \frac{1}{\sqrt T \vee \eps}\right) \| \phi \|_{\CC^1},
%\end{align}
%where the constant $C$ is independent of $T$, $y$, $\lambda$ and $\eps$.
%%\begin{align}
%%\Bigl| \eps \sum_{x \in \eps\ZZ} \int_0^T K^\eps_t(x) dt\Bigr| \leq \frac{C}{\sqrt{T \vee \eps^2}},
%%\end{align}
%%where the constant $C$ is independent of $\eps$ and $T > 0$.
%\end{lemma}
%
%\begin{proof}
% ...\km{???}
%\end{proof}

\section{Bounds on iterated stochastic integrals}
\label{sec:integrals}

In this appendix we recall some properties of c\`{a}dl\`{a}g martingales and provide moment bounds for iterated stochastic integrals.

\subsection{Iterated stochastic integrals}

Assume we have a family of square integrable martingales $(M^{\eps, \ell}_t(x))_{t \geq 0}$, parameterized by $\ell \geq 1$ and $x \in \eps\ZZ$, which have the following properties: the martingale $M^{\eps, \ell}_t(x)$ is of bounded total variation and has jumps of size $\eps^{\delta_\ell}$ for $\delta_\ell \geq \frac{1}{2}$; $M^{\eps, \ell}_t(x)$ and $M^{\eps, \ell}_t(y)$ a.s. do not jump together, if $x \neq y$; on every time interval $[0, \eps^2 T]$ the martingale $M^{\eps, \ell}_t(x)$ makes a.s. finitely many jumps. In this case, the quadratic variation can be written as $$\langle M^{\eps, \ell}(x), M^{\eps, \ell}(y)\rangle_t = \eps^{-1} \mathbbm{1}_{x=y} \int_0^t C^{\eps, \ell}_r(x) dr,$$ for some function $C^{\eps, \ell}_r(x)$, adapted to the underlying filtration in the variable $r \geq 0$. Moreover, we assume that the function $C^{\eps, \ell}_r(x)$ is a.s. bounded by a constant $c_{\eps, \ell}$, uniformly in $r$ and $x$.

When working with such martingales, the following two forms of the Burkholder-Davis-Gundy inequality will be used. The first form is standard.

\begin{lemma}\label{lem:BDGone}
For all $p\geq 1$ there exists a constant $C_p>0$ such that for all martingales $M_t$, and for all $T>0$,
\begin{equation}
\mathbb{E} \big[\sup_{t \in [0, T]} |M_t - M_0|^p\big] \leq C_p \mathbb{E} \big[[ M, M]^{p/2}_T\big].
\end{equation}
\end{lemma}

The second form of the inequality is adapted to the setting described.

\begin{lemma}\label{lem:BDG}
Let $M^\eps$ be a martingale with the just described properties. Then, for any $p \geq 1$, there exists a constant $C=C(p)$, such that for every $T>0$ one has
\begin{equation}\label{eq:BDG}
\E \bigg[  \sup_{0 \leq t \leq T} |M^\eps(t)|^p \bigg] \leq C \bigg(\E \Big[  \langle M^\eps, M^\eps \rangle_T^{\frac{p}{2}} \Big]  + \E \bigg[  \sup_{0 \leq t \leq T}  |\Delta_t M^\eps  |^p  \bigg]   \bigg)\;,
\end{equation}
where $\Delta_t M^\eps := M^\eps(t) - M^\eps({t-})$ denotes the jump of $M^\eps$ at time $t$.
\end{lemma}
\begin{proof}
This formula can be proved by discrete time approximations of the martingales $M^\eps$, and using an analogous formula \cite[Thm.~2.11]{HaHe} in discrete time.
\end{proof}

For a continuous function $F \colon [0,\infty)^n \times \eps\ZZ^n \to \R$, we define iterated stochastic integrals $(\I^\eps_n F)(t)$ as follows: for $n=1$ we set $(\I^\eps_1F)(t) := \eps \sum_{y \in \eps\ZZ} \int_{0}^t F(r,y) \, d M^{\eps,1}_r(y) $; for $n \geq 2$ we define recursively
\begin{align}\label{eq:integral_def}
(\I^\eps_{n} F)(t) := \eps \sum_{y \in \eps\ZZ} \int_{0}^t (\I^\eps_{n-1} F^{(r,y)}) (r-) \, d M^{\eps, n}_{r}(y),
\end{align}
where $r-$ is the left limit at $r$, and the function $F^{(r,y)} \colon [0,\infty)^{n-1} \times (\eps\ZZ)^{n-1}$ is defined as
\begin{equation*}
F^{(r,y)}(r_1, \ldots ,r_{n-1}; y_1, \ldots , y_{n-1}) := F(r_1, \ldots ,r_{n-1}, r; y_1 , \ldots, y_{n-1}, y).
\end{equation*}
To make notation shorter, in what follows we write $(\I^\eps_{\ell} F)(r_{\ell+1}, \ldots, r_n; y_{\ell+1}, \ldots, y_n)$ for the iterated stochastic integral, taken with respect to the variables $r_{1}, \ldots ,r_{\ell}$ and $y_{1}, \ldots, y_\ell $ and evaluated at time $t= r_{\ell+1}$, where $r_{\ell+1}, \ldots, r_n$ and $y_{\ell+1}, \ldots, y_n$ are treated as free parameters. With this definition, the process $t \mapsto (\I^\eps_{n} F)(t)$ is a martingale, and moreover we have the following moment bounds for it, which is a modification of \cite[Lem.~4.1]{MW15}.

\begin{lemma}\label{lem:integral_bound}
Let $n \geq 1$ and let $F \colon [0,\infty)^n \times (\eps\ZZ)^n \to \R$ be continuous and deterministic. Then for any $p \geq 1$, there exists a constant $C = C(n,p)$ such that
\begin{equation}\label{eq:integral_bound}
\bigg(\E \Bigl[ \sup_{0 \leq t \leq T} \big| (\I^\eps_{n}F)(t) \big|^p\Bigr] \bigg)^{2 / p}  \leq C \bigg(\prod_{\ell=1}^n c_{\eps, \ell}\bigg) \eps^{n}\!\! \sum_{\substack{\mathbf{y} \in (\eps\ZZ)^n}}\int_{r_n=0}^T \!\ldots \! \int_{r_1=0}^{r_{2}} F(\mathbf{r}; \mathbf{y})^2 d\mathbf{r} + \SE^\eps_{n}(T),
\end{equation}
where we use the shorthand notation $ \mathbf{r} = (r_1, \ldots, r_n)$ and $ d\mathbf{r} = dr_{1}\cdots d{r_n}$, and where the error term $\SE^\eps_{n}$ is given by
\begin{align}
\SE^\eps_{n}(T) &:= C  \sum_{\ell=1}^{n} \eps^{n - \ell + 2(1 + \delta_{\ell})} \bigg(\prod_{k=\ell + 1}^{n} c_{\eps, k}\bigg) \\
&\qquad \times \sum_{\mathbf{y}_{\ell+1,n} \in (\eps\ZZ)^{n - \ell}} \int_{r_n=0}^T \!\ldots\! \int_{r_{\ell + 1}=0}^{r_{\ell+2}} \bigg( \E \Bigl[ \sup_{\substack{r_{\ell} \in [0, r_{\ell+1}]\\y_{\ell} \in \eps\ZZ}}  \big| (\I^\eps_{\ell-1} F)(\mathbf{r}_{\ell, n}; \mathbf{y}_{\ell, n}) \big|^p\Bigr] \bigg)^{2 / p}\!\! d\mathbf{r}_{\ell+1,n},\label{eq:integral_error}
\end{align}
where $\mathbf{r}_{\ell, n} = (r_\ell, \ldots, r_n)$, $\mathbf{y}_{\ell,n}  =(y_\ell, \ldots , y_n)$,  $d\mathbf{r}_{\ell+1,n} = dr_{\ell+1}\cdots d{r_n}$ and $r_{n+1} =T$.
\end{lemma}

\begin{proof}
We first prove the bound \eqref{eq:integral_bound} in the case $n=1$. In order to apply the Burkholder-Davis-Gundy inequality \eqref{eq:BDG}, we need to bound the quadratic variation and the jump size of the martingale $(\I_1F)(t)$. Since the jump of the martingale $M^{\eps, 1}$ equals $\eps^{\delta_1}$, a jump of $(\I_1F)(t)$ is bounded by $\sup_{r \in [0, T]} \sup_{y \in \eps\ZZ} \eps^{1 + \delta_1}  | F(r, y) |$. For the predictable covariation we have
\begin{align}
 \big\langle \I_1 F, \I_1 F \big\rangle_T &= \eps^2\!\! \sum_{y, y' \in \eps\ZZ} \int_{r=0}^T F(r,y) F(r, y')  d \langle  M^{\eps,1}(y) , M^{\eps,1}(y') \rangle_r\\
 &= \eps\!\! \sum_{y \in \eps\ZZ} \int_{0}^T F(r,y)^2 C^{\eps,1}_r(y) d r \leq c_{\eps, 1} \eps\!\! \sum_{y \in \eps\ZZ} \int_{0}^T F(r,y)^2 d r.
\end{align}
Hence, the Burkholder-Davis-Gundy inequality \eqref{eq:BDG} yields for every $p \geq 1$ the bound
\begin{align*}
\Big(\E  \Bigl[ \sup_{0 \leq t \leq T} \big| (\I_1 F) (r) \big|^p\Bigr] \Big)^{2 / p} \lesssim c_{\eps, 1} \eps\!\! \sum_{y \in \eps\ZZ} \int_{0}^T F(r,y)^2 d r + \eps^{2(1 + \delta_1)}\!\! \sup_{r \in [0, T]} \sup_{y \in \eps\ZZ} | F(r, y) |^2,
\end{align*}
where the proportionality constant depends on $p$. This is exactly the required bound \eqref{eq:integral_bound}.

Now, we will proceed by induction. To this end, we assume that \eqref{eq:integral_bound} holds for $n-1$, and will prove it for $n \geq 2$. As before, we prove the martingale $(\I_n F)(t)$ using the Burkholder-Davis-Gundy inequality, for which we need to bound the jumps and quadratic covariation of $(\I_n F)(t)$. Using the recursive definition \eqref{eq:integral_def} and properties of the martingales $M^{\eps, n}$, the jump of $(\I_n F)(t)$ can be bounded by
\begin{align}
\eps^{1 + \delta_n}\!\! \sup_{r \in [0, t]} \sup_{y \in \eps\ZZ} \bigl|(\I^\eps_{n-1} F^{(r,y)}) (r-)\bigr|.
\end{align}
Furthermore, the quadratic covariation can be written as
\begin{align}
 \big\langle \I_n F, \I_n F \big\rangle_T &= \eps^2\!\! \sum_{y, y' \in \eps\ZZ} \int_{r=0}^T (\I^\eps_{n-1} F^{(r,y)}) (r-) (\I^\eps_{n-1} F^{(r,y')}) (r-)  d \langle  M^{\eps,n}(y) , M^{\eps,n}(y') \rangle_r\\
 &= \eps\!\! \sum_{y \in \eps\ZZ} \int_{0}^T (\I^\eps_{n-1} F^{(r,y)}) (r-)^2 C^{\eps,n}_r(y) d r \leq c_{\eps, n} \eps\!\! \sum_{y \in \eps\ZZ} \int_{0}^T (\I^\eps_{n-1} F^{(r,y)}) (r-)^2 d r.
\end{align}
Combining these bounds with the Burkholder-Davis-Gundy inequality \eqref{eq:BDG}, we obtain
\begin{align}
\bigg(\E \Bigl[ \sup_{0 \leq t \leq T} \big| (\I^\eps_{n}F)(t) \big|^p\Bigr] \bigg)^{2 / p} &\lesssim c_{\eps, n} \eps\!\! \sum_{y \in \eps\ZZ} \int_{0}^T \!\!\Big(\E \Bigl[ (\I^\eps_{n-1} F^{(r,y)}) (r-)^p\Bigr]\Big)^{2/p} d r \\
&\qquad + \eps^{2(1 + \delta_n)} \bigg(\E \Bigl[ \sup_{\substack{r \in [0, t] \\ y \in \eps\ZZ}} \bigl|(\I^\eps_{n-1} F^{(r,y)}) (r-)\bigr|^p\Bigr]\bigg)^{2 / p},
\end{align}
which holds for every $p \geq 1$ with a proportionality constant depending on $p$. Applying now the induction hypothesis to the stochastic integrals $(\I^\eps_{n-1} F^{(r,y)})(r-)$ on the r.h.s., we arrive at \eqref{eq:integral_bound}.
\end{proof}

\section{Bounds on singular kernels}
\label{sec:singular}

In this appendix, we provide some bounds on singular kernels, which are used in the proof of Lemma~\ref{lem:Z_bound}. For this, we follow the idea of \cite[Sec.~10.3]{Hai13} (or rather of \cite[Sec.~6 and 7.1]{MR3785597}, since all the kernels are discrete), and introduce a ``strength" of singularities of kernels.

More precisely, we consider a kernel $P^\eps : \RR \times (\eps \ZZ) \to \RR$, and for $\alpha > 0$ we will write $P^\eps \in \CS^\alpha_\eps$ if $P^\eps$ can be written as $P^\eps = K^\eps + R^\eps$, where the kernels $K^\eps$ and $R^\eps$ have the following properties:
\begin{enumerate}
\item The kernel $K^\eps$ is supported in a ball centered at the origin, and satisfies the bound
\begin{equation}\label{eq:kernel_bound}
|K^\eps(t,x)| \leq C (\sqrt{|t|} \vee |x| \vee \eps)^{-\alpha},
\end{equation}
uniformly in $t \in \RR$, $x \in \eps \ZZ$ and $\eps \in (0,1]$, where the constant $C$ is independent of $t$, $x$ and $\eps$.
\item For every fixed $t$, the kernel $R^\eps(t,x)$ decays at infinity faster than any polynomial in the variable $x$, i.e. for every $n \in \NN$ one has $\lim_{|x| \to \infty} R^\eps(t,x) / |x|^n = 0$.
\item Moreover, $R^\eps(t,x)$ is bounded uniformly in $\eps$ and is integrable over $\RR \times (\eps \ZZ)$.
\end{enumerate}

\noindent The value of $\alpha$ measures the ``strength" of singularity of the kernel $P^\eps$. Moreover, the following lemma shows that such singular kernels preserve these properties under multiplication and convolution.

\begin{lemma}\label{lem:singular}
Let $P^\eps \in \CS^\alpha_\eps$ and $\bar P^\eps \in \CS^{\bar \alpha}_\eps$ be two singular kernels, for $\alpha, \bar \alpha > 0$. Then one has
\begin{enumerate}
\item the product $P^\eps \bar P^\eps$ is a singular kernel and $P^\eps \bar P^\eps \in \CS^{\alpha + \bar \alpha}_\eps$;
\item if $\alpha \vee \bar \alpha < 3$ and $\alpha + \bar \alpha > 3$, then the space-time convolution of these kernels satisfies $P^\eps * \bar P^\eps \in \CS^{\alpha + \bar \alpha - 3}_\eps$.
\end{enumerate}
\end{lemma}

\begin{proof}
%Let us write $P^\eps = K^\eps + R^\eps$ and $\bar P^\eps = \bar K^\eps + \bar R^\eps$, where the functions have the properties, described above. Then we can write $P^\eps \bar P^\eps = \tilde K^\eps + \tilde R^\eps$, where $\tilde R^\eps = R^\eps \bar R^\eps$ and $\tilde K^\eps = K^\eps \bar K^\eps +  K^\eps \bar R^\eps + R^\eps \bar K^\eps$, and $P^\eps \bar P^\eps \in \CS^{\alpha + \bar \alpha}_\eps$ now follows at once.
These results is a simple version of \cite[Lem.~7.3]{MR3785597}.
\end{proof}

The discrete heat kernel $p^\eps$, defined in \eqref{eq:heat_equations}, is a singular kernel in the sense introduced above. More precisely, the following result holds:

\begin{lemma}\label{lem:heat}
One has $p^\eps \in \CS^1_\eps$ and $\nablae^\pm p^\eps \in \CS^2_\eps$, where $\nablae^\pm$ are the discrete derivatives, defined in Section~\ref{sec:notation}.
\end{lemma}

\begin{proof}
This results follow from \cite[Lem.~5.4]{MR3785597}.
\end{proof}

The following lemma shows how such singular kernels behave under convolutions with scaled test functions.

\begin{lemma}\label{lem:test}
Let $\phi_z^\lambda$ be a scaled test function, as in \eqref{eq:norm_eps}, and let $P^\eps \in \CS^\alpha_\eps$, for some $\alpha \in (0,1)$. Then the following bound holds:
\begin{equation}
\eps^2 \sum_{ x_1,  x_2 \in \eps\ZZ} \bigl|\phi^\lambda_z( x_1) \phi^\lambda_z( x_2) P^{\eps}(t,  x_2 -  x_1)\bigr| \leq C (\lambda \vee \eps)^{-\alpha},
\end{equation}
uniformly in $\lambda \in (0,1]$, $t \in \RR$, $z \in \eps \ZZ$ and $\eps \in (0,1]$.
\end{lemma}

\begin{proof}
Combining the bound \eqref{eq:kernel_bound} with the definition of the rescaled function $\phi^\lambda_z$, we obtain
\begin{align}
\eps^2 \sum_{ x_1,  x_2 \in \eps\ZZ} \bigl|\phi^\lambda_z( x_1) \phi^\lambda_z( x_2) P^{\eps}(t,  x_2 -  x_1)\bigr| \lesssim \eps^2 \sum_{ x_1,  x_2 \in \eps\ZZ} \bigl|\phi^\lambda_z( x_1) \phi^\lambda_z( x_2)\bigr| (| x_2 -  x_1| \vee \eps)^{-\alpha}\\
\lesssim \eps^2 \lambda^{-2} \sum_{\substack{ x_1,  x_2 \in \eps\ZZ:\\ | x_1| \leq \lambda, | x_2| \leq \lambda}} (| x_2 -  x_1| \vee \eps)^{-\alpha} \lesssim \eps \lambda^{-1} \sum_{\substack{ x \in \eps\ZZ:\\ | x| \leq 2 \lambda}} (| x| \vee \eps)^{-\alpha},
\end{align}
where in the last bound, we simply changed the variables of summation $ x :=  x_2 - x_1$. If $\alpha < 1$, then the last sum is bounded by a multiple of $(\lambda \vee \eps)^{-\alpha}$, uniformly in $\lambda$ and $\eps$.
\end{proof}

\section{Proofs of Lemma~\ref{Lem1} and Corollary~\ref{Cor2}}
\label{sec:stationary}

\begin{proof}[Proof of Lemma~\ref{Lem1}]
For the duration of this proof, we will drop the subscript $0$ and write $s$ and $\hat{s}^{\eps}$ in place of $s_0$ or $\hat{s}_0^{\eps}$ (recall, these are only functions of space since they are initial data).

 \smallskip
\noindent {\bf Part 1:} Substituting $q=e^{-\eps}$  into \eqref{eq:stationary}, we show by direct computation that $\hat{s}^{\eps}(0)$ converges to a standard Gaussian random variable (centered with variance one). This can be done, for instance by observing that
 $$
 \mathbb{P}\big(\hat{s}^{\eps}(0) \leq  w\big) = \mathbb{P}\big(s(0)\leq  \eps^{-\frac{1}{2}}w+ \log_{q}\alpha\big) = \sum_{n=-\infty}^{N}\frac{\alpha^{-2n} q^{n(2n-1)}(1+\alpha^{-1}q^{2n})}{(-\alpha^{-1}, -q\alpha, q;q)_{\infty}},
 $$
 where $N= \big\lfloor\frac{1}{2}\eps^{-\frac{1}{2}} w + \frac{1}{2}\log_q\alpha\big\rfloor$. The last sum can be approximated by the Gaussian integral. By stationarity in the spatial coordinate, it follows that $\hat{s}^{\eps}(x)$ likewise has a Gaussian limit (for each $x$). In the same manner, we may  show that $\hat{s}^{\eps}(x)$ satisfies \eqref{eq:NSofS1} for any $x$.
 %
% \begin{align}
% \mathbb{P}\big(\hat{s}^{\eps}_0(0) = w\big) &= \mathbb{P}\big(s_0(0)= \eps^{-\frac{1}{2}}w+ \log_{q}\alpha\big)\\
% & =\frac{\alpha^{-2n} q^{n(2n-1)}(1+\alpha^{-1}q^{2n})}{(-\alpha^{-1}, -q\alpha, q;q)_{\infty}} \text{ where }n= \frac{1}{2}\eps^{-\frac{1}{2}} w + \frac{1}{2}\log_q\alpha\\
% & \propto e^{-\frac{1}{2}w^2 -\eps w}\label{eq:ProbConv}
% \end{align}
%which shows that the one point distribution of $\hat{s}^{\eps}_0(0)$ converges to a Gaussian distribution with mean $0$ and variance $1$ as $\eps\to 0$. Moreover, \eqref{eq:ProbConv} shows that $\hat{s}^{\eps}_0(0)$ satisfies \eqref{eq:NSofS1}. Now, we turn to show that \eqref{eq:NSofS1} is satisfied by $\hat{s}^{\eps}(0,X)$ for all $X\in \RR$. If $\eps^{-1}X$ is an even integer, then, thanks to the spatial stationarity (upto parity) of the distribution of $s_{0}(\bigcdot)$ (see \cite[Definition~2.12]{BC17}), the distributions of $\hat{s}^{\eps}(0,X)$ and $\hat{s}^{\eps}_0(0)$ are same. This implies \eqref{eq:NSofS1} holds for $\hat{s}^{\eps}(0,X)$. On the other hand, if $\eps^{-1}X$ is an odd integer, then, combining $\hat{s}^{\eps}(0,X-\eps)\stackrel{d}{=}\hat{s}^{\eps}_0(0)$ and $|\hat{s}^{\eps}(0,X-\eps)- \hat{s}^{\eps}(0,X)|=\sqrt{\eps}$ with \eqref{eq:ProbConv} yields \eqref{eq:NSofS1} for $\hat{s}^{\eps}(0,X)$. Proof of \eqref{eq:NSofS1} when $\eps^{-1}X$ lies between two integers follows from similar argument.
%\smallskip

We turn to showing \eqref{eq:NSofS2}. Thanks to the spatial stationarity, it suffices to show \eqref{eq:NSofS2} for $x_1= x$ and $x_2= 0$ where $x <0$. If $x\leq -1$, then by the triangle inequality, the stationarity of $\hat{s}^{\eps}(x)$ and the bound \eqref{eq:NSofS1}, there exists a positive absolute constant $C=C(k)$ such that
$$
\|\hat{s}^{\eps}(x) - \hat{s}^{\eps}(0)\|_{2k}\leq \|\hat{s}^{\eps}(x)\|_{2k}+ \|\hat{s}^{\eps}(0)\|_{2k}\leq 2C.
$$
This proves \eqref{eq:NSofS2} when $x<-1$. It only remains to show that \eqref{eq:NSofS2} holds when $x\in (-1,0)$.

Define $M:\ZZ\to \RR$ such that $M(0)=0$ and
  \begin{align}
  M(x)= \begin{cases}
  \sum_{i=1}^{x}\big(s(i-1) - \mathbb{E}[s(i-1)\big|s(i)]\big),  & x>0,\\
  \sum_{i=-x}^{-1}\big(s(i) - \mathbb{E}[s(i)\big|s(i+1)]\big), & x\leq -1.
  \end{cases}
  \end{align}
Using $M$ we may write a difference equation for $s(x)$, with $x\in \ZZ$:
 \begin{align}\label{eq:dd}
 \nabla^{-} s(x) = s(x)-\mathbb{E}[s(x-1)\big|s(x)]- \nabla^{-} M(x) =-\frac{q^{s(x)}-\alpha}{q^{s(x)}+\alpha}- \nabla^{-} M(x).
 \end{align}
 For any $x\in \eps \ZZ$, define $M^{\eps}(x):= \sqrt{\eps}M(\eps^{-1}x)$ and extend to all $x\in \RR \backslash \eps \ZZ$ by linear interpolation. Thanks to \eqref{eq:dd},
 \begin{align}\label{eq:ddAprox}
 \nabla^{-}_{\eps} \hat{s}^{\eps}(x) = -\sqrt{\eps}\frac{q^{\eps^{-\frac{1}{2}}\hat{s}^{\eps}(x)}-1}{q^{\eps^{-\frac{1}{2}}\hat{s}^{\eps}(x)}+1} - \nabla^{-}_{\eps} M^{\eps}(x).
 \end{align}
Taylor expanding with respect to $\eps$,
 \begin{equation}\label{eq:mExpand}
\sqrt{\eps}\frac{q^{\eps^{-\frac{1}{2}}\hat{s}^{\eps}(x)}-1}{q^{\eps^{-\frac{1}{2}}\hat{s}^{\eps}(x)}+1} = -\frac{1}{2}\eps \hat{s}^{\eps}(x) +\eps^{\frac{3}{2}} \mathcal{B}^{\eps}(x),
\end{equation}
  where $|\mathcal{B}^{\eps}(x)|<(1+|\hat{s}^{\eps}(0)|)^2$ for all $|x|\leq 1$. Plugging \eqref{eq:mExpand} into the r.h.s. of \eqref{eq:ddAprox} and solving the discrete difference equation \eqref{eq:ddAprox} yields
  \begin{equation}\label{eq:sExp}
  \hat{s}^{\eps}(x)= \big(1-\frac{\eps}{2}\big)^{-\eps^{-1}x}\hat{s}^{\eps}(0)+ \big(1-\frac{\eps}{2}\big)^{-\eps^{-1}x} \sum^{-1 }_{\ell =\eps^{-1}x} \Big(1-\frac{\eps}{2}\Big)^{\ell}\Big(\nabla^{-}_{\eps}M^{\eps}(\eps\ell) +\eps^{\frac{3}{2}}\mathcal{B}^{\eps}(\eps \ell)\Big).
\end{equation}
It is easy to check that this indeed solves  \eqref{eq:dd} for all $x\in\eps\ZZ_{<0}$.

%Now, we show that \eqref{eq:sExp} indeed solves \eqref{eq:dd} for all $X\in\eps\ZZ_{<0}$. We begin by noting that \eqref{eq:sExp} is trivially satisfied when $X=-\eps$. Assume that \eqref{eq:sExp} holds for $X=-2\eps, \ldots ,-\eps x_0$. It suffices to show \eqref{eq:sExp} holds for $X=-\eps(x_0+1)$. Recall that
%\begin{align}
%\hat{s}_0(-\eps(x_0+1)) = \big(1-\frac{\eps}{2}\big)\hat{s}_0(-\eps x_0)+ \nabla^{-}_{\eps} M^{\eps}_0(\eps x_0)+\eps^{\frac{3}{2}}\mathcal{B}_{\eps}(\eps \ell).
%\end{align}
%Plugging $\hat{s}_0(-\eps x_0)$ (from \eqref{eq:sExp}) into the r.h.s. of the above equation yields \eqref{eq:sExp} for $X=-\eps(x_0+1)$. By induction, we get \eqref{eq:sExp} for all $X\in \eps \ZZ_{<0}$.

%Fix $X_1> X_2\in \RR$ and denote $x_1:=\eps^{-1}X_1, x_2:=\eps^{-1}X_2$. Then,
%\begin{equation}\label{eq:DiffRep}
%s(0, x_1) -s(0, x_2)=  \sum_{i=\lfloor x_2\rfloor+1 }^{\lfloor x_1\rfloor} \xi_i + \Xi(x_1,x_2)
%\end{equation}
%where $|\Xi(x_1,x_2)|\leq 1$. From the definition of stationary initial data (see Definition~\ref{Stationary}), we know the condition distribution of $\xi_{i-1}$ given $s_0(i)$, namely, $\xi_{i-1}=1$ with probability $\frac{q^{s_0(i)}}{\alpha+q^{s_0(i)}}$ and $\xi_{i-1}=-1$ with probability $\frac{\alpha}{\alpha+q^{s_0(i)}}$. Denoting $m_i := \mathbb{E}[\xi_{i}|s_0(i+1)]$ and Taylor expanding $m_i$

We now claim that for any $x\in (-1,0)$ and any $u>0$ %\note{Ivan: I added a 2 below since the argument seems to suggest it}
  \begin{align}
  \mathbb{E}\Big[\exp\big(u |\hat{s}^{\eps}(x)- \hat{s}^{\eps}(0)|\big)\Big]\leq 2 \exp\Big(\frac{u^2}{2}\int^{|x|}_{0} e^{y- |x|}dy\Big). \label{eq:CentMomBd}
  \end{align}
Before verifying this, let us see how to complete the proof of \eqref{eq:NSofS2}. Fix any $\beta \in (0,\frac{1}{4})$ and $u>0$. Note that the integral inside $\exp$ on the r.h.s. of \eqref{eq:CentMomBd} is bounded above by $C^{\prime}\min\big\{|x|^{4\beta}, 1\big\}$ for some constant $C^{\prime}= C^{\prime}(\beta)>0$. Combining this with Markov's inequality, we find that
  \begin{equation}
  \big\|\hat{s}^{\eps}(x) -\hat{s}^{\eps}(0)\big\|_{2k}\leq  C^{\prime} |x|^{2\beta},
  \end{equation}
   which completes the proof of \eqref{eq:NSofS2}.

We must now prove \eqref{eq:CentMomBd}. It suffices to do this for $x\in \eps \ZZ_{<0}\cap (-1,0)$ since $|\hat{s}^{\eps}(x)- \hat{s}^{\eps}([x]_{\eps})|\leq \sqrt{\eps}$ where $[x]_{\eps}$ is an element of $\eps \ZZ_{<0}$ nearest to $x$. Let $x= \eps^{-1}x$ where $x\in \eps \ZZ_{<0}\cap (-1,0)$. $\nabla^-_{\eps}M^{\eps}(x)$ is a centered Bernoulli random variable scaled by $\sqrt{\eps}$. Hence, by using Hoeffding's inequality, we find that
$$\mathbb{E}\Big[\exp\big(u \nabla^{-}_{\eps}M(x)\big)\big| s(x+1), s(x+2)\ldots \Big]\leq \exp(\eps u^2 / 2).$$
 In a similar way, we successively bound
 $$\mathbb{E}\Big[\exp\big(u\big(1-\frac{\eps}{2}\big)^{\ell-\eps^{-1}x}\nabla^{-}_{\eps}M(\eps \ell)\big)\big| s( \ell+1),s(\ell+2), \ldots\Big] \leq\exp\Big(\frac{\eps u^2}{2}\big(1-\frac{\eps}{2}\big)^{2(\ell-\eps^{-1}x)}\Big),$$
 for all $\ell=\eps^{-1}x+1, \ldots,  -1$. Combining these bounds with the bound \eqref{eq:NSofS1} for $\hat{s}^{\eps}(0)$, we see via \eqref{eq:sExp} that
  \begin{align}\label{eq:HoeffBd}
  \mathbb{E}\Big[\exp\Big(u \big(\hat{s}^{\eps}(x)- \hat{s}^{\eps}(0)\big)\Big)\Big]\leq \exp\Big(\frac{u^2}{2}\eps\sum_{\ell=\eps^{-1}x}^{-1} \big(1-\frac{\eps}{2}\big)^{2(\ell-\eps^{-1}x)}\Big).
    \end{align}
  Summing the above inequality with $u$ and $-u$ we find that $\textrm{l.h.s. \eqref{eq:CentMomBd}}\leq 2 \times\textrm{r.h.s. \eqref{eq:HoeffBd}}$. Recalling that  $(1-\eps/2)\leq e^{-\eps/2}$, we may upper bound the sum on the r.h.s. of \eqref{eq:HoeffBd} by the integral in the r.h.s. of \eqref{eq:CentMomBd}, completing the proof of \eqref{eq:CentMomBd}.

 \smallskip
\noindent {\bf Part 2:}    Owing to the Part 1 of this lemma and the Kolmogorov-Centsov criterion of tightness (see \cite[Theorem~2.23]{Kallenberg}), $\{\hat{s}^{\eps}\}_{\eps}$ is a tight sequence in $\mathcal{C}(\RR)$. Let us assume that we have  some subsequential limit  $\eps_k\to 0$ along which %\note{Ivan: I changed $-X$ to $X$ here}
 $\{\hat{s}^{\eps_k}(x)\}_{x\in \RR}$ converges weakly to some limit $\{\mathcal{Z}_0(x)\}_{x\in \RR}$ in $\mathcal{C}(\RR)$. We have already shown that for any $x$, marginally $\mathcal{Z}_0(x)$ must be standard Gaussian. Thus, it suffices to show that $\mathcal{Z}_0$ satisfies \eqref{eq:OU}. By uniqueness of that solution, this will then show that all subsequential limits are the same, and hence we have convergence as desired.

So, we seek to prove that $\mathcal{Z}_0$ satisfies \eqref{eq:OU}. Let us abuse notation and write $\eps$ in place of $\eps_k$. Using similar argument to that used in proving  \eqref{eq:CentMomBd}, there exists $C>0$ such that
 \begin{align}\label{eq:Mconc}
 \mathbb{E}\Big[\exp\big(u|M^{\eps}(x_1) - M^{\eps}(x_2)|\big)\Big]\leq \exp\big(Cu^2|x_1-x_2|\big),
\end{align}
for any $x_1, x_2\in \RR$ and $u>0$. This shows the tightness of $\{M^{\eps}\}_{\eps}$ in $\mathcal{C}(\RR)$. Moreover, the weak convergence of $\hat{s}^{\eps}$, \eqref{eq:dd} and the tightness of $\{M^{\eps}\}_{\eps}$ imply that $M^{\eps}$ also weakly converges to some limiting spatial process $W$ taking values in $ \mathcal{C}(\RR)$. Notice that %\note{Ivan: and here}
$\eps^{-\frac{1}{2}}(q^{\eps^{-1/2}\hat{s}^{\eps}(x)}-1)/(q^{\eps^{-1/2}\hat{s}^{\eps}(x)}+1)\Rightarrow -\frac{1}{2}\mathcal{Z}_0(x)$ in $\mathcal{C}(\RR)$. Combining this with \eqref{eq:ddAprox}  yields
 \begin{align}
 d\mathcal{Z}_0(x) = -\frac{1}{2}\mathcal{Z}_0(x)dx+ dW(x).
 \end{align}
 Now, it suffices to show that $W$ is a two sided Brownian motion up to some scaling. To prove  this, for any $\psi\in L^2(\RR)$, we define
 \begin{align}
 \mathcal{M}^{\eps}_{\psi}:= \sum_{x\in \ZZ} \psi(\eps x)\nabla^{-}_{\eps}M^{\eps}(\eps x).
\end{align}
Owing to the decay of $\psi$, the weak convergence of $M^{\eps}$ to $W$ and the moment bound of $M^{\eps}$ from \eqref{eq:CentMomBd}, we see that $\mathcal{M}^{\eps}_{\psi}$ converges weakly to $\int_{\RR} \psi(x)dW(x)$. Furthermore, we have
\begin{align}
\mathbb{E}\big[(\mathcal{M}^{\eps}_{\psi})^2\big] &= \sum_{x,y \in \ZZ} \psi(\eps x) \psi(\eps y)\mathbb{E}[\nabla^{-}_{\eps}M^{\eps}(\eps x)\nabla^{-}_{\eps}M^{\eps}(\eps y)]\\& = \sum_{x\in \ZZ} \psi(\eps x) \psi(\eps x)\mathbb{E}[(\nabla^{-}_{\eps}M^{\eps}(x))^2], \label{eq:qV0}
\end{align}
where the last line follows by noting that $\mathbb{E}[\nabla^{-}_{\eps}M^{\eps}(\eps x)\nabla^{-}_{\eps}M^{\eps}(\eps y)]=0$ if $x\neq y$. Furthermore, $\mathbb{E}[(\nabla^{-}_{\eps}M^{\eps}(x))^2] $ is equal to $\mathbb{E}[q^{\eps^{-1/2}\hat{s}^{\eps}(\eps x)}/(q^{\eps^{-1/2}\hat{s}^{\eps}(\eps x)}+1)^2]$ which converges to $1/4$ as $\eps\to 0$. Combining this with the decays of $\psi$ shows that the r.h.s. of \eqref{eq:qV0} converges to $\frac{1}{4}\int_{R}\psi^2(x)dx$ via dominated convergence theorem. Thus, for any $\psi\in L^2(\RR)$, we have $\mathbb{E}[(\int \psi(x)dW(x))^2]=\frac{1}{4}\int_\RR \psi^2(x)dx$. Owing to this and $W(0)=0$, $W$ must be $1/2$ times a two sided Brownian motion. This shows that $\mathcal{Z}_0$ satisfies \eqref{eq:OU} as desired.
\end{proof}

%Thus, we see $\delta_{\eps}\hat{s}^{\eps}(0,X)= \frac{q^{\hat{s}^{\eps}(0,X)}-1}{q^{\hat{s}^{\eps}(0,X)}+1} + \delta_{\eps}\xi^{\eps}(x)$ where $\delta_{\eps}\xi^{\eps}(x)$ is the martingale. Note that as $\eps$ goes to $0$, the equation converges to $d\mathfrak{s}(X) = \frac{1}{2}dB(X)$ where $B$ is a two sided Brownian motion.  Note that there exi

\begin{proof}[Proof of Corollary~\ref{Cor2}]

  The first part of Lemma~\ref{Lem1} shows that the initial data $\{\hat{s}^{\eps}_0\}$ satisfies \eqref{eq:NSofS1} and \eqref{eq:NSofS2} whereas the second part of the lemma shows that the initial data converges weakly to the stationary solution of the spatial Ornstein-Uhlenbeck process \eqref{eq:OU}. Combining both parts of Lemma~\ref{Lem1} with Theorem~\ref{Cor1} completes the proof of this corollary.
\end{proof}

\bibliographystyle{alpha}
\bibliography{Reference}

\end{document}